\documentclass{ws-m3as}
\usepackage{latexsym}
\usepackage{amsmath}%
\usepackage{amsfonts}%
\usepackage{amssymb}%
\usepackage{epsfig}
\usepackage[francais,english]{babel}
\usepackage[latin1]{inputenc}
\usepackage[T1]{fontenc}
\usepackage[cyr]{aeguill} 

\numberwithin{equation}{section} 
\allowdisplaybreaks
\newcommand\beq{\begin{equation}}
\newcommand\eeq{\end{equation}}
\renewcommand{\emph}{\textbf}
\newcommand{\brk}[1]{\left(#1\right)}          
\newcommand{\Brk}[1]{\left[#1\right]}          
\newcommand{\BRK}[1]{\left\{#1\right\}}        
\newcommand{\Abs}[1]{\left| #1 \right|}        
\newcommand{\Scal}[2]{\left(#1,#2\right)}      
\newcommand{\Norm}[1]{\left\| #1 \right\|}     

\newcommand{\jump}[1]{[\![#1]\!]}
\newcommand{\xx}{\boldsymbol{x}}
\newcommand{\bolda}{\boldsymbol{a}}
\newcommand{\boldb}{\boldsymbol{b}}
\newcommand{\e}{\varepsilon}

\newcommand{\D}{\mathcal{D}}
\newcommand{\f}{\boldsymbol{f}}
\newcommand{\I}{\boldsymbol{I}}
\renewcommand{\to}{\rightarrow}

\newcommand{\str}{{\boldsymbol{\tau}}}
\newcommand{\strs}{{\boldsymbol{\sigma}}}

\newcommand{\bzero}{{\boldsymbol{0}}}
\newcommand{\bu}{{\boldsymbol{u}}}
\newcommand{\bv}{\boldsymbol{v}}
\newcommand{\bw}{\boldsymbol{w}}
\newcommand{\bz}{\boldsymbol{z}}
\newcommand{\be}{\boldsymbol{e}}

\newcommand{\gbu}{{\grad\bu}}
\newcommand{\gbv}{{\grad\bv}}

\newcommand{\bphi}{\boldsymbol{\phi}}
\newcommand{\bpsi}{\boldsymbol{\psi}}
\newcommand{\bchi}{\boldsymbol{\chi}}

\newcommand{\bxi}{\boldsymbol{\xi}}
\newcommand{\bvarsigma}{\boldsymbol{\varsigma}}
\newcommand{\buh}{\bu_h}

\newcommand{\gbuh}{\grad\bu_h}
\newcommand{\strh}{\strs_h}

\newcommand{\prd}{p_{\delta}}

\newcommand{\paorL}{p_{\alpha}^{(L)}}
\newcommand{\padorL}{p_{\alpha,\delta}^{(L)}}
\newcommand{\bud}{\bu_\delta}
\newcommand{\gbud}{\grad\bu_\delta}
\newcommand{\strd}{\strs_\delta}
\newcommand{\bua}{\bu_\alpha}

\newcommand{\stra}{\strs_\alpha}

\newcommand{\buaL}{\bu_{\alpha}^{L}}

\newcommand{\buaorL}{\bu_{\alpha}^{(L)}}
\newcommand{\buadorL}{\bu_{\alpha,\delta}^{(L)}}

\newcommand{\gbuaorL}{\grad\bu_{\alpha}^{(L)}}
\newcommand{\gbuadorL}{\grad\bu_{\alpha,\delta}^{(L)}}

\newcommand{\straL}{\strs_{\alpha}^{L}}

\newcommand{\straorL}{\strs_{\alpha}^{(L)}}
\newcommand{\stradorL}{\strs_{\alpha,\delta}^{(L)}}
\newcommand{\buhd}{\bu_{\delta,h}}
\newcommand{\gbuhd}{\grad\bu_{\delta,h}}
\newcommand{\strhd}{\strs_{\delta,h}}

\newcommand{\stradh}{\strs_{\alpha,\delta,h}}

\newcommand{\buhdLa}{\bu_{\alpha,\delta,h}}
\newcommand{\gbuhdLa}{\grad\bu_{\alpha,\delta,h}}
\newcommand{\strhdLa}{\strs_{\alpha,\delta,h}}
\newcommand{\buhLa}{\bu_{\alpha,h}}
\newcommand{\gbuhLa}{\grad\bu_{\alpha,h}}
\newcommand{\strhLa}{\strs_{\alpha,h}}

\newcommand{\dt}{\Delta t}
\newcommand{\Wi}{{\text{Wi}}}
\renewcommand{\Re}{{\text{Re}}}
\renewcommand{\div}{\operatorname{div}}

\newcommand{\tr}{\operatorname{tr}}

\newcommand{\grad}{\boldsymbol{\nabla}}

\newcommand{\R}{\mathbb{R}}
\newcommand{\RS}{\mathbb{R}^{d \times d}_S}
\newcommand{\RSPD}{\mathbb{R}^{d \times d}_{SPD}}
\newcommand{\PP}{\mathbb{P}}

\newcommand{\pd}[2]{\frac{\partial#1}{\partial#2}}
\newcommand{\deriv}[2]{\frac{d#1}{d#2}}

\newcommand{\intd}{\int_\D}
\usepackage{verbatim}
\def\Gd{G_\delta}

\def\GdL{{G_{\delta}^L}}
\def\GdorL{{G_{\delta}^{(L)}}}

\def\GorL{{G^{(L)}}}
\def\HdL{{H_{\delta}^L}}

\def\HdorL{{H_{\delta}^{(L)}}}
\def\HorL{{H^{(L)}}}

\def\BL{\beta^L}
\def\Bd{\beta_\delta}
\def\BdL{{\beta_{\delta}^L}}
\def\BdorL{{\beta_{\delta}^{(L)}}}

\def\BorL{{\beta^{(L)}}}
\def\Fd{F_{\delta}}
\def\FdorL{F_{\delta}^{(L)}}

\def\FdorLh{F_{\delta,h}^{(L)}}
\def\ForLh{F_{h}^{(L)}}
\def\bn{\boldsymbol n}
\def\U{\mathrm{W}}  
\def\Uz{\mathrm{V}}
\def\S{\mathrm{S}}
\newcommand{\SPD}{\S_{PD}}
\def\Uh{\mathrm{W}_h}  
\def\Vhzero{\mathrm{V}_{h}^0}
\def\Vhone{\mathrm{V}_{h}^1}
\def\Sh{\mathrm{S}_h}

\def\Shone{\mathrm{S}_h^1}
\def\ShonePD{\mathrm{S}_{h,PD}^1}

\begin{document}
\markboth{John W.\ Barrett and S\'ebastien Boyaval}
{Existence and Approximation of a (Regularized) Oldroyd-B model}

\title{Existence and Approximation of a (Regularized) Oldroyd-B Model}
\author{John W. Barrett}
\address{Department of Mathematics, Imperial College London,\\ London SW7 2AZ, UK\\
jwb@ic.ac.uk}

\author{S\'ebastien Boyaval}
\address{CERMICS, Ecole Nationale des Ponts et Chauss\'ees 
(ParisTech/Universit\'e Paris-Est) \\ 
6 \& 8 avenue Blaise Pascal,
Cit\'e Descartes, 77455 Marne-la-Vall\'ee Cedex 2, France\\[2mm]       
sebastien.boyaval@cermics.enpc.fr\\[2mm]
Present address: Laboratoire d'hydraulique Saint-Venant,\\ Universit\'e Paris-Est
(Ecole des Ponts ParisTech) EDF R\&D,\\ 6 quai Watier, 78401 Chatou Cedex, France
}

\maketitle

\begin{abstract}
We consider the finite element approximation  
of the Oldroyd-B system of equations,
which models a dilute polymeric fluid,
in a bounded domain $\D \subset {\mathbb R}^d$,
$d=2$ or $3$,
subject to no flow boundary conditions.
Our schemes are based on approximating 
the pressure and the   
symmetric conformation tensor 
by either
(a) piecewise constants 
or (b) continuous piecewise linears.  
In case (a) the velocity field is approximated 
by continuous piecewise quadratics or a reduced version, 
where the tangential component on each simplicial edge ($d=2$) 
or face ($d=3$) is linear.
In case (b) the velocity field is approximated 
by continuous piecewise quadratics or the mini-element. 
We show that both of these types of schemes 
satisfy a {\it free energy} bound, which involves the logarithm of the
conformation tensor,  
without any constraint on the time step 
for the backward Euler type time discretization. 
This extends the results of 
 Boyaval {\it et al.}\cite{boyaval-lelievre-mangoubi-09}
on this free energy bound.
There a piecewise constant  
approximation of the conformation tensor
was necessary to treat the advection term in the stress equation,
and a restriction on the time step, based on the initial data, was required
to ensure that the approximation to the conformation 
tensor remained positive definite.
Furthermore, 
for
our approximation (b)
in the presence of an additional dissipative term in the stress equation
and a cut-off on the conformation tensor on certain terms in the system,
similar to those introduced in Barrett and S\"{u}li\cite{barrett-suli:2008} 
for the microscopic-macroscopic FENE model of a dilute polymeric fluid,
we show (subsequence) {\it convergence}, as the spatial and temporal discretization parameters
tend to zero, 
towards global-in-time weak solutions of this
{\it regularized} Oldroyd-B system.
Hence, we prove existence of global-in-time weak solutions 
to this regularized model.  
Moreover, in the case $d=2$ we carry out this convergence in the absence of cut-offs,
but with a time step restriction dependent on the spatial discretization parameter, 
and hence show existence of a global-in-time weak solution to the Oldroyd-B system  
with an additional dissipative term in the stress equation.  
\end{abstract}



\section{Introduction}
\label{sec1}
\subsection{The standard Oldroyd-B model}
We consider the Oldroyd-B model for a dilute polymeric fluid.  
The fluid,
confined to an open bounded domain $\D\subset\R^d$ ($d=2$ or $3$)
with a Lipschitz boundary $\partial \D$, 
is governed by the following non-dimensionalized system: 

\noindent
{\bf (P)} Find 
$\bu : (t,\xx)\in[0,T)\times\D \mapsto \bu(t,\xx)\in\R^d$,  
$p : (t,\xx)\in \D_T :=(0,T)\times\D \mapsto p(t,\xx)\in\R$ 
and $\strs : (t,\xx)\in[0,T)\times\D \mapsto\strs(t,\xx)\in \RS$ 
such that
\begin{subequations}
\begin{alignat}{2}
\label{eq:oldroyd-b-sigma}
\Re \brk{\pd{\bu}{t} + (\bu\cdot\grad)\bu} & =  -\grad p 
+ (1-\e)\Delta\bu +  \frac{\e}{\Wi}\div\strs + \f 
\qquad &&\mbox{on } \D_T 
\,, 
\\
\div\bu & = 0 
\qquad &&\mbox{on } \D_T 
\,, 
\label{eq:oldroyd-b-sigma1}
\\
\pd{\strs}{t}+(\bu\cdot\grad)\strs & = 
(\gbu)\strs+\strs(\gbu)^T-\frac{1}{\Wi}\brk{\strs-\I} 
\qquad &&\mbox{on } \D_T 
\,,
\label{eq:oldroyd-b-sigma2}
\\
\bu(0,\xx) &= \bu^0(\xx) 
\label{eq:initial}
&&\forall \xx \in \D\,,\\
\strs(0,\xx) &= \strs^0(\xx) 
&&\forall \xx \in \D\,,
\label{eq:initial1}
\\
\label{eq:dirichlet}
\bu&=\bzero &&\text{on $ (0,T) \times \partial\D$}
\,.
\end{alignat}
\end{subequations}
Here $\bu$ is the velocity of the fluid, 
$p$ is the hydrostatic pressure,
and $\strs$ 
 is the symmetric conformation tensor 
 of the polymer molecules linked to the symmetric 
 polymeric extra-stress tensor $\str$ 
 through the relation
  $\strs = \I + \frac{\Wi}{\e}\str,$
where $\I$ is the $d$-dimensional identity tensor
and $\RS$ denotes symmetric real $d \times d$ 
matrices.
In addition,
$\f : (t,\xx)\in 
\D_T \mapsto \f(t,\xx)\in\R^d$ is the
given density of body forces acting on the fluid;  
and the following given parameters are dimensionless:
the Reynolds number $\Re \in \R_{>0}$, the Weissenberg number $\Wi \in \R_{>0}$,  
and the elastic-to-viscous viscosity fraction $\e \in (0,1)$.
For the sake of simplicity,
we will limit ourselves to the no flow boundary condition
(\ref{eq:dirichlet}).
Finally, $\grad \bu (t,\xx) \in \R^{d \times d}$ with $[\grad \bu]_{ij}=
\frac{\partial \bu_i}{\partial \xx_j}$,
and $({\rm div}\, \strs)(t,\xx) \in \R^d$
 with $[{\rm div}\, \strs]_i = \sum_{j=1}^d 
\frac{\partial \strs_{ij}}{\partial \xx_j}$.

Unfortunately, at present there is no proof of existence of global-in-time weak solutions
to (P) available in the literature.  
Local-in-time existence results for (P) for sufficiently smooth initial data,
and global-in-time existence results for sufficiently small initial data  
can be found in  Guillop\'e and Saut\cite{guillope-saut:1990} for a Hilbert space framework,
and in Fern\'{a}ndez-Cara {\it et al.}\cite{CGO} 
for a more general Banach space framework. 
Global-in-time existence results for the corotational version of (P);
that is, where $\gbu$ in
(\ref{eq:oldroyd-b-sigma2}) is replaced by its anti-symmetric part
$\frac{1}{2}(\gbu - (\gbu)^T)$
can be found in Lions and Masmoudi.\cite{LM} 
We note that such a simple change to the model leads to a vast simplification 
mathematically, but, of course, it is not justified on physical grounds. 
Finally, global-in-time existence results for (P) in the case $\f \equiv \bzero$
and for initial data close to equilibrium can be found in 
Lei {\it et al.}\cite{lei-liu-zhou}.

This paper considers 
some finite element 
approximations of the Oldroyd-B system, possibly with some regularization. 
In the regularized case, we show (subsequence) convergence of the approximation,
as the spatial and temporal discretization parameters tend to zero,
and so establish  
the existence of global-in-time weak solutions  
of these regularized versions of the Oldroyd-B system. 
The first of these regularized problems is (P$_\alpha$) 
obtained by adding the dissipative 
term $\alpha\,\Delta \strs$ for a given $\alpha \in \R_{>0}$ 
to the right-hand side of (\ref{eq:oldroyd-b-sigma2}), 
as considered computationally in Sureshkumar and Beris,\cite{sureshkumar-beris}
with an additional no flux boundary condition for $\strs$
on $\partial \D$. 
The second is (P$_\alpha^L$) where, in addition to the 
regularization in (P$_\alpha$),
the conformation tensor $\strs$ is replaced by the cut-off $\BL(\strs)$
on the right-hand side of (\ref{eq:oldroyd-b-sigma})
and in the terms involving $\bu$  
in (\ref{eq:oldroyd-b-sigma2}), where 
$\BL(s):= \min \{s,L\}$ for a given $L \gg 1$.
Similar regularizations have been 
introduced for the microscopic-macroscopic dumbbell model of dilute polymers
with a finitely extensible nonlinear elastic (FENE) spring law,
see Barrett and S\"{u}li,\cite{barrett-suli:2008} and
for the convergence of the finite element approximation 
of such models, see 
Barrett and S\"{u}li.\cite{barrett-suli:2009}
In fact, it is argued in Barrett and S\"{u}li\cite{barrett-suli:2007} and Schieber\cite{SCHI}
that 
the dissipative 
term $\alpha\,\Delta \strs$ is not a regularization,
but is present in the original model with a positive $\alpha \ll 1$.    
Here we recall that the Oldroyd-B system is   
the macroscopic closure of 
the microscopic-macroscopic dumbbell model 
with a Hookean spring law, see e.g.\ Barrett and S\"{u}li.\cite{barrett-suli:2007}
 
Overall the aims of this paper are threefold.
First, we extend previous results in Boyaval {\it et al.}\cite{boyaval-lelievre-mangoubi-09}
for a finite element approximation of (P) 
using essentially the backward Euler scheme in time and 
based on approximating 
the pressure and the   
symmetric conformation tensor 
by piecewise constants;  
and
the velocity field
with continuous piecewise quadratics 
or a reduced version, 
where the tangential component on each simplicial edge ($d=2$) 
or face ($d=3$) is linear.
We show that solutions of this numerical scheme
satisfy a discrete free energy bound, which involves the logarithm of
the conformation tensor,
{\it without} any constraint on the time step,
whereas a time constraint based on the initial data 
was required in Boyaval {\it et al.}\cite{boyaval-lelievre-mangoubi-09}
in order to ensure that the approximation to the conformation tensor $\strs$ remained 
positive definite.
See also Lee and Xu,\cite{lee-xu} where the difficulties of maintaining
the positive definiteness of approximations to $\strs$ are also discussed.
We achieve our result by first  
introducing problem (P$_\delta$), based on 
a regularization parameter $\delta \in \R_{>0}$. 
(P$_\delta$) satisfies a regularized free energy estimate
based on a regularization of $\ln$ and is valid
without the positive definiteness constraint on the deformation tensor.

Second, we show that it is possible to approximate (P)
with a {\it continuous} (piecewise linear) approximation of the conformation tensor,
such that a discrete free energy bound still holds.
We note  that a piecewise constant  
approximation of the conformation tensor
was necessary 
in Boyaval {\it et al.}\cite{boyaval-lelievre-mangoubi-09}
in order to treat the advection term
in  (\ref{eq:oldroyd-b-sigma2}) and still obtain
a discrete free energy bound.   

Third, we show (subsequence) convergence, as the spatial and temporal
discretization parameters tend to zero, of this latter approximation 
in the presence of the regularization terms stated above
to global-in-time weak solutions of the corresponding regularized form of (P).

The outline of this paper is as follows.
We end Section \ref{sec1} by introducing our notation and auxiliary results.
In Section \ref{sec:delta} we introduce our regularizations of $\ln$ based on the 
parameter $\delta \in (0,\frac{1}{2}]$ and the cut-off $L \ge 2$.
We introduce our regularized problem (P$_\delta$),
and show a formal free energy estimate for it.
In Section \ref{sec:deltah}, on assuming that $\D$ is a polytope, 
we introduce our finite element approximation of 
(P$_\delta$), (P$^{\Delta t}_{\delta,h}$)
based on approximating 
the pressure and the   
symmetric conformation tensor 
by piecewise constants;  
and
the velocity field
with continuous piecewise quadratics 
or a reduced version, 
where the tangential component on each simplicial edge ($d=2$) 
or face ($d=3$) is linear.
Using the Brouwer fixed point theorem, we prove existence of a 
solution to (P$^{\Delta t}_{\delta,h}$)
and show that it satisfies 
a discrete regularized free energy estimate for any choice of time step; see Theorem~\ref{dstabthm}.
We conclude by showing that, in the limit $\delta \to 0_+$, 
these solutions of 
(P$^{\Delta t}_{\delta,h})$
converge to a solution of (P$^{\Delta t}_{h}$) 
with the approximation of the conformation tensor being positive definite.
Moreover, this solution of (P$^{\Delta t}_{h}$) 
satisfies a discrete free energy estimate;  see Theorem~\ref{dconthm}.

In Section \ref{sec:Palpha} we introduce our regularizations
(P$_\alpha^{(L)}$) of (P) involving  
the dissipative term $\alpha\,\Delta \strs$  
on the right-hand side of (\ref{eq:oldroyd-b-sigma2}), 
and possibly the cut-off $\BL(\strs)$
on certain terms involving $\strs$ 
in (\ref{eq:oldroyd-b-sigma},c).
We then introduce the corresponding regularized version 
(P$_{\alpha,\delta}^{(L)}$),
and show a 
formal free energy estimate for it.
In Section \ref{sec:Palphah}  
we introduce our finite element approximation of (P$_{\alpha,\delta}^{(L)}$), 
(P$^{(L,)\Delta t}_{\alpha,\delta,h}$)
based on approximating 
the pressure and the   
symmetric conformation tensor 
by continuous piecewise linears;  
and
the velocity field
with continuous piecewise quadratics 
or the mini-element.  
Here we assume that 
the finite element mesh 
consists of non-obtuse simplices.
Using the Brouwer fixed point theorem, we prove existence of a 
solution to (P$^{(L,)\Delta t}_{\alpha,\delta,h}$)
and show that is satisfies 
a discrete regularized free energy estimate 
for any choice of time step; see Theorem~\ref{dstabthmaorL}.
We conclude by showing that, in the limit $\delta \to 0_+$, 
these solutions of 
(P$^{(L,)\Delta t}_{\alpha,\delta,h})$
converge to a solution of (P$^{(L,)\Delta t}_{\alpha,h}$)
with the approximation of the conformation tensor being positive definite.
Moreover, this solution of (P$^{(L,)\Delta t}_{\alpha,h}$) 
satisfies a discrete free energy estimate;  see Theorem~\ref{dLconthm}.

In Section~\ref{sec:convergence} we assume, in addition, that 
$\D$ is a convex polytope
and that the finite element mesh 
consists of quasi-uniform simplices.
We then prove (subsequence) convergence of the solutions 
of (P$_{\alpha,h}^{L,\Delta t}$), as the spatial and temporal
discretization parameters tend to zero,  
to global-in-time weak solutions of (P$_\alpha^L$); see Theorem~\ref{convaL}.
Finally
in Section~\ref{sec7}, on further assuming that $d=2$
and a time step restriction dependent on the spatial discretization parameter,
we prove (subsequence) convergence of the solutions 
of (P$_{\alpha,h}^{\Delta t}$), as the spatial and temporal
discretization parameters tend to zero,  
to global-in-time weak solutions of (P$_\alpha$); see Theorem~\ref{conva}.
We note that these existence results for (P$_\alpha^{(L)}$) are new to the literature.
In addition, the corresponding $L^\infty(0,T;L^2(\Omega))\cap L^2(0,T;H^1(\Omega))$
norms of the velocity solution $\bu_\alpha^{(L)}$ of (P$_\alpha^{(L)}$) are independent
of the regularization parameters $\alpha$ (and $L$).

In a forthcoming paper,\cite{barrett-boyaval} we will extend the ideas in this paper to
a related macroscopic model, the FENE-P model;
see Hu and Leli\`{e}vre,\cite{hu-lelievre} 
where a free energy estimate is developed for such a model,
as well as Oldroyd-B. In addition, we will report 
in the near future
on numerical computations based
on the finite element approximations in this paper and Barrett and Boyaval.\cite{barrett-boyaval}

\subsection{Notation and auxiliary results}

The absolute value and the negative part of a real number $s\in\R$ 
are denoted by $\Abs{s}:=\max\{s,-s\}$ and $[s]_{-}=\min\{s,0\}$, respectively.
In addition to $\RS$,
the set of symmetric $\R^{d \times d}$ matrices, we let   $\RSPD$
be  the set of symmetric positive definite  $\R^{d \times d}$ matrices.
We adopt the following notation for inner products: 
\begin{subequations}
\begin{alignat}{2}
 \bv\cdot\bw &:=\sum_{i=1}^d \bv_i \bw_i
 \equiv \bv^T \bw = \bw^T \bv  \qquad &&\forall \bv,\bw \in {\mathbb R}^d,
 \label{ipvec}
 \\
 \bphi : \bpsi &:=\sum_{i=1}^d \sum_{j=1}^d \bphi_{ij} \bpsi_{ij}
 \equiv \tr\brk{\bphi^T \bpsi}=\tr\brk{\bpsi^T \bphi} \qquad 
 &&\forall \bphi,\bpsi \in \R^{d\times d},
 \label{ipmat}
 \\
 \grad \bphi :: \grad \bpsi &:=\sum_{i=1}^d \sum_{j=1}^d \grad \bphi_{ij} 
 \cdot \grad \bpsi_{ij}\qquad 
 &&\forall \bphi,\bpsi \in \R^{d\times d};
 \label{ipgrad}
\end{alignat}
\end{subequations}
where $\cdot^T$ and $\tr\brk{\cdot}$ denote transposition and trace, respectively.
The corresponding norms are
\begin{subequations}
\begin{alignat}{3}
 \Norm{\bv}&:= (\bv \cdot \bv)^{\frac{1}{2}},
 \qquad &&\Norm{\grad \bv}:= (\grad \bv : \grad \bv)^{\frac{1}{2}}
 \qquad &&\forall \bv \in \R^d;
 \label{normv}
 \\
 \Norm{\bphi} &:= (\bphi : \bphi)^{\frac{1}{2}}, 
 \qquad &&\Norm{\grad \bphi} := (\grad \bphi :: \grad \bphi)^{\frac{1}{2}}, 
 \qquad &&\forall \bphi \in \R^{d\times d}.
 \label{normphi}
\end{alignat}
\end{subequations}
We will use on several occasions that $\tr(\bphi)=\tr(\bphi^T)$ and
$\tr(\bphi\bpsi)=\tr(\bpsi\bphi)$ for all $\bphi, \bpsi \in \R^{d \times d}$.
In particular, we note that:
\begin{subequations}
\begin{alignat}{2} \label{eq:symmetric-tr}
 \bphi\bchi^T:\bpsi 
 &= \bchi\bphi:\bpsi = \bchi : \bpsi \bphi 
 \qquad &&\forall \bphi,\bpsi\in \RS, 
 \ \bchi\in \R^{d\times d}
 \,,\\
 \Norm{\bpsi\bphi} &\leq \Norm{\bpsi} \Norm{\bphi}
\qquad &&\forall \bphi,\bpsi\in\R^{d\times d}\,. 
\label{normprod}
\end{alignat}
\end{subequations}

In addition, for any $\bphi \in \RS$,
there exists a diagonal decomposition
\beq \label{eq:diagonal-decomposition} 
\bphi= \mathbf{O}^T \mathbf{D} \mathbf{O} 
\qquad \Rightarrow \qquad \tr\brk{\bphi} = \tr\brk{\mathbf{D}} 
\,,
\eeq
where $\mathbf{O}\in \R^{d \times d}$ is 
an orthogonal matrix and $\mathbf{D} \in \R^{d \times d}$ 
a diagonal matrix. 
Hence, for any $g:\R\to\R$,  
one can define 
$g(\bphi) \in \RS$ as
\beq \label{eq:tensor-g}
g(\bphi) := \mathbf{O}^T g({\bf D})\mathbf{O} 
\qquad \Rightarrow \qquad \tr\brk{g(\bphi)} = \tr\brk{g(\mathbf{D})}
\,,
\eeq
where $g({\bf D}) \in \RS$ is the diagonal matrix with entries $[g({\bf D})]_{ii} 
= g({\bf D}_{ii})$, $i=1 \rightarrow d$.
Although the diagonal decomposition~\eqref{eq:diagonal-decomposition} is not unique,
\eqref{eq:tensor-g} uniquely defines $g(\bphi)$.
Similarly, one can define $g(\bphi) \in \RSPD$, 
when $\bphi \in \RSPD$
and $g:\R_{>0}\to\R$.
We note for later purposes that the choice $g(s)=|s|$ for $s \in {\mathbb R}$ yields
that
\begin{align}
d^{-1} (\tr(|\bphi|))^2 \leq \Norm{\bphi}^2 
\leq (\tr(|\bphi|))^2 
\qquad \forall \bphi \in \RS\,.
\label{modphisq}
\end{align}


We adopt the standard notation for Sobolev spaces, e.g.\ 
$H^1(\D) := \{ \eta:\D \rightarrow 
\R \,:\, \intd [\,|\eta|^2 + \|\nabla \eta\|^2\,] $ $ < \infty\}$
with $H^1_0(\D)$ being the closure of $C^\infty_0(\D)$ for 
the corresponding norm $\|\cdot\|_{H^1(\D)}$.
We denote the associated semi-norm as $|\cdot|_{H^1(\D)}$.
The topological dual of the Hilbert space $H^1_0(\D)$, with pivot space  $L^2(\D)$,
will be denoted by $H^{-1}(\D)$.
We denote the duality pairing between $H^{-1}(\D)$ and $H^1_0(\D)$
as $\langle \cdot,\cdot\rangle_{H^1_0(\D)}$.
Such function spaces are naturally extended when the range $\R$ is replaced
by $\R^d$, $\R^{d \times d}$ and $\RS$; 
e.g.\ $H^1(\D)$ becomes
$[H^1(\D)]^d$,  $[H^1(\D)]^{d \times d}$ and  
$[H^1(\D)]^{d \times d}_{S}$ , 
respectively. 
For ease of notation, we write the corresponding norms 
and semi-norms as
$\| \cdot\|_{H^1(\D)}$ and $| \cdot|_{H^1(\D)}$, respectively,
as opposed to e.g.\ $\|\cdot\|_{[H^1(\D)]^d}$
and $|\cdot|_{[H^1(\D)]^d}$, respectively. 
Similarly, we write $\langle \cdot,\cdot\rangle_{H^1_0(\D)}$
for the duality pairing between e.g.\  
$[H^{-1}(\D)]^d$ and $[H^1_0(\D)]^d$.
We recall the Poincar\'e inequality
\beq \label{eq:poincare}
 \int_\D \|\bv\|^2 \leq C_P \int_\D \|\gbv\|^2 \qquad \forall \bv \in [H^1_0(D)]^d\,, 
\eeq
where $C_P \in \R_{>0}$ depends only on $\D$.
For notational convenience, we introduce also convex sets such as
$[H^1(\D)]^{d \times d}_{SPD} := \{ \bphi \in [H^1(\D)]^{d \times d}_S :
\bphi \in \RSPD \mbox{ a.e. in } \D \}$. 
Moreover, 
in order to analyse (P),
we adopt the notation 
\begin{align}
\nonumber
\U  &:= [H^1_0(\D)]^d, \qquad  {\rm Q}:= L^2(\D), 
\qquad  \Uz := \BRK{\bv\in\U \,:\,\intd q\,\div\bv =0 \quad \forall q\in{\rm Q}},\\
\ \S &:= [L^1(\D)]^{d \times d}_S
\qquad \mbox{and} \qquad 
\SPD :=  [L^1(\D)]^{d \times d}_{SPD}\,.
\label{spaces}
\end{align}
Finally, throughout the paper $C$ will denote a generic positive constant independent
of the regularization parameters $\delta,\,L$ and $\alpha$;
and the mesh parameters $h$ and $\Delta t$.

\section{Formal free energy estimates for a regularized problem (P$_\delta$)}
\label{sec:delta}

\subsection{Some regularizations}

Let $G:s \in \R_{>0} \mapsto \ln s\in\R$ be the logarithm function,
whose domain of definition can be 
straightforwardly extended to the set of symmetric positive definite 
matrices 
using (\ref{eq:diagonal-decomposition}).
We define the following two concave 
$C^1(\R)$ regularizations of $G$ based on given parameters  
$L>1>\delta>0$:
\begin{align} 
\nonumber
&\Gd : s \in \R \mapsto 
\begin{cases}
G(s) \,, & \forall s \ge \delta 
\\
\frac{s}{\delta}+G(\delta)-1 \,, & \forall s \leq \delta
\end{cases}
\quad \mbox{and} \quad
\GdL : s \in \R \mapsto \
\begin{cases}
G^L(s) \,, & \forall s \geq \delta 
\\
\Gd(s) \,, & \forall s \leq \delta
\end{cases}
\,,\\
&\mbox{where}
\quad 
G^L : s \in \R_{>0} \mapsto \
\begin{cases}
\frac{s}{L}+G(L)-1 \,, & \forall s \geq L 
\\
G(s) \,, & \forall s \in (0,L] 
\end{cases}
\,. 
\label{eq:Gd}
\end{align}
We define also the following scalar functions 
\begin{align}
\BdorL(s):=\brk{\GdorL'(s)}^{-1} 
\quad \forall s\in \R 
\quad \mbox{and}
\quad
\BorL(s):=\brk{\GorL'(s)}^{-1} 
\quad
\forall s\in \R_{>0}\,;
\label{eq:BGd}
\end{align}
where, here and throughout this paper, $\cdot^{(\star)}$ denotes
an expression with or without the superscript $\star$, and 
a similar convention with subscripts.
Hence we have that
\begin{align} \nonumber
&\Bd : s \in \R \mapsto 
\max\{s,\delta\}\,,
\qquad 
\BdL : s \in \R \mapsto 
\min\{\beta_\delta(s),L\}
\,,\\
&\beta : s \in \R_{>0} \mapsto s 
\qquad \mbox{and} \qquad
\BL : s \in \R_{>0} \mapsto \min\{\beta(s),L\}
\,.
\label{eq:Bd}
\end{align}
We note for example that
\begin{align}
& \Norm{\BdL(\bphi)}^2  
\leq dL^2 \quad \forall \bphi \in \RS
\quad \mbox{and}\quad
\Norm{\BL(\bphi)}^2 
\leq dL^2 
\quad \forall \bphi \in \RSPD\,.
\label{BdLphi}
\end{align}
\begin{figure}
\begingroup
\setlength{\unitlength}{0.0200bp}%
\begin{picture}(18000,8640)(500,0)%
\includegraphics{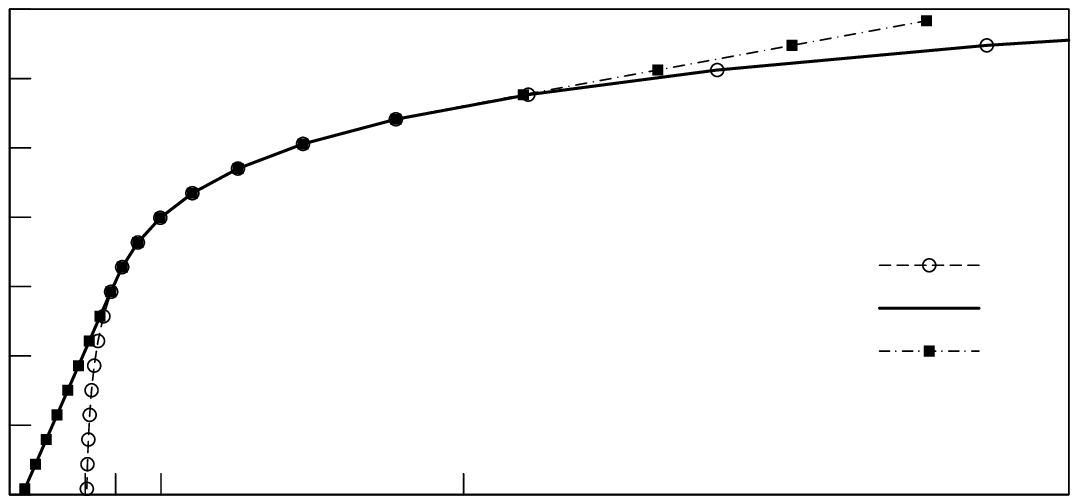}%
\end{picture}%
\begin{picture}(18000,8640)(18500,0)%
\put(1650,1100){\makebox(0,0)[r]{\strut{}$-4$}}%
\put(1650,2099){\makebox(0,0)[r]{\strut{}$-3$}}%
\put(1650,3097){\makebox(0,0)[r]{\strut{}$-2$}}%
\put(1650,4096){\makebox(0,0)[r]{\strut{}$-1$}}%
\put(1650,5094){\makebox(0,0)[r]{\strut{}$0$}}%
\put(1650,6093){\makebox(0,0)[r]{\strut{}$1$}}%
\put(1650,7091){\makebox(0,0)[r]{\strut{}$2$}}%
\put(1650,8090){\makebox(0,0)[r]{\strut{}$3$}}%
\put(8461,550){\makebox(0,0){\strut{}$L$}}%
\put(4104,550){\makebox(0,0){\strut{}$1$}}%
\put(3450,550){\makebox(0,0){\strut{}$\delta$}}%
\put(3014,550){\makebox(0,0){\strut{}$0$}}%
\put(14177,4402){\makebox(0,0){$G \hspace{7mm}$}}%
\put(14177,3840){\makebox(0,0)[r]{\strut{}$G_\delta\,$}}%
\put(14177,3278){\makebox(0,0)[r]{\strut{}$G_\delta^L$}}%
\end{picture}%
\endgroup
\caption{The function $G$ and its regularizations.}
\end{figure}
Introducing the concave $C^1(\R)$ functions
\begin{align} 
\HdL(s):=G_{L^{-1}}^{\delta^{-1}}(s) \qquad \forall s \in \R
\qquad\mbox{and}\qquad
H_\delta(s) :=G^{\delta^{-1}}(s) \qquad \forall s \in \R_{>0}\,,
\label{eq:HdL}
\end{align} 
it follows from (\ref{eq:Gd}) and (\ref{eq:Bd}) that
\begin{align}
\HdorL'(\GdorL'(s)) = \BdorL(s) \qquad \forall s \in \R
\,.
\label{eq:HBdL}
\end{align}
For later purposes, we prove the following results concerning these functions.
\begin{lemma}\label{GLemma}
The following hold 
for any $\bphi, \bpsi \in \RS$ and  
for any $L >1 >\delta >0$ that
\begin{subequations}
\begin{align} 
\label{eq:inverse-Gd}
[\BdorL(\bphi)][\GdorL'(\bphi)]&=[\GdorL'(\bphi)][\BdorL(\bphi)]=\I\,,  
\\
\label{eq:positive-term}
\tr\brk{\BdorL(\bphi)+[\BdorL(\bphi)]^{-1}-2\I} &\ge 0 \,,
\\
\label{Entropy1}
\tr\brk{\bphi-\GdorL(\bphi)-\I} & \ge 0\,,\\
\label{expositive}
\brk{\bphi - \BdorL(\bphi)}:\brk{\I-\GdorL'(\bphi)} &\geq 0\,,\\
\label{eq:concavity}
\brk{\bphi-\bpsi}:\brk{\GdorL'(\bpsi)}&\ge\tr\brk{\GdorL(\bphi)-\GdorL(\bpsi)} \,,
\\
\label{matrix:Lipschitz}
- \brk{\bphi-\bpsi}:\brk{\GdorL'(\bphi)-\GdorL'(\bpsi)}
& \ge
\delta^2 \Norm{\GdorL'(\bphi)-\GdorL'(\bpsi)}^2\,.
\end{align}
\end{subequations}
In addition, if $\delta \in (0,\frac{1}{2}]$ and $L \ge 2$ we have that 
\begin{align}
\label{Entropy2}
\tr\brk{\bphi-\GdorL(\bphi)}&\ge \left\{\begin{array}{ll}
\frac{1}{2} \|\bphi\|  
\\[2mm]
\frac{1}{2\delta} \|[\bphi]_{-}\| 
\end{array}\right. 
\quad \mbox{\rm and}
\quad  
\bphi:\brk{\I-\GdorL'(\bphi)}\ge 
\textstyle \frac{1}{2}  
\|\bphi\|
-d
\,.
\end{align}
\end{lemma}
\begin{proof}
The result
(\ref{eq:inverse-Gd})
follows immediately from (\ref{eq:tensor-g}) 
and as $\BdorL(s)=\GdorL'(s)$ for all $s \in \R$. 
The desired results  (\ref{eq:positive-term}--d) follow similarly, on noting 
the scalar inequalities $\BdorL(s)+[\BdorL(s)]^{-1}\ge2$,  
$s-\GdorL(s)\ge 1$ 
and $(s-\BdorL(s))(1-\GdorL'(s)) \geq 0$
for all $s \in \R$.

We note that $\GdorL$ are concave functions like $G$,
and hence they satisfy the following inequality 
\beq \label{eq:scalar-concavity}
(s_1-s_2)\GdorL'(s_2)\ge\GdorL(s_1)-\GdorL(s_2) \qquad \forall s_1,s_2 \in \R\,;
\eeq
Hence for any $\bphi,\bpsi \in \RS$ with $\bphi 
= {\bf O}_\phi^T {\bf D}_\phi {\bf O}_\phi$ and $\bpsi= 
{\bf O}_\psi^T {\bf D}_\psi {\bf O}_\psi$,
where ${\bf O}_{\phi}, {\bf O}_{\psi} \in \R^{d \times d}$ orthogonal
and  ${\bf D}_{\phi}, {\bf D}_{\psi} \in \R^{d \times d}$ diagonal, we have, on noting 
the properties of trace, that
\begin{align}
(\bphi-\bpsi)
: \GdorL'(\bpsi)=
\tr\brk{(\bphi-\bpsi)\GdorL'(\bpsi)} =
\tr\brk{({\bf O}^T {\bf D}_\phi {\bf O} -{\bf D}_\psi)\GdorL'({\bf D}_\psi)}\,,
\label{eq:matrix-concavity}
\end{align}
where ${\bf O}={\bf O}_\phi {\bf O}_\psi^T \in  \R^{d \times d}$ is orthogonal
and hence $\sum_{i=1}^d [{\bf O}_{ij}]^2=  \sum_{i=1}^d [{\bf O}_{ji}]^2=1$
for $j=1 \rightarrow d$.
Therefore we have, on noting these properties of ${\bf O}$,
(\ref{eq:scalar-concavity}) and (\ref{eq:tensor-g}), that
\begin{align}
\nonumber
\tr\brk{({\bf O}^T {\bf D}_\phi {\bf O} -{\bf D}_\psi)\GdorL'({\bf D}_\psi)}
& = \sum_{i=1}^d    
\brk{\sum_{j=1}^d [{\bf O}_{ji}]^2   [{\bf D}_\phi]_{jj} 
-[{\bf D}_\psi]_{ii}}[\GdorL'({\bf D}_\psi)]_{ii}
\\
& = \sum_{i=1}^d  \sum_{j=1}^d
[{\bf O}_{ji}]^2   
\brk{[{\bf D}_\phi]_{jj} 
-[{\bf D}_\psi]_{ii}}[\GdorL'({\bf D}_\psi)]_{ii}
\nonumber
\\
& \geq \sum_{i=1}^d  \sum_{j=1}^d
[{\bf O}_{ji}]^2   
\brk{[\GdorL({\bf D}_\phi)]_{jj}
-[\GdorL({\bf D}_\psi)]_{ii}}
\nonumber
\\
&= \tr\brk{\GdorL({\bf D}_\phi)}-\tr\brk{\GdorL({\bf D}_\psi)}
\nonumber
\\
&=\tr\brk{\GdorL({\bphi})-\GdorL(\bpsi)}\,.
\label{Entproof1}
\end{align}
Combining (\ref{eq:matrix-concavity}) and (\ref{Entproof1})
yields the desired result (\ref{eq:concavity}).

We note that $-\GdorL' \in C^{0,1}(\R)$ is monotonically increasing with Lipschitz
constant $\delta^{-2}$ and so
\begin{align}
-(s_1-s_2) (\GdorL'(s_1)-\GdorL'(s_2))
& \ge \delta^2
[\GdorL'(s_1)-\GdorL'(s_2)]^2
\qquad \forall s_1,\,s_2 \in \R.
\label{scalar:Lipschitz}
\end{align} 
Then, similarly to (\ref{eq:matrix-concavity}) and (\ref{Entproof1}),
we have, on noting (\ref{scalar:Lipschitz}), that
\begin{align}
&- (\bphi-\bpsi) : (\GdorL'(\bphi)-\GdorL'(\bpsi)) 
\nonumber
\\
& \hspace{0.5in}
=-\left[ 
\tr \brk{({\bf D}_\phi -{\bf O}{\bf D}_\psi {\bf O}^T)\GdorL'({\bf D}_\phi)}
-\tr\brk{({\bf O}^T {\bf D}_\phi {\bf O} -{\bf D}_\psi)\GdorL'({\bf D}_\psi)}
\right]
\nonumber \\
& \hspace{0.5in}= -\sum_{i=1}^d  \sum_{j=1}^d
[{\bf O}_{ji}]^2   
\brk{[{\bf D}_\phi]_{jj} 
-[{\bf D}_\psi]_{ii}}([\GdorL'({\bf D}_\phi)]_{jj}-[\GdorL'({\bf D}_\psi)]_{ii})
\nonumber \\
& \hspace{0.5in}\geq  \delta^2 \sum_{i=1}^d  \sum_{j=1}^d
[{\bf O}_{ji}]^2 ([\GdorL'({\bf D}_\phi)]_{jj}-[\GdorL'({\bf D}_\psi)]_{ii})^2
\nonumber \\
& \hspace{0.5in}= \delta^2
\tr\brk{(\GdorL'(\bphi)-\GdorL'(\bpsi))^2} 
= \delta^2 
\,\|\GdorL'(\bphi)-\GdorL'(\bpsi)\|^2
\label{Entproof2}
\end{align}
and hence the desired result (\ref{matrix:Lipschitz}).

Finally the results (\ref{Entropy2}) follow 
from (\ref{eq:tensor-g}) and (\ref{modphisq}) on noting  
the following scalar inequalities 
\beq \label{eq:scalar-ineq2}
s-\GdorL(s)\ge \left\{ \begin{array}{ll} \frac{1}{2}\,\Abs{s} \\[2mm]
\frac{1}{2\delta}\,\Abs{[s]_{-}}
\end{array}
\right.
\qquad \mbox{and}
\qquad 
s(1-\GdorL'(s)) \geq \textstyle \frac{1}{2}|s|-1 
\qquad\forall s \in \R\,,
\eeq 
which are easily deduced if $\delta \in (0,\frac{1}{2}]$ and $L \ge 2$.
\end{proof}

Clearly (\ref{eq:concavity}) holds for any concave function $g \in C^1(\R)$,
not just $\GdorL$, and this implies that
\begin{align}
(\bphi-\bpsi) : g'(\bpsi) \geq
\tr\brk{g(\bphi)-g(\bpsi)} \geq (\bphi-\bpsi) : g'(\bphi)
\qquad \forall \bphi,\bpsi \in \RS.
\label{gconcave2}
\end{align}
For a convex function $g \in C^1(\R)$, the inequalities in  
(\ref{gconcave2})
are reversed.  
Hence for any concave or convex function 
$g \in C^1(\R)$ and for any $\bphi \in C^1([0,T);\RS)$
one can deduce from the above that
\begin{equation}\label{eq:deriv-tensor-g}
\deriv{}{t}\tr\brk{g(\bphi)}
= \tr\brk{\deriv{\bphi}{t} g'(\bphi)} = 
\brk{\deriv{\bphi}{t}}:g'(\bphi) 
\qquad \forall t \in [0,T).
\end{equation}
Of course, a similar result holds true for spatial derivatives. 
Furthermore, these results hold true 
if $\bphi$ is in addition positive definite,
and $g \in C^1(\R_{>0})$ is a concave or convex function.
Finally, we note that one can use the approach in (\ref{Entproof1}) to show that
if $g \in C^{0,1}(\mathbb R)$ with Lipschitz constant $g_{\rm Lip}$, then
\begin{align}
\Norm{g(\bphi)-g(\bpsi)} \leq g_{\rm Lip} \Norm{\bphi-\bpsi}
\qquad \forall \bphi, \bpsi \in \RS\,.
\label{Lip}
\end{align}

\subsection{Regularized problem (P$_\delta$)}

Using the regularizations $\Gd$ introduced above with parameter $\delta$
we consider the following regularization of (P) for a given $\delta \in (0,\frac{1}{2}]$:

\noindent
{\bf (P$_{\delta}$)} Find 
$\bud : (t,\xx)\in[0,T)\times\D\mapsto\bud(t,\xx)\in\R^d$,  
$\prd : (t,\xx)\in(0,T)\times\D\mapsto \prd(t,\xx)\in\R$ 
and $\strd : (t,\xx)\in[0,T)\times\D\mapsto\strd(t,\xx)\in \RS$ 
such that
\begin{subequations}
\begin{align}
\Re \brk{\pd{\bud}{t} + (\bud\cdot\grad)\bud} & =  -\grad \prd 
+ (1-\e)\Delta\bud 
+  \frac{\e}{\Wi}\div\Bd(\strd) + \f 
\nonumber \\
\label{eq:oldroyd-b-sigmadL}
&\hspace{1.8in}
\quad\mbox{on } \D_T 
\,, 
\\
\div\bud & = 0 
\hspace{1.7in} \mbox{on } \D_T 
\,, 
\label{eq:oldroyd-b-sigma1dL}
\\
\pd{\strd}{t}+(\bud\cdot\grad)\strd & = 
(\gbud)\Bd(\strd)+\Bd(\strd)(\gbud)^T
-\frac{1}{\Wi}\brk{\strd-\I} 
\nonumber \\
& \hspace{1.8in}
\quad \mbox{on } \D_T 
\,,
\label{eq:oldroyd-b-sigma2dL}
\\
\bud(0,\xx) &= \bu^0(\xx) 
\label{eq:initialdL}
\hspace{1.4in}
\forall \xx \in \D\,,\\
\strd(0,\xx) &= \strs^0(\xx) 
\hspace{1.4in}
\forall \xx \in \D\,,
\label{eq:initial1dL}
\\
\label{eq:dirichletdL}
\bud&=\bzero 
\hspace{1.7in}
\text{on $ (0,T) \times \partial\D$}
\,.
\end{align}
\end{subequations}

\subsection{Formal energy estimates for (P$_\delta$)}

In this section, we derive \textit{formal} energy estimates, see e.g.\  
\eqref{eq:estimate-Pd} below, 
where we will assume that the triple $(\bud,\prd,\strd)$, 
which is a solution to problem $\brk{\rm P_{\delta}}$,  
has sufficient regularity for all the subsequent manipulations.

We will assume throughout that 
\begin{align}
&\f \in L^2\brk{0,T;[H^{-1}(\D)]^d}, \quad 
\bu^0 \in [L^2(\D)]^d,\quad
\mbox{and}\quad \strs^0 \in [L^\infty(\D)]^{d \times d}_{SPD} 
\nonumber \\
&\mbox{with}
\quad 
\sigma_{\rm min}^0\, \|\bxi\|^2 \leq {\bxi}^T \strs^0(\xx) \,{\bxi} 
\leq \sigma_{\rm max}^0\, \|\bxi\|^2
\quad \forall \bxi \in {\mathbb R}^d
\quad \mbox{for } {a.e.} \ \xx  \mbox{ in }\D;
\label{freg}
\end{align}
where 
$\sigma_{\rm min}^0,\,\sigma_{\rm max}^0 \in \R_{>0}$.

Let $\Fd(\bud,\strd)$ denote the free energy of the solution 
$(\bud,\prd,\strd)$ 
to problem $(\rm P_{\delta})$, where $\Fd(\cdot,\cdot): \U \times \S 
\rightarrow \R$ is defined as  
\beq \label{eq:free-energy-Pd}
\Fd(\bv,\bphi) := \frac{\Re}{2}\intd\|\bv\|^2 
+ \frac{\e}{2\Wi}\intd\tr(\bphi-\Gd(\bphi)-\I) \,.
\eeq
Here
the first term $\frac{\Re}{2}\intd\|\bv\|^2$ 
corresponds to the usual kinetic energy term, 
and the second term $\frac{\e}{2\Wi}\intd\tr(\bphi-\Gd(\bphi)-\I)$ is a 
regularized version of the 
relative entropy term $\frac{\e}{2\Wi}\intd\tr(\bphi-G(\bphi)-\I)$
introduced in Hu and Leli\`{e}vre,\cite{hu-lelievre} see also
Jourdain {\it et al.}\cite{jourdain-lebris-lelievre-otto:2006}.

\begin{proposition} \label{prop:free-energy-Pd}
Let $(\bud,\prd,\strd)$  
be a sufficiently smooth 
solution to problem $(\rm P_{\delta})$.
Then the free energy $\Fd(\bud,\strd)$  
satisfies for a.e.\ $t \in (0,T)$
\begin{align} 
\nonumber
&\deriv{}{t}\Fd(\bud,\strd)
+(1-\e)\intd\|\gbud\|^2
+ \displaystyle \frac{\e}{2 {\rm Wi}^2}\intd\tr(\Bd(\strd)
+[\Bd(\strd)]^{-1}
-2\I)
\\
&\hspace{3.5in}
\leq 
\langle \f, \bud\rangle_{H^1_0(\D)}\,.
\label{eq:estimate-Pd}
\end{align}
\end{proposition}

\begin{proof} 
Multiplying the Navier-Stokes equation with $\bud$
and the stress equation with $\frac{\e}{2\Wi}(\I-\Gd'(\strd))$,
summing and integrating over $\D$ yields,
after 
using integrations by parts and the incompressibility property in the standard
way, 
that
\begin{align}
\nonumber
&\intd  \Brk{ \frac{\Re}{2} 
\pd{}{t}\|\bud\|^2 + (1-\e)\|\gbud\|^2 +\frac{\e}{\Wi}\Bd(\strd):\gbud }
\\ 
& \hspace{0.2in}  + \frac{\e}{2\Wi} \intd \Bigg[ \brk{\pd{}{t}\strd+(\bud\cdot\grad)\strd}
 + \frac{1}{\Wi}\brk{\strd-\I} \Bigg] :\brk{\I-\Gd'(\strd)} 
\nonumber
\\
&\hspace{0.2in}  - \frac{\e}{2\Wi} \intd 
 \brk{ \brk{\gbud}\Bd(\strd) + \Bd(\strd)\brk{\gbud}^T } 
 :\brk{\I-\Gd'(\strd)} 
= 
\langle \f, \bud\rangle_{H^1_0(\D)}\,.
\label{Pdener}
\end{align}
Using (\ref{eq:deriv-tensor-g}) and its spatial counterpart, 
we first note that
\begin{align}
&\brk{\pd{}{t}\strd+(\bud\cdot\grad)\strd}:\brk{\I-\Gd'(\strd)} 
= \brk{\pd{}{t}+(\bud\cdot\grad)}\tr\brk{\strd-\Gd(\strd)} \,.
\label{Pdener1}
\end{align}
On integrating over $\D$, the $(\bud \cdot \grad)$ part of this term vanishes
as $\bud(t,\cdot) \in \Uz$.  
On using trace properties,
\eqref{eq:inverse-Gd} and the incompressibility property, we obtain that
\begin{align}
\nonumber
\brk{\brk{\gbud}\Bd(\strd)}:\brk{\I-\Gd'(\strd)} 
&= \tr\brk{ \brk{\gbud}\Bd(\strd) - \brk{\gbud}\Bd(\strd)
\Gd'(\strd) } \,,
\\
& = \tr\brk{ \brk{\gbud}\Bd(\strd) - \gbud } \,,
\nonumber
\\
& = \tr\brk{ \brk{\gbud}\Bd(\strd) } - \div\bud \,,
\nonumber
\\
& = \tr\brk{ \brk{\gbud}\Bd(\strd) } \,.
\label{Pdener2}
\end{align}
On noting (\ref{eq:symmetric-tr}), we have also that
\begin{align}
\label{Pdener3}
\brk{\Bd(\strd)\brk{\gbud}^T }:\brk{\I-\Gd'(\strd)} 
= \tr\brk{ \brk{\gbud}\Bd(\strd) } \,.
\end{align}
Therefore the terms involving the left-hand sides of (\ref{Pdener2})
and (\ref{Pdener3}) in (\ref{Pdener})  
cancel with the term $\frac{\e}{\Wi}\Bd(\strd):\gbud$
in (\ref{Pdener})  
arising from the Navier-Stokes equation. 
Finally, for the remaining term we have on noting (\ref{ipmat}) and
 (\ref{eq:inverse-Gd},d) that
\begin{align}
\brk{\strd-\I}:
\brk{\I-\Gd'(\strd)}
&= [(\strd-\Bd(\strd)) + \brk{\Bd(\strd)-\I}] :
\brk{\I-\Gd'(\strd)} 
\nonumber
\\
&\geq
\brk{\Bd(\strd)-\I} :
\brk{\I-\Gd'(\strd)}
\nonumber \\
&=\tr(\Bd(\strd)+[\Bd(\strd)]^{-1}-2\I)\,.  
\label{remterm}
\end{align}
Hence we obtain the 
desired free energy inequality  ~\eqref{eq:estimate-Pd}.
\end{proof}

\begin{corollary} 
\label{cor:free-energy-Pd}
Let $(\bud,\prd,\strd)$  
be a sufficiently smooth solution to problem $(\rm P_{\delta})$. 
Then it follows that
\begin{align}
\nonumber
&\sup_{t \in (0,T)}\Fd(\bud(t,\cdot),\strd(t,\cdot))
\\
& \hspace{0.5in}+\int_{\D_T} \Brk{ 
\frac{1}{2}(1-\e)
\|\gbud\|^2
+\frac{\e}{2{\rm Wi}^2}\tr(\Bd(\strd)+[\Bd(\strd)]^{-1}-2\I) }
\nonumber 
\\
& \hspace{1.5in}
\le 2\brk{ \Fd(\bu^0,\strs^0) + \frac{1+C_P}{2(1-\e)} \Norm{\f}_{L^2(0,T;H^{-1}(\D))}^2 }\,.
 \label{eq:free-energy-bound}
\end{align}
\end{corollary}
\begin{proof}
Smooth solutions $(\bud,\prd,\strd)$      
of $(\rm P_{\delta})$ 
satisfy the free energy estimate~\eqref{eq:free-energy-Pd}. 
One can bound the term 
$\langle \f, \bud\rangle_{H^1_0(\D)}$
there,
using the Cauchy-Schwarz and Young inequalities 
for $\nu\in\R_{>0}$,
and the Poincar\'e inequality
(\ref{eq:poincare}),  by
\begin{align}
\nonumber
\langle \f, \bud\rangle_{H^1_0(\D)}
&\le \Norm{\f}_{H^{-1}(\D)} \Norm{\bud}_{H^1(\D)} 
\le \frac{1}{2\nu^2} \Norm{\f}_{H^{-1}(\D)}^2 + \frac{\nu^2}{2} \Norm{\bud}_{H^1(\D)}^2 
\\
& \le 
\frac{1}{2\nu^2} \Norm{\f}_{H^{-1}(\D)}^2 + \frac{\nu^2}{2} 
(1+C_P)
\Norm{\gbud}_{L^2(\D)}^2 
\,. 
\label{fbound}
\end{align} 
Combining (\ref{fbound}) and (\ref{eq:free-energy-Pd}) with 
$\nu^2= (1-\e)/(1+C_P)$, and integrating in time yields the desired 
result (\ref{eq:free-energy-bound}). 
\end{proof}

We note that the right-hand side of (\ref{eq:free-energy-bound}) is independent
of the regularization parameter $\delta$ if $\strs^0$ is positive definite.

\section{Finite element approximation of (P$_\delta$) and (P)}
\label{sec:deltah}

\subsection{Finite element discretization}
\label{FEd}
We now introduce a finite element discretization   
of the problem $\brk{\rm P_\delta}$,
which satisfies a discrete analogue of 
(\ref{eq:estimate-Pd}).

The time interval $[0,T)$ is split into intervals $[t^{n-1},t^n)$ 
with $\dt_{n} = t^{n}-t^{n-1}$, $n=1, \ldots, N_T$. We set
$\dt:= \max_{n=1, \ldots, N_T} \dt_n$.
We will assume throughout that the domain $\D$ is a polytope.
We define a regular family of meshes $\{\mathcal{T}_h\}_{h>0}$ with 
discretization parameter $h>0$,
which is built from partitionings  
of the domain $\D$ into regular open simplices  
so that 
$$ 
\overline{\D} = \mathcal{T}_h := \mathop{\cup}_{k=1}^{N_K} \overline{K_k}\qquad
\mbox{with} \qquad \max_{k=1,\ldots,N_K} \frac{h_{k}}{\rho_{k}} \leq C \,. $$
Here $\rho_{k}$ is the diameter of the largest inscribed ball contained in 
the simplex $K_k$ and $h_{k}$ is the diameter of $K_k$, 
so that $h = \max_{k=1,\ldots,N_K} h_{k}$. 
For each element $K_k$, $k=1,\ldots, N_K$, 
of the mesh $\mathcal{T}_h$ let $\{P^k_i\}_{i=0}^d$ denotes its vertices,
and $\{\bn^k_i\}_{i=0}^d$ the outward unit normals of the edges $(d=2)$ or faces $(d=3)$
with $\bn^k_i$ being that of the edge/face opposite vertex $P^k_i$, $i=0,\ldots,d$.
In addition, let $\{\eta^k_i(\xx)\}_{i=0}^d$ denote the barycentric coordinates 
of $\xx \in K_k$ with respect to the vertices $\{P^k_i\}_{i=0}^d$;
that is, $\eta^k_i \in \PP_1$ and $\eta^k_i(P^k_j)=\delta_{ij}$,
$i,\,j =0 ,\ldots,d$.  
Here $\PP_m$ denote polynomials of maximal degree $m$ in $\xx$,
and $\delta_{ij}$ the Kronecker delta notation.
Finally, we introduce $\partial\mathcal{T}_h := \{E_j\}_{j=1}^{N_E}$ 
as the set of internal edges $E_j$ of triangles in the mesh $\mathcal{T}_h$ when $d=2$, 
or the set of internal faces $E_j$ of tetrahedra when $d=3$.

We approximate the problem $(\rm P_\delta)$ by the problem $(\mathrm{P}_{\delta,h}^{\dt})$
based on the finite element spaces $\Uh^0 \times \mathrm{Q}_h^0 \times \S_h^0$.
As is standard, we require the discrete velocity-pressure spaces 
$\Uh^0\times\mathrm{Q}_h^0 \subset \U \times {\rm Q}$  
 satisfy the discrete Ladyshenskaya-Babu\v{s}ka-Brezzi (LBB) inf-sup condition
 \beq \label{eq:LBB}
 \inf_{q\in\mathrm{Q}_h^0} \sup_{\bv\in\Uh^0} 
 \frac{\displaystyle \int_{\D} q\,\div\bv }{\Norm{q}_{L^2(\D)}\Norm{\bv}_{H^1(\D)}} 
 \geq \mu_{\star} 
 > 0 \,,
 \eeq 
 see e.g.\ p114 in Girault and Raviart.\cite{GiraultRaviart}
 In the following, we set
 \begin{subequations}
 \begin{align}
 \Uh^0&:=\Uh^2  \subset \U \quad \mbox{or} \quad \Uh^{2,-}
 \subset \U
 \label{Vh}\,,\\
 \mathrm{Q}_h^0 &:=\{q \in {\rm Q} \,:\, q \mid_{K_k} \in \PP_0 \quad
 k=1,\ldots, N_K\} \subset {\rm Q}\,,  
 \label{Qh}\\
\mbox{and} \qquad \S_h^0 &:=\{\bphi \in \S \,:\, \bphi \mid_{K_k} 
\in [\PP_0]^{d \times d}_S \quad
 k=1,\ldots, N_K\} \subset \S\,; 
 \label{Sh}
 \end{align}
 \end{subequations}
 where 
 \begin{subequations}
 \begin{align}
 \Uh^2&:=\{\bv \in [C(\overline{\D})]^d\cap \U \,:\, \bv \mid_{K_k} \in [\PP_2]^d \quad
 k=1,\ldots, N_K\} \label{Vh2}\,,\\
  \Uh^{2,-}&:=\{\bv \in [C(\overline{\D})]^d\cap \U \,:\, \bv \mid_{K_k} \in [\PP_1]^d 
  \oplus \mbox{span} \{\bvarsigma^{k}_i\}_{i=0}^d  \quad
 k=1,\ldots, N_K\}\label{Vh2m}\,.
 \end{align} 
\end{subequations} 
Here, for $k=1,\ldots, N_K$ and $i=0,\ldots,d$
\begin{align}
\bvarsigma^{k}_i(\xx) = \bn^{k}_i \prod_{j=0, j\neq i}^d 
\eta^{k}_j(\xx) \qquad \mbox{for } \xx \in K_k\,.
\label{bvarsig}
\end{align}
We introduce also  
$$ \Vhzero := \BRK{\bv \in\Uh^0 \,:\, \int_{\D} q \, {\div\bv} = 0 \quad  
\forall q\in\mathrm{Q}_h^0} \,, $$
which approximates $\Uz$.
 It is well-known the choices (\ref{Vh},b)   
 satisfies (\ref{eq:LBB}), see e.g.\ p221 in Brezzi and Fortin\cite{Brezzi-Fortin} 
 for $\Uh^0 =\Uh^2$, and Chapter II, Sections 2.1 ($d=2$) and 2.3 
 ($d=3$) in Girault and Raviart\cite{GiraultRaviart}
 for $\Uh^0 =\Uh^{2,-}$. 
 Moreover, these particular choices of $\S_h^0$ and $\mathrm{Q}_h^0$
 have the desirable property   
 that
 \beq \label{eq:inclusion1}
 \bphi \in \Sh^0 \qquad \Rightarrow \qquad 
 \I-\Gd'(\bphi) \in \Sh^0 \qquad \mbox{and} \qquad
 \tr\brk{ \bphi-\Gd(\bphi)-\I } \in \mathrm{Q}_h^0 \,,
 \eeq
which makes it a straightforward matter to mimic the free
energy inequality (\ref{eq:estimate-Pd}) at a discrete level.    
Since  
$\Sh^0$ is discontinuous,
we will use the discontinuous Galerkin method to approximate the advection term 
$(\bud\cdot\nabla)\strd$ in the following. 
Then, for the boundary integrals, we will make use of the following definitions
(see e.g.\ 
p267 in Ern and Guermond\cite{ern-guermond:2004}).
Given $\bv \in \Uh^0$,  
then for any $\bphi \in \Sh^0$ (or ${\rm Q}_h^0$) 
and for any point $\xx$ that is in the interior of some $E_j \in \partial\mathcal{T}_h$,
we define the downstream and upstream values of $\bphi$ at $\xx$ by
\begin{align}\label{eq:streams}
\bphi^{+\bv}(\xx) = \lim_{\rho \rightarrow 0^+} \bphi(\xx+\rho\,\bv(\xx)) \qquad
\mbox{and} \qquad 
\bphi^{-\bv}(\xx) = \lim_{\rho \rightarrow 0^-} \bphi(\xx+\rho\,\bv(\xx)) \,;
\end{align}
respectively. 
In addition, we denote by
\begin{align}
\label{eq:difsum}
\jump{\bphi}_{\to\bv}(\xx) = \bphi^{+\bv}(\xx) - \bphi^{-\bv}(\xx) 
\qquad \mbox{and}
\qquad  
\BRK{\bphi}^{\bv}(\xx) = \frac{\bphi^{+\bv}(\xx) + \bphi^{-\bv}(\xx)}{2}\,,
\end{align}
the jump and mean value, respectively, of $\bphi$ 
at the point $\xx$ of boundary $E_j$.
From 
(\ref{eq:streams}), 
it is clear that the values of $\bphi^{+\bv}|_{E_j}$ and $\bphi^{-\bv}|_{E_j}$
can change along $E_j \in \partial\mathcal{T}_h$. 
Finally, it is easily deduced that
\begin{align}\nonumber
\sum_{j=1}^{N_E} \int_{E_j} |\bv\cdot\bn| \jump{q_1}_{\to\bv} \,q_2^{+\bv}
&= -\sum_{k=1}^{N_K} \int_{\partial K_k} \brk{\bv \cdot \bn_{K_k}} q_1\,q_2^{+\bv}
\\
& \hspace{1in}
\qquad \forall \bv \in \Uh^0, \ q_1,\,q_2 \in {\rm Q}_h^0\,;
\label{eq:jump}
\end{align}
where $\bn\equiv\bn(E_j)$ is a unit normal to $E_j$, 
whose sign is of no importance,
and $\bn_{K_k}$ is the outward unit normal vector 
of boundary $\partial K_k$ of $K_k$. 
We note that similar ideas appear in upwind schemes; 
e.g.\
see Chapter IV, Section 5 in Girault and Raviart\cite{GiraultRaviart}
for the Navier-Stokes equations.

\subsection{A free energy preserving approximation, (P$^{\Delta t}_{\delta,h}$), of (P$_\delta$)}

For any source term $\f\in L^2\brk{0,T;[H^{-1}(\D)]^d}$, we define 
the following piecewise constant function with respect to the time variable
\beq
\label{eq:constant-source}
\f^{\Delta t,+}(t,\cdot)=\f^n(\cdot) 
:=\frac{1}{\dt_n} \int_{t^{n-1}}^{t^{n}} \f(t,\cdot)\, dt,
\qquad t\in [t^{n-1},t^n), \qquad n=1,\dots,N_T  
\,. 
\eeq
It is easily deduced that
\begin{subequations}
\begin{align}
&\sum_{n=1}^{N_T} \dt_n\,\|\f^n\|_{H^{-1}(\D)}^r 
\leq \int_{0}^T \|f(t,\cdot)\|_{H^{-1}(\D)}^r dt
\qquad \mbox{for any } r \in [1,2]\,,
\label{fncont}\\
\mbox{and} \qquad
& \f^{\Delta t,+} \rightarrow \f \quad \mbox{strongly in }
L^2(0,T;[H^{-1}(\D)]^d) \mbox{ as } \dt \rightarrow 0_+\,.
\label{fnconv}
\end{align}
\end{subequations}

Throughout this section we choose $\buh^0 \in \Vhzero$
to be a suitable approximation of $\bu^0$ such as the $L^2$
projection of $\bu^0$ onto $\Vhzero$. We will also choose
$\strh^0 \in \Sh^0$ to be the $L^2$
projection of $\strs^0$ onto $\Sh^0$. 
Hence for $k=1,\ldots,N_K$
\begin{subequations}
\begin{align}
\strh^0 \mid_{K_k} = \frac{1}{|K_k|} \int_{K_k} \strs^0
\,,
\label{strh0def}
\end{align}
where $|K_k|$
is the measure of $K_k$;
and it immediately follows from (\ref{freg}) that
\begin{align}
\sigma_{\rm min}^0\, \|\bxi\|^2 \leq {\bxi}^T \strh^0\mid_{K_k} \,{\bxi} 
\leq \sigma_{\rm max}^0\, \|\bxi\|^2
\qquad \forall \bxi \in {\mathbb R}^d\,.
\label{strh0pd}
\end{align}
\end{subequations}

Our approximation $(\mathrm{P}_{\delta,h}^{\dt})$ of $\brk{\mathrm{P}_\delta}$
is then:

({\bf P}$_{\delta,h}^{\Delta t}$) 
Setting
$(\buhd^0,\strhd^0)=(\buh^0,\strh^0) \in\Vhzero\times\Sh^0$,
then for $n = 1, \ldots, N_T$ 
find $(\buhd^{n},\strhd^{n})\in\Vhzero\times\Sh^0$ such that for any test functions 
$(\bv,\bphi)\in\Vhzero\times\Sh^0$
\begin{subequations}
\begin{align} 
\nonumber 
&\int_\D \biggl[ \Re\left(\frac{\buhd^{n}-\buhd^{n-1}}{\dt_{n}}\right)\cdot \bv 
 + \frac{\Re}{2}\Brk{ \left( (\buhd^{n-1}\cdot\nabla)\buhd^{n}\right) \cdot \bv - 
 \buhd^{n} \cdot \left( (\buhd^{n-1}\cdot\nabla)\bv \right)}
\\
& \hspace{0.85in}
 + (1-\e) \gbuhd^{n}:\grad\bv + \frac{\e}{\Wi} \Bd(\strhd^{n}) : \grad\bv 
\biggr] = 
\langle \f^n, \bv\rangle_{H^1_0(\D)}\,,
\label{eq:Pdh}
\\ 
\nonumber
&\int_\D \left[ \left(\frac{\strhd^{n}-\strhd^{n-1}}{\dt_{n}}\right) : \bphi 
 - 2\left( (\gbuhd^{n})\,\Bd(\strhd^{n})\right) :\bphi 
 + \frac{1}{\Wi}\left(\strhd^{n}-\I\right): \bphi 
\right]
\\
& \hspace{0.85in}
 + \sum_{j=1}^{N_E} \int_{E_j}  \Abs{\buhd^{n-1}\cdot\bn} 
 \jump{\strhd^{n}}_{\to\buhd^{n-1}}:\bphi^{+\buhd^{n-1}}
 = 0 \,.
\label{eq:Pdhb}
\end{align}
\end{subequations}

In deriving $(\mathrm{P}_{\delta,h}^{\dt})$, we have noted 
(\ref{eq:symmetric-tr}) and 
that 
\begin{align}
\int_\D \bv \cdot [(\bz \cdot\nabla)\bw]= - 
 \int_\D \bw \cdot [(\bz\cdot\nabla)\bv] 
\qquad 
\forall \bz \in \Uz, \qquad \forall \bv, \bw \in [H^1(\D)]^d
\,.
\label{conv0c}
\end{align}
Once again we refer to p267 in Ern and Guermond\cite{ern-guermond:2004}
for the consistency of our stated approximation of the stress convection term,
see also Boyaval {\it et al.}\cite{boyaval-lelievre-mangoubi-09}.

Before proving existence of a solution to (P$^{\Delta t}_{\delta,h}$),
we first derive a discrete 
analogue of the energy estimate (\ref{eq:estimate-Pd}) for
(P$^{\Delta t}_{\delta,h}$); which uses the elementary equality
\begin{align}
2s_1 (s_1-s_2)=s_1^2-s_2^2 +(s_1-s_2)^2 \qquad \forall s_1,s_2 \in \R.
\label{elemident}
\end{align}

\subsection{Energy bound for (P$^{\Delta t}_{\delta,h}$)}

\begin{proposition} \label{prop:free-energy-Pdh} 
For $n= 1, \dots, N_T$, a solution $\brk{\buhd^{n},\strhd^{n}}
\in \Vhzero\times\Sh^0$ to (\ref{eq:Pdh},b), if it exists, satisfies 
\begin{align}\nonumber
&\frac{F_\delta(\buhd^{n},\strhd^{n})-F_\delta(\buhd^{n-1},\strhd^{n-1})}{\dt_{n}} 
+ \frac{ {\rm Re}}{2\dt_{n}}\intd\|\buhd^{n}-\buhd^{n-1}\|^2
+(1-\e)\intd\|\gbuhd^{n}\|^2
\\
& \hspace{1in} +\frac{\e}{2{\rm Wi}^2}\intd\tr(\Bd(\strhd^{n})+[\Bd(\strhd^{n})]^{-1}-2\I)
\nonumber \\
& \qquad
\le
\langle \f^n, \buhd^n \rangle_{H^1_0(\D)}
\le 
\frac{1}{2}(1-\e)\intd\|\gbuhd^{n}\|^2
+ \frac{1+C_P}{2(1-\e)} \|\f^{n}\|_{H^{-1}(\D)}^2
\,.
 \label{eq:estimate-Pdh}
\end{align}
\end{proposition}
\begin{proof}
Similarly to the proof of  Proposition~\ref{prop:free-energy-Pd},
we choose as test functions $\bv=\buhd^{n}\in\Vhzero$ and 
$\bphi=\frac{\e}{2\Wi}\brk{\I-\Gd'(\strhd^{n})}\in\Sh^0$ in (\ref{eq:Pdh},b), 
and obtain, on noting (\ref{elemident}) and (\ref{eq:inverse-Gd},d), that
\begin{align}
\nonumber
&\langle \f^n, \buhd^n \rangle_{H^1_0(\D)}
\\
&\hspace{0.35in}\ge
\intd
\Brk{ \frac{\Re}{2} \brk{ \frac{\|\buhd^{n}\|^2-\|\buhd^{n-1}\|^2}{\dt_{n}} 
+ \frac{\|\buhd^{n}-\buhd^{n-1}\|^2}{\dt_{n}} } + (1-\e)\|\gbuhd^{n}\|^2 }
\nonumber
\\
&\hspace{0.55in} + \frac{\e}{2\Wi}  \intd 
\brk{ \frac{\strhd^{n}-\strhd^{n-1}}{\dt_{n}} }:\brk{\I-\Gd'(\strhd^{n})} 
\nonumber \\
& \hspace{0.55in}
+ \frac{\e}{2\Wi^2}  \intd 
\tr(\Bd(\strhd^{n})+[\Bd(\strhd^{n})]^{-1}-2\I) 
\nonumber
\\
&\hspace{0.55in} 
 +  \frac{\e}{2\Wi}\sum_{j=1}^{N_E} \int_{E_j} 
     \Brk{ \Abs{\buhd^{n-1}\cdot\bn} \jump{\strhd^{n}}_{\to\buhd^{n-1}}:
     \brk{\I-\Gd'(\strhd^{n})}^{+\buhd^{n-1}} } \,.
\label{eq:free-energy-Pdh-demo1}
\end{align}
We consequently obtain from \eqref{eq:free-energy-Pdh-demo1},
on noting (\ref{ipmat}) and (\ref{eq:concavity}) applied
to the edge terms as well as the discrete time derivative term for the stress variable, 
that
\begin{align} 
\nonumber 
&\langle \f^n, \buhd^n \rangle_{H^1_0(\D)}
\\
&\hspace{0.35in}\ge 
\nonumber 
\intd
 \Brk{\frac{\Re}{2}\brk{\frac{\|\buhd^{n}\|^2-\|\buhd^{n-1}\|^2}{\dt_{n}} 
  + \frac{\|\buhd^{n}-\buhd^{n-1}\|^2}{\dt_{n}}} +(1-\e)\|\gbuhd^{n}\|^2 }
\\ 
& \hspace{0.55in}   
+ \frac{\e}{2\Wi}  \intd
  \brk{\frac{\tr\brk{\strhd^{n}-\Gd(\strhd^{n})}
  -\tr\brk{\strhd^{n-1}-\Gd(\strhd^{n-1})}}{\dt_{n}}}
\nonumber
\\
& \hspace{0.55in} 
+ \frac{\e}{2\Wi^2}  \intd 
\tr(\Bd(\strhd^{n})+[\Bd(\strhd^{n})]^{-1}-2\I) 
\nonumber
\\
& \hspace{0.55in} +  \frac{\e}{2\Wi}  \sum_{j=1}^{N_E}\int_{E_j}  
 \Abs{\buhd^{n-1}\cdot\bn} \jump{ \tr\brk{\strhd^{n}-\Gd(\strhd^{n})} 
 }_{\to\buhd^{n-1}} \,.
\label{eq:free-energy-Pdh-demo2}
\end{align}
Finally, we note from (\ref{eq:jump}), (\ref{eq:inclusion1}) and as $\buhd^{n-1} 
\in \Vhzero$ that
\begin{align} 
&\sum_{j=1}^{N_E} \int_{E_j} \Abs{\buhd^{n-1}\cdot\bn} \jump{\tr \brk{ \strhd^{n}-
\Gd(\strhd^{n}) } }_{\to\buhd^{n-1}}
\nonumber
\\
& \hspace{1in}= 
 -\sum_{k=1}^{N_K} \int_{\partial K_k} \brk{\buhd^{n-1}\cdot\bn_{K_k}} 
 \tr\brk{\strhd^{n}
 -\Gd(\strhd^{n})} 
\nonumber
 \\
& \hspace{1in} = 
-\sum_{k=1}^{N_K} \int_{K_k} \div\brk{\buhd^{n-1}\tr\brk{\strhd^{n}-\Gd(\strhd^{n})}} 
\nonumber
\\
&\hspace{1in}= - 
\int_{\D} 
 \tr\brk{\strhd^{n}-\Gd(\strhd^{n})}
\div \buhd^{n-1} 
= 0 \,.
\label{edgeterm} 
\end{align}
Combining (\ref{eq:free-energy-Pdh-demo2}) and (\ref{edgeterm})
yields the first desired inequality in (\ref{eq:estimate-Pdh}).
The second inequality in (\ref{eq:estimate-Pdh}) follows immediately from (\ref{fbound})
with $\nu^2 = (1-\e)/(1+C_P)$. 
\end{proof}

\subsection{Existence of a solution to (P$^{\Delta t}_{\delta,h}$)}

\begin{proposition} 
\label{prop:existence-Pdh}
Given $(\buhd^{n-1},\strhd^{n-1}) \in \Vhzero \times \Sh^0$
and for any time step $\dt_{n} > 0$,
then there exists at least 
one solution $\brk{\buhd^{n},\strhd^{n}} \in \Vhzero\times\Sh^0$ 
to 
(\ref{eq:Pdh},b).
\end{proposition}
\begin{proof}
We introduce the following inner product on 
the Hilbert space $\Vhzero\times\Sh^0$
\begin{align} 
 \Scal{ (\bw,\bpsi) }{ (\bv,\bphi) }_\D = 
\intd \Brk{ \bw\cdot\bv + \bpsi:\bphi } 
\qquad \forall (\bw,\bpsi),(\bv,\bphi) \in \Vhzero\times\Sh^0\,.
\label{IPD}
\end{align}
Given $(\buhd^{n-1},\strhd^{n-1})\in\Vhzero\times\Sh^0$, let
$\mathcal{F} : \Vhzero\times\Sh^0 \rightarrow \Vhzero\times\Sh^0$ be such that for any
$(\bw,\bpsi) \in \Vhzero\times\Sh^0$
\begin{align} 
&\Scal{\mathcal{F}(\bw,\bpsi)}{(\bv,\bphi)}_\D 
\nonumber \\
&\qquad := \int_\D \biggl[\Re \left(\frac{\bw-\buhd^{n-1}}{\dt_{n}}\right) \cdot \bv
  + \frac{\Re}{2} \left[\left((\buhd^{n-1}\cdot\nabla)\bw \right) \cdot \bv 
  - \bw \cdot \left((\buhd^{n-1}\cdot\nabla)\bv \right) \right]  
\nonumber \\
& \hspace{0.6in} + (1-\e) \grad\bw :\grad\bv + \frac{\e}{\Wi} \Bd(\bpsi) 
:\grad\bv 
+ \left(\frac{\bpsi-\strhd^{n-1}}{\dt_{n}}\right):\bphi 
\nonumber \\ 
& \hspace{0.6in} 
 - 2\left((\grad\bw) \,\Bd(\bpsi)\right):\bphi 
 + \frac{1}{\Wi}\left(\bpsi-\I\right): \bphi
 \biggr] 
-\langle \f^n, \bv\rangle_{H^1_0(\D)}
\nonumber
\\
& \hspace{0.6in} + \sum_{j=1}^{N_E} \int_{E_j}  
\Abs{\buhd^{n-1}\cdot\bn} \jump{\bpsi}_{\to\buhd^{n-1}}:\bphi^{+\buhd^{n-1}} 
\qquad \forall (\bv,\bphi) \in \Vhzero \times \Sh^0 
\,.
\label{eq:mapping}
\end{align}
We note that a solution 
$(\buhd^{n},\strhd^{n})$
to 
(\ref{eq:Pdh},b), if it exists,
corresponds to a zero of $\mathcal{F}$; that is,
\beq \label{eq:solution}
\Scal{\mathcal{F}(\buhd^{n},\strhd^{n})}{(\bv,\bphi)}_\D = 0 \qquad 
\forall (\bv,\bphi) \in \Vhzero\times\Sh^0 \,. 
\eeq
In addition, it is easily deduced that the mapping $\mathcal{F}$ is continuous.

For any $(\bw,\bpsi) \in \Vhzero\times\Sh^0$, on choosing
$(\bv,\bphi) = \brk{\bw,\frac{\e}{2\Wi}\brk{\I-\Gd'(\bpsi}}$,
we obtain analogously to (\ref{eq:estimate-Pdh}) that
\begin{align} 
\nonumber
&\Scal{\mathcal{F}(\bw,\bpsi)}{\brk{\bw,\frac{\e}{2\Wi}\brk{\I-\Gd'(\bpsi)}}}_\D
\\
& \qquad \ge \frac{F_\delta(\bw,\bpsi)-
F_\delta(\buhd^{n-1},\strhd^{n-1})}{\dt_{n}} 
+ \frac{\Re}{2\dt_{n}}\intd\|\bw-\buhd^{n-1}\|^2
+ 
\frac{1-\e}{2} 
\intd\|\grad\bw\|^2
\nonumber
\\
& \qquad\qquad 
+\frac{\e}{2\Wi^2}\intd\tr(\Bd(\bpsi)+[\Bd(\bpsi)]^{-1}-2\I) 
-\frac{1+C_P}{2(1-\e)}\Norm{\f^{n}}_{H^{-1}(\D)}^2 \,.
\label{eq:inequality1}
\end{align}

Let us now assume that for any $\gamma \in \R_{>0}$, 
the continuous mapping $\mathcal{F}$ has no zero $(\buhd^{n},\strhd^{n})$ 
satisfying~\eqref{eq:solution}, which lies in the ball
\begin{align}
\mathcal{B}_\gamma := \BRK{ (\bv,\bphi) \in \Vhzero\times\Sh^0 \,: \, 
\Norm{(\bv,\bphi)}_\D\le\gamma } \,;
\label{Bgamma}
\end{align}
where 
\begin{align}
\Norm{(\bv,\bphi)}_D := 
\left[((\bv,\bphi),(\bv,\bphi))_D\right]^{\frac{1}{2}} =
\brk{\intd [\,\|\bv\|^2+\|\bphi\|^2\,] }^\frac12 \,. 
\label{NormD}
\end{align}
Then for such $\gamma$, we can define the continuous mapping $\mathcal{G}_\gamma
: \mathcal{B}_\gamma \rightarrow  \mathcal{B}_\gamma$
such that for all $(\bv,\bphi) \in \mathcal{B}_\gamma$
\begin{align}
\mathcal{G}_\gamma(\bv,\bphi) := -\gamma 
\frac{\mathcal{F}(\bv,\bphi)}{\Norm{\mathcal{F}(\bv,\bphi)}_\D} \,. 
\label{Ggamma}
\end{align}
By the Brouwer fixed point theorem, $\mathcal{G}_\gamma$ has at least 
one fixed point $(\bw_\gamma,\bpsi_\gamma)$ in $\mathcal{B}_\gamma$. 
Hence it satisfies
\begin{equation}\label{eq:fixed-point}
\Norm{(\bw_\gamma,\bpsi_\gamma)}_\D=
 \Norm{\mathcal{G}_\gamma(\bw_\gamma,\bpsi_\gamma)}_\D=\gamma.
\end{equation}

It follows, on noting  
(\ref{Sh}), (\ref{NormD}) and (\ref{eq:fixed-point}), that
\begin{equation}\label{eq:norm_equivalence} 
\|\bpsi_\gamma\|_{L^\infty(\D)}^2
\leq \frac{1}{\min_{k\in N_K} |K_k|} \int_\D \|\bpsi_\gamma\|^2
\equiv \mu_h^2 \int_\D \|\bpsi_\gamma\|^2
\leq \mu_h^2\,\gamma^2,
\end{equation}
where $\mu_h := [1/(\min_{k\in N_K} |K_k|)]^{\frac{1}{2}}$. 
It follows from (\ref{eq:free-energy-Pd}), (\ref{Entropy2}), 
(\ref{eq:norm_equivalence})  
and (\ref{eq:fixed-point}) 
that 
\begin{align} 
\nonumber
F_\delta(\bw_\gamma,\bpsi_\gamma) & =
\frac{\Re}{2} \intd \|\bw_\gamma\|^2 + \frac{\e}{2\Wi}
\intd
\tr(\bpsi_\gamma-\Gd(\bpsi_\gamma)-\I)
\\
& \ge
\frac{\Re}{2} \intd\|\bw_\gamma\|^2 + \frac{\e}{4\Wi}
\left[\intd \|\bpsi_\gamma\| -2d|\D|\right]
\nonumber
\\
& \ge 
\frac{\Re}{2} \intd\|\bw_\gamma\|^2 + \frac{\e}{4\Wi\,\mu_h\gamma}
\|\bpsi_\gamma\|_{L^\infty(\D)}
\left[\intd \|\bpsi_\gamma\| \right]
-\frac{\e d |\D|}{2\Wi}
\nonumber
\\
& \ge
\min\brk{\frac{\Re}{2},\frac{\e}{4\Wi\,\mu_h\gamma}}
\brk{ \intd\left[\,\|\bw_\gamma\|^2+ \|\bpsi_\gamma\|^2\,\right] }
-\frac{\e d |\D|}{2\Wi}
\nonumber
\\
& = 
\min\brk{\frac{\Re}{2},\frac{\e}{4\Wi\,\mu_h\gamma}}
\gamma^2 
-\frac{\e d |\D|}{2\Wi}
\,.
\label{bb1} 
\end{align}
Hence 
for all $\gamma$ sufficiently large,
it follows from (\ref{eq:inequality1}) and (\ref{bb1})   
that
\beq \label{eq:one-hand}
\Scal{\mathcal{F}(\bw_\gamma,\bpsi_\gamma)}{\brk{\bw_\gamma,\frac{\e}{2\Wi}\brk{\I
-\Gd'(\bpsi_\gamma)}}}_\D
\ge 0\,.
\eeq

On the other hand as $(\bw_\gamma,\bpsi_\gamma)$ is a fixed point 
of ${\mathcal G}_\gamma$, we have that
\begin{align}
\nonumber
&\Scal{\mathcal{F}(\bw_\gamma,\bpsi_\gamma)}{\brk{\bw_\gamma,\frac{\e}{2\Wi}
\brk{\I-\Gd'(\bpsi_\gamma)}}}_D
\\ 
& \hspace{0.5in} = 
-\frac{\Norm{\mathcal{F}(\bw_\gamma,\bpsi_\gamma)}_\D}{\gamma} 
\intd \left[ \|\bw_\gamma\|^2+ \frac{\e}{2\Wi} 
\bpsi_\gamma :\brk{\I-\Gd'(\bpsi_\gamma)}\right] \,.
\label{eq:whereas}
\end{align}
It follows from (\ref{Entropy2}), and similarly to (\ref{bb1}),
on noting (\ref{eq:norm_equivalence}) and (\ref{eq:fixed-point}) that
\begin{align}
\intd \left[
\|\bw_\gamma\|^2 + \frac{\e}{2\Wi} \bpsi_\gamma :\brk{\I-\Gd'(\bpsi_\gamma)} \right] 
&\ge  
\intd \left[
\|\bw_\gamma\|^2 + \frac{\e}{4\Wi} \left[\|\bpsi_\gamma\| - 2d \right]\right]
\nonumber
\\
&\ge \min\brk{1,\frac{\e}{4\Wi\,\mu_h \gamma}}\gamma^2 - 
\frac{\e d|\D|}{2 \Wi}\,. 
\label{bb2} 
\end{align}
Therefore on combining (\ref{eq:whereas}) and (\ref{bb2}), we have for all 
$\gamma$ sufficiently large that 
\begin{align} \label{eq:other-hand}
\Scal{\mathcal{F}(\bw_\gamma,\bpsi_\gamma)}{\brk{\bw_\gamma,\frac{\e}{2\Wi}
\brk{\I-\Gd'(\bpsi_\gamma)}}}_D < 0 \,,
\end{align}
which obviously contradicts~\eqref{eq:one-hand}.
Hence the mapping $\mathcal{F}$ has a zero in $\mathcal{B}_\gamma$
for $\gamma$ sufficiently large. 
\end{proof}

\begin{theorem}
\label{dstabthm}
For any $\delta \in (0, \frac{1}{2}]$,
$N_T \geq 1$ and any
partitioning of $[0,T]$ into $N_T$
time steps,  
then
there exists a solution  
$\{(\buhd^{n},\strhd^{n})\}_{n=1}^{N_T}
\in [\Vhzero \times \Sh^0]^{N_T}
$ 
to (P$^{\dt}_{\delta,h}$).

In addition, it follows for $n=1,\ldots,N_T$ that
\begin{align}
&F_{\delta}(\buhd^n,\strhd^n) + \frac{1}{2} \sum_{m=1}^n
\int_\D \left[ {\rm Re} \|\buhd^m-\buhd^{m-1}\|^2
+ (1-\e)\dt_{m}\|\gbuhd^m\|^2
\right]
\nonumber
\\
&\hspace{0.5in} 
+\frac{\e}{2{\rm Wi}^2}
\sum_{m=1}^n \dt_m
\intd\tr(\Bd(\strhd^{m})+[\Bd(\strhd^{m})]^{-1}-2\I)
\nonumber
\\
& \hspace{1in}
\leq
F_{\delta}(\buh^0,\strh^0)
+ \frac{1+C_P}{2(1-\e)} \sum_{m=1}^n \dt_m \|\f^m\|_{H^{-1}(\D)}^2
\leq C
\,. 
\label{Fstab1}
\end{align}
 Moreover, it follows that 
\begin{align}
&\max_{n=0, \ldots, N_T} \int_\D \left[ \|\buhd^n\|^2 + \|\strhd^n\| + 
\delta^{-1}\,\|[\strhd^n]_{-}\| \right]
+ \sum_{n=1}^{N_T} \dt_n \intd
\|[\beta_\delta(\strhd^n)]^{-1}\| 
\leq C\,. 
\label{Fstab2}
\end{align} 
 \end{theorem}
\begin{proof}
Existence and the stability result (\ref{Fstab1}) follow 
immediately from Propositions \ref{prop:existence-Pdh}
and \ref{prop:free-energy-Pdh}, respectively,
on noting 
(\ref{eq:free-energy-Pd}), 
(\ref{strh0pd}),
(\ref{fncont}) and (\ref{freg}).
The bounds (\ref{Fstab2}) follow immediately from 
(\ref{Fstab1}), on noting   (\ref{eq:positive-term}),
(\ref{Entropy2}), (\ref{modphisq}) and the fact that 
$\beta_\delta(\bphi) \in \RSPD$ for any $\bphi \in \RS$. 
 \end{proof}

\subsection{Convergence of (P$^{\Delta t}_{\delta,h}$) to (P$^{\Delta t}_{h}$)}

We now consider the
corresponding direct finite element approximation of (P),
i.e. (P$^{\Delta t}_{\delta,h}$) without the regularization $\delta$: 

({\bf P}$^{\Delta t}_h$)
Given initial conditions 
$(\buh^0,\strh^0)\in\Vhzero\times\Sh^0$ with $\strh^0$ satisfying (\ref{strh0def},b),
then for $n = 1, \ldots, N_T$ 
find $(\buh^{n},\strh^{n})\in\Vhzero\times\Sh^0$ such that for any test functions 
$(\bv,\bphi)\in\Vhzero\times\Sh^0$
\begin{subequations}
\begin{align} 
\nonumber
&\int_\D \biggl[ \Re\left(\frac{\buh^{n}-\buh^{n-1}}{\dt_{n}}\right)\cdot \bv 
 + \frac{\Re}{2}\Brk{ \left( (\buh^{n-1}\cdot\nabla)\buh^{n}\right) \cdot \bv - 
 \buh^{n} \cdot \left( (\buh^{n-1}\cdot\nabla)\bv \right)}
\\
& \hspace{1in}
 + (1-\e) \gbuh^{n}:\grad\bv + \frac{\e}{\Wi} \strh^{n} : \grad\bv 
\biggr] = 
\langle \f^n,\bv\rangle_{H^1_0(\D)}\,,
 \label{eq:Pdh-limit}
\\ 
\nonumber
&\int_\D \left[ \left(\frac{\strh^{n}-\strh^{n-1}}{\dt_{n}}\right) : \bphi 
 - 2\left((\gbuh^{n})\,\strh^{n}\right) :\bphi 
 + \frac{1}{\Wi}\left(\strh^{n}-\I\right): \bphi 
\right]
\\
& \hspace{1in}
 + \sum_{j=1}^{N_E} \int_{E_j}  \Abs{\buh^{n-1}\cdot\bn} 
 \jump{\strh^{n}}_{\to\buh^{n-1}}:\bphi^{+\buh^{n-1}}
 = 0 \,.
\label{eq:Pdhb-limit}
\end{align}
\end{subequations}
We introduce also the unregularised free energy 
\begin{align} 
F(\bv,\bphi) := \frac{\Re}{2}\intd\|\bv\|^2 + 
\frac{\e}{2\Wi}\intd\tr(\bphi-G(\bphi)-\I) \,,
\label{unreg-energy}
\end{align}
which is well defined for $(\bv,\bphi) \in \Vhzero \times \Sh^0$ 
with $\bphi$ being positive definite on $\D$.

\begin{theorem} 
\label{dconthm}
For all regular partitionings $\mathcal{T}_h$ of $\D$
into simplices $\{K_k\}_{k=1}^{N_K}$
and all partitionings $\{\Delta t_n\}_{n=1}^{N_T}$
of $[0,T]$,  
there exists a subsequence
$\{\{(\buhd^{n},\strhd^{n})\}_{n=1}^{N_T}\}_{\delta>0}$, where  
$\{(\buhd^{n},$ $\strhd^{n})\}_{n=1}^{N_T} 
\in [\Vhzero \times \Sh^0]^{N_T}$ 
solves (P$^{\dt}_{\delta,h}$), and  
$\{(\buh^{n},\strh^{n})\}_{n=1}^{N_T} \in [\Vhzero \times \Sh^0]^{N_T}$ 
such that for the subsequence
\begin{align}
\buhd^{n} \rightarrow \buh^{n}, \qquad 
\strhd^{n} \rightarrow \strh^{n} \qquad \mbox{as } \delta \rightarrow 0_+\,,
\qquad\mbox{for} \quad n=1, \ldots, N_T\,.
\label{conv}
\end{align}
In addition, for 
$n=1,\ldots, N_T$, $\strh^{n}\mid_{K_k} \in \RSPD$, $k=1,\ldots, N_K,$. 
Moreover, $\{(\buh^{n},\strh^{n})\}_{n=1}^{N_T} 
\in [\Vhzero \times \Sh^0]^{N_T}$
solves (P$_{h}^{\dt}$) and 
for $n=1,\ldots,N_T$
\begin{align}
\nonumber
&\frac{F(\buh^{n},\strh^{n})-F(\buh^{n-1},\strh^{n-1})}{\dt_{n}} 
+ \frac{ {\rm Re}}{2\dt_{n}}\intd\|\buh^{n}-\buh^{n-1}\|^2
+(1-\e)\intd\|\gbuh^{n}\|^2
\\
& \hspace{1in} +\frac{\e}{2{\rm Wi}^2}\intd\tr(\strh^{n}+[\strh^{n}]^{-1}-2\I)
\nonumber \\
& \qquad
\le 
\frac{1}{2}(1-\e)\intd\|\gbuh^{n}\|^2
+ \frac{1+C_P}{2(1-\e)} \|\f^{n}\|_{H^{-1}(\D)}^2
\,.
\label{eq:estimate-Ph}
\end{align}  
\end{theorem}
\begin{proof}
For any integer $n \in[1,N_T]$,
the desired subsequence convergence result (\ref{conv}) follows immediately 
from (\ref{Fstab2}),
as $(\buhd^{n},$ $\strhd^{n})$ are finite dimensional for fixed $\Vhzero \times S_h^0$. 
It also follows from (\ref{Fstab2}),
(\ref{conv}) and (\ref{Lip}) that $[\strh^{n}]_{-}$ vanishes on $\D$, 
so that
$\strh^{n}$ must be non-negative definite on $\D$. Hence 
on noting this,  (\ref{Lip}) and (\ref{conv}), 
we have the following subsequence convergence results
\begin{align}
\beta_{\delta}(\strh^{n}) \rightarrow \strh^{n} \quad \mbox{as} \quad
\delta \rightarrow 0_+
\qquad \mbox{and}
\qquad  
\beta_{\delta}(\strhd^{n}) \rightarrow \strh^{n} \quad \mbox{as} \quad
\delta \rightarrow 0_+\,.
\label{betacon}
\end{align}

It also follows from (\ref{Fstab2}), (\ref{betacon}) and as $[\Bd(\strhd^{n})]^{-1}
\Bd(\strhd^{n})=\I$ that the following subsequence result 
\begin{align}
[\Bd(\strhd^{n})]^{-1} \rightarrow [\strh^{n}]^{-1}\qquad \mbox{as} \quad
\delta \rightarrow 0_+
\label{posdef}
\end{align}
holds, and so $\strh^{n}$ is positive definite on $\D$. Therefore, we have 
from (\ref{conv}) and (\ref{eq:Gd}) that
\begin{align}
\Gd(\strhd^{n}) \rightarrow G(\strh^{n}) \qquad \mbox{as} \quad
\delta \rightarrow 0_+\,.
\label{Gcon}
\end{align}

Since $\buhd^{n-1},\,\buh^{n-1} \in C(\overline{\D})$, it follows from 
the $\Sh^0$ version of (\ref{eq:jump}), (\ref{eq:streams}) and (\ref{conv}) that
for $j=1,\ldots, N_E$ and for all $\bphi \in \Sh^0$
\begin{align}
\sum_{j=1}^{N_E} 
\int_{E_j}  \Abs{\buhd^{n-1}\cdot\bn} 
 \jump{\strhd^{n}}_{\to\buhd^{n-1}}:\bphi^{+\buhd^{n-1}}
&= -\sum_{k=1}^{N_K} \int_{\partial K_k} \brk{\buhd^{n-1} \cdot \bn_{K_k}} \strhd^n : 
\bphi^{+\buhd^{n-1}}
\nonumber \\
\rightarrow 
-\sum_{k=1}^{N_K} \int_{\partial K_k} \brk{\buh^{n-1} \cdot \bn_{K_k}} \strh^n : 
\bphi^{+\buh^{n-1}}
&=
\sum_{j=1}^{N_E} 
\int_{E_j}  \Abs{\buh^{n-1}\cdot\bn} 
 \jump{\strh^{n}}_{\to\buh^{n-1}}:\bphi^{+\buh^{n-1}}
\nonumber \\
& \hspace{1in}
\quad \mbox{as} \quad
\delta \rightarrow 0_+\,.
\label{Edgcon}
\end{align} 
Hence using (\ref{conv}), (\ref{betacon}) and (\ref{Edgcon}),
we can pass to the limit $\delta \rightarrow 0_+$ in (P$_{\delta,h}^{\dt}$), 
(\ref{eq:Pdh},b),
to show that $\{(\buh^{n},\strh^{n})\}_{n=1}^{N_T}
\in [\Vhzero \times \Sh^0]^{N_T}$
solves (P$_{h}^{\dt}$), (\ref{eq:Pdh-limit},b). 
Similarly, using (\ref{conv}), (\ref{betacon}), (\ref{posdef}) and (\ref{Gcon}), 
and noting (\ref{eq:free-energy-Pd}) and (\ref{unreg-energy}),
we can pass to the limit $\delta \rightarrow 0_+$ in (\ref{eq:estimate-Pdh})
to obtain the desired result (\ref{eq:estimate-Ph}). 
\end{proof}

\begin{remark}
\label{rk:complement}
Most numerical approximations of (P) suffer from instabilities
when ${\rm Wi}$ is relatively large, the so-called {\it high} Weissenberg
number problem (HWNP). 
This problem is still not fully understood.
Some reasons for these instabilities are 
discussed in
Boyaval {\it et al.},\cite{boyaval-lelievre-mangoubi-09}
e.g.\ poor numerical scheme or the lack of existence of a solution
to (P) itself. 
In addition in Boyaval {\it et al.},\cite{boyaval-lelievre-mangoubi-09} 
finite element approximations of (P) 
such as (P$^{\Delta t}_h$),
approximating the primitive
variables $(\bu,p,\strs)$, are compared
with   
finite element approximations of the 
log-formulation of (P), introduced in Fattal and Kupferman,\cite{fattal-05}
which is based on the variables $(\bu,p,\bpsi)$, where $\bpsi = \ln \strs$. 
The equivalent free energy estimate for this log-formulation
is based on testing the Navier-Stokes equation with $\bu$ as before,
but the log-form of the stress equation with $(\exp {\bpsi} - \I)$.  
Whereas the free energy estimate for (P) requires $\strs$ to be positive
definite, due to the testing with $\ln \strs$,
the free energy estimate for the log-formulation 
requires no such constraint.   
In  Boyaval {\it et al.}\cite{boyaval-lelievre-mangoubi-09} a constraint,
based on the initial data,
was required on the time step
in order to ensure that the approximation to $\strs$ remained 
positive definite for schemes such as (P$^{\Delta t}_h$) approximating (P);
whereas existence of a solution 
to finite element approximations
of the log-formulation, and  
satisfying a discrete log-form of the free energy estimate,
were shown for any choice of time step.  
It was suggested 
in Boyaval {\it et al.}\cite{boyaval-lelievre-mangoubi-09}
that this may be the reason why the
approximations of the log-formulation 
are reported to be more stable than those based
on (P).
However, Theorem~\ref{dconthm} above shows that there does exist (at least) one solution
to 
(P$^{\Delta t}_h$), 
which satisfies the free energy estimate (\ref{eq:estimate-Ph}), whatever the time step.
Of course, we do not have a uniqueness proof for (P$^{\Delta t}_h$).  
\end{remark}

\section{Regularized problems with stress diffusion and possibly the cut-off
$\BL$}
\label{sec:Palpha}

\subsection{Regularizations, (P$^{(L)}_\alpha$), of (P) with stress diffusion 
and possibly the cut-off $\BL$}

In this section, we consider 
the following
modified versions of (P) for given constants 
$\alpha \in \R_{>0}$ and 
$L \geq 2$:

\noindent
{\bf (P$_\alpha^{(L)}$)} Find 
$\buaorL : (t,\xx)\in[0,T)\times\D \mapsto \buaorL(t,\xx)\in\R^d$,  
$\paorL : (t,\xx)\in(0,T)\times\D \mapsto \paorL(t,\xx)\in\R$ 
and $\straorL : (t,\xx)\in[0,T)\times\D \mapsto\straorL(t,\xx)
\in \RS$ 
such that
\begin{subequations}
\begin{align}
\Re \brk{\pd{\buaorL}{t} + (\buaorL\cdot\grad)\buaorL} & =  -\grad \paorL 
+ (1-\e)\Delta\buaorL +  \frac{\e}{\Wi}\div\BorL(\straorL) 
\nonumber \\
& \hspace{0.55in} + \f 
\qquad \mbox{on } \D_T 
\,, 
\label{eq:aoldroyd-b-sigma}
\\
\div\buaorL & = 0 
\qquad \qquad \qquad \mbox{on } \D_T 
\,, 
\label{eq:aoldroyd-b-sigma1}
\\
\pd{\straorL}{t}+(\buaorL\cdot\grad)\BorL(\straorL) & = 
(\gbuaorL)\BorL(\straorL)+\BorL(\straorL)(\gbuaorL)^T
\nonumber
\\
&
\hspace{0.2in} 
-\frac{1}{\Wi}\brk{\straorL-\I} 
+ \alpha \Delta \straorL
\qquad 
\mbox{on } \D_T 
\,,
\label{eq:aoldroyd-b-sigma2}
\\
\buaorL(0,\xx) &= \bu^0(\xx) 
\label{eq:ainitial}\qquad \qquad \forall \xx \in \D\,,\\
\straorL(0,\xx) &= \strs^0(\xx) 
\qquad \qquad \forall \xx \in \D\,,
\label{eq:ainitial1}
\\
\label{eq:adirichlet}
\buaorL&=\bzero \qquad \qquad \qquad \text{on $ (0,T) \times \partial\D$}
\,,
\\
\label{eq:aneumann}
(\bn_{\partial \D} \cdot
\grad) \straorL &=\bzero \qquad 
\qquad \qquad 
\text{on $ (0,T) \times \partial\D$}
\,;
\end{align}
\end{subequations}
where $\bn_{\D}$ is normal to the boundary $\partial \D$.

Hence problem (P$_\alpha^{(L)}$) is the same as (P), but with 
the added diffusion term $\alpha \Delta \straorL$
for the stress equation (\ref{eq:aoldroyd-b-sigma2}),
and the associated Neumann boundary condition (\ref{eq:aneumann});
and in the case of (P$_\alpha^L)$ with certain terms in  
(\ref{eq:aoldroyd-b-sigma},c) involving $\straL$ replaced by
$\BL(\straL)$, recall (\ref{eq:Bd}).
Of course, it is naturally assumed in (P$_\alpha^{(L)}$) 
that $\straorL$ is positive definite
on $\D_T$ in order for $\BorL(\straorL)$ to be well defined.

We will also be interested in the corresponding regularization
(P$_{\alpha,\delta}^{(L)}$) of (P$_\alpha^{(L)}$) with solution
$(\bu_{\alpha,\delta}^{(L)}, p_{\alpha,\delta}^{(L)}, \strs_{\alpha,\delta}^{(L)})$;
where $\BorL(\cdot)$ in (\ref{eq:aoldroyd-b-sigma}--g) 
is replaced by $\BdorL(\cdot)$, and so that 
$\strs_{\alpha,\delta}^{(L)}$ is not required to be positive definite.

\subsection{Formal energy estimates for (P$^{(L)}_{\alpha,\delta}$)}

Let $\FdorL(\buadorL,\stradorL)$ denote the free energy of the solution 
$(\buadorL,$ $\padorL,\stradorL)$ 
to problem $(\rm P_{\alpha,\delta}^{(L)})$, where $\FdorL : \U \times \S 
\rightarrow \R$ is defined as 
\beq \label{eq:free-energy-PadL}
\FdorL(\bv,\bphi) := \frac{\Re}{2}\intd\|\bv\|^2 + \frac{\e}{2\Wi}
\intd \tr(\bphi-\GdorL(\bphi)-\I)  
\,.
\eeq
We have the following analogue of Proposition  \ref{prop:free-energy-Pd}.
\begin{proposition} \label{prop:free-energy-PadL}
Let $(\buadorL,\padorL,\stradorL)$ 
be a sufficiently smooth 
solution to problem $(\rm P_{\alpha,\delta}^{(L)})$.
Then the free energy $\FdorL(\buadorL,\stradorL)$  
satisfies for a.a.\ $t \in (0,T)$
\begin{align} 
\nonumber
&\deriv{}{t}\FdorL(\buadorL,\stradorL)
+(1-\e)\intd\|\gbuadorL\|^2
\\
& \hspace{0.5in}+ \displaystyle \frac{\e}{2 {\rm Wi}^2}\intd\tr(\BdorL(\stradorL)
+[\BdorL(\stradorL)]^{-1}
-2\I)
\nonumber
\\
& \hspace{1in}
+ \frac{\alpha\e\delta^2}{2{\rm Wi}}
\intd\|\grad \GdorL'(\stradorL)\|^2
\leq
\langle \f,\buadorL\rangle_{H^1_0(\D)}\,.
\label{eq:estimate-PadL}
\end{align}
\end{proposition}
\begin{proof} 
Multiplying the Navier-Stokes equation (\ref{eq:aoldroyd-b-sigma}) with $\buadorL$
and the stress equation 
(\ref{eq:aoldroyd-b-sigma2})
with $\frac{\e}{2\Wi}(\I-\GdorL'(\stradorL))$,
summing and integrating over $\D$ yields,
after 
using integrations by parts and the incompressibility property in the standard
way, 
that
\begin{align}
\nonumber
&\intd  \Brk{ \frac{\Re}{2} 
\pd{}{t}\|\buadorL\|^2 + (1-\e)\|\gbuadorL\|^2 }
\\ 
& \hspace{0.1in}+\frac{\e}{\Wi} \,\intd \biggl[
 \BdorL(\stradorL):\gbuadorL 
- \frac{\alpha}{2}\,\grad \stradorL :: \grad \GdorL'(\stradorL)
\biggr. 
\nonumber
\\ 
& \hspace{0.1in}  + \frac{1}{2} \brk{ \brk{\pd{}{t}\stradorL
+(\buadorL\cdot\grad)
\BdorL(\stradorL)}
 + \frac{1}{\Wi}\brk{\stradorL-\I} } :\brk{\I-\GdorL'(\stradorL)} 
\nonumber
\\
&\hspace{0.1in}  - \frac{1}{2} 
 \brk{ \brk{\gbuadorL}\BdorL(\stradorL) + \BdorL(\stradorL)\brk{\gbuadorL}^T } 
 :\brk{\I-\GdorL'(\stradorL)} \biggr]
\nonumber \\
&\hspace{3.3in}
= 
\langle \f,\buadorL\rangle_{H^1_0(\D)}\,.
\label{PdaLener}
\end{align}
Similarly to (\ref{scalar:Lipschitz}), we have that 
\begin{align}
-\grad \stradorL :: \grad \GdorL'(\stradorL)
\geq \delta^2 \| \grad \GdorL'(\stradorL)\|^2 \qquad a.e.\ \mbox{in } \D_T.
\label{PdaLener0}
\end{align}
Using (\ref{eq:deriv-tensor-g}),  
we have that
\begin{align}
&\pd{}{t}\stradorL : \brk{\I-\GdorL'(\stradorL)} 
= \pd{}{t} \tr\brk{\stradorL-\GdorL(\stradorL)} \,.
\label{PadLener1}
\end{align}
We will deal with the convection term differently to the
approach used in (\ref{Pdener1}),
as that cannot be mimicked at a discrete level using
continuous piecewise linear elements to approximate $\stradorL$.
Note that we cannot use $S_h^0$ with the desirable property
(\ref{eq:inclusion1}) to approximate $\stradorL$, 
as we now have the added diffusion term.      
Instead, as 
$\stradorL$ has been replaced by $\BdorL(\stradorL)\equiv \HdorL'
(\GdorL'(\stradorL))$, on recalling (\ref{eq:HBdL}),
in this convective term 
and as
$\buadorL \in \Uz$, we have that
\begin{align}
&\int_\D (\buadorL\cdot\grad)\BdorL(\stradorL):\brk{\I-\GdorL'(\stradorL)} 
\nonumber \\
& \hspace{1.8in} = \int_\D \BdorL(\stradorL): (\buadorL\cdot\grad)\GdorL'(\stradorL) 
\nonumber \\
& \hspace{1.8in}
=
\int_\D
(\buadorL\cdot\grad)\tr\brk{\HdorL(\GdorL'(\stradorL))} = 0,
\label{PadLeber1a} 
\end{align}
where we have noted the spatial counterpart of (\ref{eq:deriv-tensor-g}).  
Similarly to (\ref{Pdener2}) and (\ref{Pdener3}) we obtain that
\begin{align}
&\brk{\brk{\gbuadorL}\BdorL(\stradorL)+ 
\BdorL(\stradorL)\brk{\gbuadorL}^T
}:\brk{\I-\GdorL'(\stradorL)} 
\nonumber \\
&
\hspace{2.5in}
= 2\tr\brk{ \brk{\gbuadorL}\BdorL(\stradorL) } \,,
\label{PdaLener2} 
\end{align}
and once again the terms involving the left-hand side of (\ref{PdaLener2})
in (\ref{PdaLener})  
cancel with the term $\frac{\e}{\Wi}\BdorL(\stradorL):\gbuadorL$
in (\ref{PdaLener})  
arising from the Navier-Stokes equation.  
Finally, the treatment of the remaining term $\brk{\stradorL-\I}:
\brk{\I-\GdorL'(\stradorL)}$
follows similarly to (\ref{remterm}); and so we obtain the 
desired free energy inequality \eqref{eq:estimate-PadL}.
\end{proof}

The following Corollary follows from   
(\ref{eq:estimate-PadL}) on noting the proof of 
Corollary \ref{cor:free-energy-Pd}.

\begin{corollary} 
\label{cor:free-energy-PadL}
Let $(\buadorL,\padorL,\stradorL)$ 
be a sufficiently smooth solution to problem $(\rm P_{\alpha,\delta}^{(L)})$. 
Then it follows that
\begin{align}
\nonumber
&\sup_{t \in (0,T)}\FdorL(\buadorL(t,\cdot),\stradorL(t,\cdot))
+
\frac{\alpha\e\delta^2}{{2\rm Wi}} 
\,\int_{D_T}
\|\grad \GdorL'(\stradorL)\|^2
\\
& \qquad
+
\frac{1}{2}
\int_{\D_T} \!\!\Brk{ 
(1-\e)
\|\gbuadorL\|^2
+\frac{\e}{{\rm Wi}^2}\tr(\BdorL(\stradorL)+[\BdorL(\stradorL)]^{-1}-2\I)}
\nonumber
\\
& \hspace{1.3in}
\le 2\brk{ \FdorL(\bu^0,\strs^0) + \frac{1+C_P}{2(1-\e)} 
\Norm{\f}_{L^2(0,T;H^{-1}(\D))}^2 }\,.
\label{eq:free-energy-boundaL}
\end{align}
\end{corollary}

\section{Finite element approximation of (P$_{\alpha,\delta}^{(L)}$) and
(P$_{\alpha}^{(L)}$)}
\label{sec:Palphah}

\subsection{Finite element discretization}

\label{5.1}
We now introduce a conforming finite element discretization of
(P$_{\alpha,\delta}^{(L)}$), which satisfies a discrete analogue of 
(\ref{eq:estimate-PadL}).  
As noted in the proof of Proposition \ref{prop:free-energy-PadL} above, 
we cannot use $S_h^0$ with the desirable property
(\ref{eq:inclusion1}) to approximate $\stradorL$, 
as we now have the added diffusion term. 
In the following, we choose 
 \begin{subequations}
 \begin{align}
 \Uh^1 &:= \Uh^2  \subset \U \quad \mbox{or} \quad \Uh^{1,+}
 \subset \U\,,
 \label{Vhmini} \\
 {\rm Q}_h^1 &=\{q \in C(\overline{\D}) \,:\, q \mid_{K_k} \in \PP_1 \quad
 k=1,\ldots, N_K\} \subset {\rm Q}\,,  
 \label{Qh1}\\
 \Shone &=\{\bphi \in [C(\overline{\D})]^{d \times d}_S 
 \,:\, \bphi \mid_{K_k} \in [\PP_1]^{d \times d}_S \quad
 k=1,\ldots, N_K\} \subset \S\, 
 \label{Sh1}\\
\mbox{and} \qquad
\Vhone &= \BRK{\bv \in\Uh^1 \,:\, \int_{\D} q \, {\div\bv} = 0 \quad  
\forall q\in\mathrm{Q}_h^1} \,;
\label{Vh1}
\end{align}
\end{subequations}
where $\Uh^2$ is defined as in (\ref{Vh2}) and, on recalling the barycentric
coordinate notation used in (\ref{bvarsig}),
\begin{align}
  \Uh^{1,+}&:=\left\{\bv \in [C(\overline{\D})]^d\cap \U \,:\, \bv \mid_{K_k} \in \left[\PP_1
  \oplus \mbox{span} \prod_{i=0}^d \eta^k_i \right]^d \quad
 k=1,\ldots, N_K\right\}\label{Vh1p}\,.
\end{align}

The velocity-pressure choice, $\Uh^2 \times {\rm Q}_h^1$, is the lowest order
Taylor-Hood element. It    
satisfies (\ref{eq:LBB}) with $\Uh^0$ and ${\rm Q}_h^0$
replaced by $\Uh^2$ and
${\rm Q}_h^1$, respectively,
provided, in addition to $\{{\mathcal T}_h\}_{h>0}$ being a regular family of meshes,
that each simplex has at least one vertex in $\D$, see p177 in Girault and 
Raviart\cite{GiraultRaviart} in the case $d=2$ and Boffi\cite{Boffi} in the 
case $d=3$.      
Of course, this is a very mild restriction on $\{{\mathcal T}_h\}_{h>0}$. 
The velocity-pressure choice, $\Uh^{1,+} \times {\rm Q}_h^1$, is called the
mini-element. It    
satisfies (\ref{eq:LBB}) with $\Uh^0$ and ${\rm Q}_h^0$
replaced by $\Uh^{1,+}$ and
${\rm Q}_h^1$, respectively;
see Chapter II, Section 4.1 in Girault and 
Raviart\cite{GiraultRaviart} in the case $d=2$ 
and Section 4.2.4 in Ern and Guermond\cite{ern-guermond:2004} in the 
case $d=3$.  
Hence for both choices of $\Uh^1$, it follows that
for all $\bv \in \Uz$ there exists a sequence $\{\bv_h\}_{h>0}$, with
$\bv_h \in \Vhone$, such that
\begin{align}
\lim_{h \rightarrow 0_+} \|\bv-\bv_h\|_{H^1(\D)} =0\,.
\label{Vhconv} 
\end{align}  
 
We recall the well-known local inverse inequality
for ${\rm Q}^1_h$
\begin{align}
\|q\|_{L^{\infty}(K_k)} &\leq C\,|K_k|^{-1}\,\int_{K_k} |q| 
\textbf{}\qquad \forall q \in {\rm Q}^1_h, \qquad
k=1,\dots,N_K  
\nonumber
\\
\Rightarrow \qquad
\|\bchi\|_{L^{\infty}(K_k)} &\leq C\,|K_k|^{-1}\,\int_{K_k} \|\bchi\| \qquad \forall 
\bchi \in {\rm S}^1_h, \qquad
k=1,\dots,N_K. 
\label{inverse}
\end{align}
We recall a similar well-known local inverse inequality for $\Vhone$
\begin{align}
\|\grad \bv\|_{L^2(K_k)} \leq C\, h_k^{-1} \|\bv\|_{L^2(K_k)}
\qquad \forall 
\bv \in \Vhone, \qquad
k=1,\dots,N_K.  
\label{Vinverse}
\end{align} 

We introduce the interpolation operator $\pi_h : C(\overline{\D})
\rightarrow {\rm Q}^1_h$, and extended naturally to 
$\pi_h : [C(\overline{\D})]^{d \times d}_S
\rightarrow {\rm S}^1_h$,
such that for all $\eta \in C(\overline{\D})$
and $\bphi \in [C(\overline{\D})]^{d \times d}_S$ 
\begin{align}
\pi_h \eta( P_p) =\eta(P_p) 
\qquad
\mbox{and}
\qquad  
\pi_h \bphi( P_p) =\bphi(P_p) 
\qquad p=1,\dots, N_P,  
\label{eq:pi1h}
\end{align}  
where $\{P_p\}_{p=1}^{N_P}$ are the vertices of $\mathcal{T}_h$.
As $\bphi \in \Shone$ does not imply that  
$\GdorL'(\bphi) \in \Shone$, we have to test the finite element 
approximation of the (P$_{\alpha,\delta}^{(L)}$) version 
of (\ref{eq:aoldroyd-b-sigma2}) with 
$\I - \pi_h[\GdL'(\stradh^{(L,)n})] \in \Shone$, where  
$\stradh^{(L,)n}\in \Shone$ is our finite element approximation
to $\stradorL$ at time level $t_{n}$.
This approximation of the (P$_{\alpha,\delta}^{(L)}$) version 
of (\ref{eq:aoldroyd-b-sigma2})
has to be constructed to mimic the results 
(\ref{PdaLener0})--(\ref{PdaLener2}),
when tested with $\I - \pi_h[\GdorL'(\stradh^{(L,)n})] \in \Shone$.

In order to mimic (\ref{PdaLener0})
we shall assume from now on that the family of meshes, $\{\mathcal{T}_h\}_{h>0}$, 
for the polytope
$\D$ 
consists of non-obtuse simplices only,
i.e. all dihedral angles of any simplex in
$\mathcal{T}_h$
are less than or equal to
$\frac{\pi}{2}$. 
Of course, the construction of such a non-obtuse mesh in the case 
$d=3$ is not straightforward for a general polytope ${\cal D}$.
We then have the following result.

\begin{lemma} \label{lem:inf-bound}
Let $g \in C^{0,1}(\mathbb R)$ be monotonically increasing  
with Lipschitz constant $g_{\rm Lip}$.
As $\mathcal{T}_h$ consists of only non-obtuse simplices,
then we have for all $q \in {\rm Q}_h^1 $, $\bphi\in\Shone$ that
\begin{align} 
g_{\rm Lip} \,\grad\pi_h[g(q)] \cdot \grad q
&\ge 
\|\grad\pi_h[g(q)]\|^2 \quad 
\mbox{and} \quad
g_{\rm Lip} \,\grad\pi_h[g(\bphi)] :: \grad \bphi
\ge 
\|\grad\pi_h[g(\bphi)]\|^2 
\nonumber \\
&
\quad 
\hspace{1.5in}
\mbox{on } K_k, \quad
k=1, \dots, N_K. 
\label{eq:inf-bound}
\end{align}
\end{lemma}
\begin{proof}
Let $K_k$ have vertices $\{P^k_j\}_{j=0}^d$, and let $\eta^k_j(\xx)$ be the basis 
functions on $K_k$ associated with $\rm{Q}^1_h$ and $\Shone$, i.e.
$\eta^k_j \mid_{K_k} \in \mathbb{P}_1$ and $\eta^k_j(P^k_i)=\delta_{ij}$,
$i,j = 0,\dots,d$. As $K_k$ is non-obtuse it follows that
\begin{align}
\grad \eta^k_i \cdot \grad \eta^k_j \leq 0 \qquad \mbox{on }K_k, \qquad
i,j = 0, \dots , d, \ \mbox{ with }\ i \neq j\,.
\label{obtuse}
\end{align}  
We note  that
\begin{align}
&\sum_{j=0}^d \eta^k_j \equiv 1 \quad \mbox{on}\quad K_k
\quad \Rightarrow \quad \nonumber \\
& \hspace{1in} \|\grad \eta^k_i\|^2 = - \sum_{j=0,\ j\neq i}^d  
\grad \eta^k_i \cdot \grad \eta^k_j \quad \mbox{on}\quad K_k,
\quad i=0,\dots,d\,.
\label{sumto1}
\end{align}
Hence for $a_i,b_i \in \mathbb R$, $i=0, \dots d$, we have that 
\begin{align} \grad(\sum_{i=0}^d a_i\,\eta^k_i) \cdot  
\grad(\sum_{j=0}^d b_j\,\eta^k_j)
& = \sum_{i=0}^d \left[ a_i\,b_i\,\|\grad \eta^k_i\|^2
+ \sum_{j=0,\ j\neq i}^d  
a_i \,b_j \grad \eta^k_i \cdot \grad \eta^k_j \right]
\nonumber
\\
& =
-\sum_{i=0}^d \sum_{j=0,\ j\neq i}^d  
a_i \,(b_i-b_j)\, \grad \eta^k_i \cdot \grad \eta^k_j
\nonumber \\
& =
-\sum_{i=0}^d \sum_{j> i}^d  
(a_i-a_j) \,(b_i-b_j)\, \grad \eta^k_i \cdot \grad \eta^k_j\,.
\label{aibj}
\end{align}
Similarly for $\bolda_i,\boldb_i \in \RS$, 
$i=0, \dots, d$, we have that 
\begin{align} \grad(\sum_{i=0}^d \bolda_i\,\eta^k_i) ::  
\grad(\sum_{j=0}^d \boldb_j\,\eta^k_j)
&=
-\sum_{i=0}^d \sum_{j> i}^d  
\left[(\bolda_i-\bolda_j): (\boldb_i-\boldb_j)
\right]
\, \grad \eta^k_i \cdot \grad \eta^k_j\,.
\label{baibbj}
\end{align}
The desired result (\ref{eq:inf-bound}) 
then follows immediately from (\ref{aibj}), (\ref{baibbj}), (\ref{obtuse})
and our assumptions on $g$.
\end{proof}

In order to mimic (\ref{PadLener1}) and (\ref{PdaLener2}), we need to use 
numerical integration (vertex sampling).
We note the following results.
As the basis functions associated with ${\rm Q}^1_h$ and $\Shone$
are nonnegative and sum to unity everywhere, we have 
for $k=1,\ldots,N_K$ that  
\begin{align}
\|[\pi_h \bphi](\xx)\|^2 &\leq (\pi_h[\,\|\bphi\|^2\,])(\xx)
\quad \forall \xx \in K_k, \quad \forall 
\bphi \in [C(\overline{K_k})]^{d \times d}\,.
\label{interpinf} 
\end{align}
In addition, we have for $k=1,\ldots,N_K$ that
\begin{align}
\int_{K_k} \|\bchi\|^2 \leq \int_{K_k} \pi_h[\,\|\bchi\|^2] 
\leq C \int_{K_k} \|\bchi\|^2 \qquad  \forall \bchi \in \Shone.
\label{eqnorm}
\end{align}
The first inequality in (\ref{eqnorm}) follows immediately from
(\ref{interpinf}), and the second from applying (\ref{inverse})
and a Cauchy--Schwarz inequality.

In order to mimic (\ref{PadLeber1a}), we have to carefully construct our finite
element approximation of the convective term  
in the (P$_{\alpha,\delta}^{(L)}$) version of (\ref{eq:aoldroyd-b-sigma2}) 
Our construction is a non-trivial extension of an approach that has been used 
in the finite element approximation of fourth-order degenerate nonlinear parabolic
equations, such as the thin film equation; see e.g.\ 
Gr\"{u}n and Rumpf\cite{GrunRumpf}
and Barrett and N\"{u}rnberg.\cite{surf2d} 
Let $\{\be_i\}_{i=1}^d$ be the orthonormal vectors 
in $\R^d$, 
such that the $j^{th}$ component of $\be_i$ is $\delta_{ij}$, $i,j=1,\dots,d$.
Let $\widehat{K}$   
be the standard open reference simplex in $\R^d$
with vertices $\{\widehat{P}_i\}_{i=0}^d$, 
where $\widehat{P}_0$ is the origin and 
$\widehat{P}_i = \be_i$, $i=1,\dots,d$.
Given a simplex $K_k\in\mathcal{T}_h$ with vertices $\{P^k_i\}_{i=0}^d$,
then there exists 
a non-singular matrix $B_k$ 
such that the linear mapping
\beq \label{eq:mapping-RK}
\mathcal{B}_k : \widehat{\xx}\in\R^d \mapsto P_0^k + B_k\widehat{\xx} \in \R^d 
\eeq
maps vertex $\widehat{P}_i$  
to vertex $P^k_i$, $i=0,\dots,d$. 
Hence $\mathcal{B}_k$ maps $\widehat{K}$ 
to $K_k$.
For all $\eta\in  
{\rm Q}^1_h$
and $K_k \in \mathcal{T}_h$, we define
\begin{align}
\label{eq:mapping-derivative}
\widehat{\eta}(\widehat{\xx}) := \eta\brk{\mathcal{B}_k(\widehat{\xx})} \quad 
\forall \widehat{\xx}\in\widehat{K}
\quad \Rightarrow
\quad
\grad \eta(
\mathcal{B}_k(\widehat \xx)
) = 
(B_k^T)^{-1}  
\widehat{\grad} \widehat{\eta}
(\widehat \xx)
\quad \forall \widehat \xx\in \widehat{K}\,,
\end{align}
where for all  $\widehat \xx\in \widehat{K}$ 
\begin{align}
\label{diff}
[\widehat{\grad} \widehat{\eta}
(\widehat \xx)]_j = \frac{\partial}{\partial \widehat x_j}\widehat{\eta}
(\widehat \xx) = \widehat \eta(\widehat{P}_j)-\widehat \eta (\widehat{P}_0)
= \eta({P}_j^k)-\eta ({P}_0^k)\qquad j=1,\dots,d.
\end{align}
Such notation is easily extended to $\bphi \in \S^1_h$.

Given $\bphi \in \Shone$ and $K_k \in \mathcal{T}_h$, then first, for  
$j=1, \dots, d$, we find  
$\widehat \Lambda_{\delta,j}^{(L)}(\widehat \bphi) \in \RS$,
which depends continuously on $\bphi$,
such that
\begin{align}
\widehat \Lambda_{\delta,j}^{(L)}(\widehat \bphi) :  
\frac{\partial}{\partial \widehat x_j}
\widehat{\pi}_h[\GdorL'(\widehat \bphi)] =
\frac{\partial}{\partial \widehat x_j}
\widehat{\pi}_h[\tr(\HdorL(\GdorL'(\widehat \bphi)))]
\qquad \mbox{on } \widehat K,
\label{hLambdaj}
\end{align}
where $(\widehat \pi_h \widehat \eta)(\widehat \xx) \equiv
(\pi_h \eta)(\mathcal{B}_k \widehat \xx)$ for all
$\widehat \xx \in \widehat K$ and $\eta \in C(\overline{K_k})$.
This leads to a unique choice of  $\widehat \Lambda_{\delta,j}^{(L)}(\widehat \bphi)$.
For the construction of $\widehat \Lambda_{\delta,j}^{(L)}(\widehat \bphi)$
in the simpler scalar case $(d=1)$, see p329 in Barrett and N\"{u}rnberg.\cite{surf2d} 
To construct $\widehat \Lambda_{\delta,j}^{(L)}(\widehat \bphi)$
satisfying (\ref{hLambdaj}),
we note the following. 
We have from (\ref{eq:HdL}), (\ref{eq:HBdL}) 
and (\ref{gconcave2}) that
\begin{align}
&\BdorL(\bphi(P_j^k)) :
(\GdorL'(\bphi(P_j^k)) -
\GdorL'(\bphi(P_0^k)))
\nonumber \\
& \hspace{1.3in}\leq 
\tr(\HdorL(\GdorL'(\bphi(P_j^k)) -
\HdorL(\GdorL'(\bphi(P_0^k)))
\nonumber
\\
& \hspace{1.3in}\leq
\BdorL(\bphi(P_0^k)) :
(\GdorL'(\bphi(P_j^k)) -
\GdorL'(\bphi(P_0^k))).
\label{hLambdajineq}
\end{align}
Next we note from
(\ref{eq:HdL}), (\ref{eq:HBdL}), (\ref{matrix:Lipschitz}) 
and (\ref{ipmat})
that
\begin{align}
&-(\BdL(\bphi(P_j^k))-
\BdL(\bphi(P_0^k))):
(\GdL'(\bphi(P_j^k)) -
\GdL'(\bphi(P_0^k)))
\nonumber \\
& \hspace{2in}
\geq
L^{-2} \,\|\BdL(\bphi(P_j^k))-
\BdL(\bphi(P_0^k))\|^2\,;
\label{betajtr}
\end{align}
and so the left-hand side is zero if and only if 
$\BdL(\bphi(P_j^k))=\BdL(\bphi(P_0^k))$.
Similarly, we see from 
(\ref{eq:HdL}), (\ref{eq:HBdL}) and the proof of (\ref{matrix:Lipschitz});
that is, (\ref{Entproof2}); that
\begin{align}
&-(\Bd(\bphi(P_j^k))-
\Bd(\bphi(P_0^k))):
(\Gd'(\bphi(P_j^k)) -
\Gd'(\bphi(P_0^k))) \geq 0
\label{betajtrnoL} 
\end{align}
with equality if and only if  
$\Bd(\bphi(P_j^k))=\Bd(\bphi(P_0^k))$.
Hence, on noting (\ref{diff}), 
(\ref{hLambdajineq}), (\ref{betajtr}), 
(\ref{betajtrnoL})
and (\ref{ipmat}), we have that
\begin{subequations}
\begin{align}
\widehat \Lambda_{\delta,j}^{(L)}(\widehat \bphi)
&:= (1-\lambda_{\delta,j}^{(L)})\BdorL(\bphi(P_j^k))
+ \lambda_{\delta,j}^{(L)}\BdorL(\bphi(P_0^k))
\nonumber \\
&\hspace{0.1in} \mbox{if} \quad  (\BdorL(\bphi(P_j^k))-
\BdorL(\bphi(P_0^k))):
(\GdorL'(\bphi(P_j^k)) -
\GdorL'(\bphi(P_0^k))) \neq 0\,,
\label{hLambdajdef}   
\\
\widehat \Lambda_{\delta,j}^{(L)}(\widehat \bphi)
&:= \BdorL(\bphi(P_j^k))
= \BdorL(\bphi(P_0^k))
\nonumber \\
& \hspace{0.1in} \mbox{if} \quad  (\BdorL(\bphi(P_j^k))-
\BdorL(\bphi(P_0^k))):
(\GdorL'(\bphi(P_j^k)) -
\GdorL'(\bphi(P_0^k))) = 0
\label{hLambdajdefb}   
\end{align}
\end{subequations}
satisfies (\ref{hLambdaj}) for $j=1,\dots,d$;
where $\lambda_{\delta,j}^{(L)} \in [0,1]$ is defined as  
\begin{align*}
\lambda_{\delta,j}^{(L)} &:= \Bigl[\tr(\HdorL(\GdorL'(\bphi(P_j^k)) -
\HdorL(\GdorL'(\bphi(P_0^k)))\Bigr.\nonumber \\
& \qquad \frac {\Bigl. \qquad \qquad \qquad -\BdorL(\bphi(P_j^k)):
(\GdorL'(\bphi(P_j^k)) -\GdorL'(\bphi(P_0^k)))\Bigr]}
{(\BdorL(\bphi(P_0^k))-
\BdorL(\bphi(P_j^k))):
(\GdorL'(\bphi(P_j^k)) -
\GdorL'(\bphi(P_0^k)))}\,.
\end{align*}
Furthermore, $\widehat \Lambda_{\delta,j}^{(L)}(\widehat \bphi)
\in \RS$, $j=1,\dots,d$, depends continuously on $\bphi \mid_{K_k}$.

Therefore given $\bphi \in \Shone$, we introduce, for $m,p =1, \dots, d$,
\begin{align}
\Lambda_{\delta,m,p}^{(L)}(\bphi) 
&= \sum_{j=1}^d
[(B_k^T)^{-1}]_{mj}\, \widehat \Lambda_{\delta,j}^{(L)}(\widehat \bphi)
\,[B_k^T]_{jp} \in \RS
\qquad \mbox{on } K_k,\nonumber \\
& \hspace{2.5in} \qquad k=1,\dots,N_K.  
\label{Lambdampdef}
\end{align}
It follows from 
(\ref{Lambdampdef}),
(\ref{hLambdaj}) and (\ref{eq:mapping-derivative}) that
\begin{align}
\Lambda_{\delta,m,p}^{(L)}(\bphi) \approx \BdorL(\bphi) \,\delta_{mp}
\qquad  m,p=1,\dots,d\,;
\label{Lambdampap}
\end{align}
and for $m=1,\dots,d$ 
\begin{align}
&\sum_{p=1}^d 
\Lambda_{\delta,m,p}^{(L)}(\bphi) : 
\frac{\partial}{\partial x_p}\pi_h[\GdorL'(\bphi)]
=\frac{\partial}{\partial x_m}\pi_h [\tr(\HdorL(\GdorL'(\bphi)))]
\qquad \mbox{on } K_k, 
\nonumber \\
\label{Lambdaj}
&\hspace{3.2in}
\qquad k=1,\dots,N_K.  
\end{align}
For a more precise version of (\ref{Lambdampap}), see Lemma \ref{lemMXitt}
below. 
Finally, as the partitioning ${\mathcal T}_h$ consists of regular simplices, we have that
\begin{align}
\|(B_k^T)^{-1}\|\,\|B_k^T\| \leq C, \qquad k=1,\dots,N_K\,.
\label{Breg}
\end{align}
Hence, it follows  from (\ref{Lambdampdef}), (\ref{Breg}), (\ref{hLambdajdef},b) 
and (\ref{BdLphi}) that
\begin{align}
\|\Lambda_{\delta,m,p}^{L}(\bphi)\|_{L^{\infty}(\D)} \leq C\,L \qquad
\forall \bphi \in \Shone.
\label{LdmpLinf}
\end{align}

\subsection{A free energy preserving approximation, (P$^{(L,)\Delta t}_{\alpha,\delta,h}$),
of (P$^{(L)}_{\alpha,\delta}$)}

In addition to the assumptions on the finite element discretization stated
in subsection \ref{5.1}, and our definition of $\Delta t$ in subsection \ref{FEd},    
we shall assume that
there exists a $C \in {\mathbb R}_{>0}$ such that
\begin{align}
\Delta t_n \leq C\, \Delta t_{n-1}, \qquad n=2, \dots, N,\qquad
\mbox{as} \quad \Delta t \rightarrow 0_+.
\label{Deltatqu}
\end{align}
With $\Delta t_1$ and $C$ as above, let $\Delta t_0 \in {\mathbb R}_{>0}$ be such that
$\Delta t_1 \leq C \Delta t_0$.
Given initial data satisfying (\ref{freg}), we choose
$\buh^{0} \in \Vhone$ and $\strh^{0} \in \Shone$
throughout the rest of this paper  
such that
\begin{subequations}
\begin{alignat}{2}
\intd \left[ \buh^{0}  \cdot \bv  + \Delta t_0
\grad  \buh^{0} : \grad \bv \right]
&=   \intd \bu^0 \cdot \bv \qquad
&&\forall \bv \in \Vhone\,,
\label{proju0} \\
\intd \left[ \pi_h[\strh^{0} : \bchi]  + \Delta t_0
\grad  \strh^{0} :: \grad \bchi \right]
&=
\intd \strs^{0} : \bchi  
&& \forall \bchi \in \Shone\,.
\label{projp0}
\end{alignat}
\end{subequations}
It follows from (\ref{proju0},b), (\ref{eqnorm}) and (\ref{freg}) that
\begin{align}
&\int_{\D} \left[\, \|\buh^{0}\|^2 + \|\strh^{0}\|^2 +
\Delta t_0 \, \left[\|\grad \buh^{0}\|^2 + \|\grad \strh^0\|^2
\right]\,\right] 
\leq C\,.
 \label{idatah}
\end{align}
In addition, we note the following result.

\begin{lemma}
\label{lem:idatahspd}
For $p=1,\ldots, N_P$ we have that
\begin{align}
&\sigma_{\rm min}^0\, \|\bxi\|^2 \leq {\bxi}^T \strh^0(P_p) \,{\bxi} 
\leq \sigma_{\rm max}^0\, \|\bxi\|^2
\quad \forall \bxi \in {\mathbb R}^d\,.
\label{idatahspd}
\end{align}
\end{lemma}
\begin{proof}
It follows from (\ref{projp0}) that 
\begin{align}
\intd \left[ \pi_h[(\strh^{0}-\strs^0_{\rm min}\I) : \bchi]  + \Delta t_0
\grad  (\strh^{0}-\strs^0_{\rm min}\I) :: \grad \bchi \right]
&=
\intd (\strs^{0}-\strs^0_{\rm min}\I) : \bchi  
\nonumber \\
& \hspace{0.3in}
\qquad \forall \bchi \in \Shone\,.
\label{projp0c}
\end{align}
Choosing $\bchi = \bxi\, \bxi^T \eta$, with $\eta \in {\rm Q}^1_h$ yields that $z_h := 
\bxi^T (\strh^{0}-\strs^0_{\rm min}\I) \bxi \in  {\rm Q}^1_h$ satisfies
\begin{align}
\intd \left[ \pi_h[z_h \,\eta]  + \Delta t_0
\grad z_h \cdot \grad \eta \right]
=
\intd z \, \eta  
\qquad \forall \eta \in {\rm Q}^h_1\,,
\label{projp0z}
\end{align}  
where $z := 
\bxi^T (\strs^{0}-\strs^0_{\rm min}\I) \bxi \in L^\infty(\D)$ 
and is non-negative on recalling (\ref{freg}).

Choosing $\eta=\pi_h[z_h]_{-}\in {\rm Q}^1_h$, it follows, on noting
the  ${\rm Q}^1_h$ version of (\ref{eqnorm}) and
(\ref{eq:inf-bound}) with $g(\cdot) = [\,\cdot\,]_{-}$, that
\begin{align}
\intd \left [\pi_h [z_h]_{-}]^2 + \Delta t_0 \,\|\grad \pi_h[z_h]_{-} \|^2
\right] & \leq \intd \left[ \pi_h \left[[z_h]_{-}^2 \right] + \Delta t_0
\grad z_h \cdot \grad \pi_h [z_h]_{-} \right]
\nonumber \\
&=
\intd z \,  \pi_h [z_h]_{-} \leq 0 \,.
\end{align} 
Hence $\pi_h[z_h]_{-} \equiv 0$ and so the first inequality in (\ref{idatahspd}) holds.
Repeating the above with $\strs^0_{\rm min}$ and $[\,\cdot\,]_{-}$ replaced by
$\strs^0_{\rm max}$ and $[\,\cdot\,]_{+}$, respectively, yields the second inequality 
in (\ref{idatahspd}).   
\end{proof}

Our approximation $(\mathrm{P}_{\alpha,\delta,h}^{(L,)\dt})$ 
of $(\mathrm{P}_{\alpha,\delta}^{(L)})$
is then:

({\bf P}$_{\alpha,\delta,h}^{(L,)\Delta t}$)
Setting 
$(\buhdLa^{(L),0},\strhdLa^{(L),0})
= (\buh^{0},\strh^{0})
\in\Vhone\times\Shone$, 
then for $n = 1, \ldots, N_T$ 
find $(\buhdLa^{(L,)n},$ $\strhdLa^{(L,)n})\in\Vhone\times\Shone$ 
such that for any test functions 
$(\bv,\bphi)\in\Vhone\times\Shone$
\begin{subequations}
\begin{align} 
\nonumber 
&\int_\D \biggl[ \Re\left(\frac{\buhdLa^{(L,)n}-\buhdLa^{(L,)n-1}}{\dt_{n}}\right)
\cdot \bv 
\\
\nonumber  
& \quad  
+ \frac{\Re}{2}\Brk{ \left( (\buhdLa^{(L,)n-1}\cdot\nabla)\buhdLa^{(L,)n}\right) 
 \cdot \bv - 
 \buhdLa^{(L,)n} \cdot \left( (\buhdLa^{(L,)n-1}\cdot\nabla)\bv \right)}
\\
& \quad
 + (1-\e) \gbuhdLa^{(L,)n}:\grad\bv + \frac{\e}{\Wi} \pi_h[\BdorL(\strhdLa^{(L,)n})] : 
 \grad\bv 
\biggr] = 
\langle \f^n,\bv\rangle_{H^1_0(\D)}\,,
\label{eq:PdLaha}
\\ 
\nonumber
&\int_\D \pi_h \left[\left(\frac{\strhdLa^{(L,)n}-\strhdLa^{(L,)n-1}}{\dt_{n}}\right) 
: \bphi + \frac{1}{\Wi}\left(\strhdLa^{(L,)n}-\I\right): \bphi \right]
\\ \nonumber
&\quad +  \int_D \left[ \alpha \grad\strhdLa^{(L,)n}::\grad\bphi
- 2
\gbuhdLa^{(L,)n}: \pi_h[\bphi\,\BdorL(\strhdLa^{(L,)n})\right]\\
& \quad 
- \int_\D \sum_{m=1}^d \sum_{p=1}^d 
[\buhdLa^{(L,)n-1}]_m \,\Lambda^{(L)}_{\delta,m,p}(\strhdLa^{(L,)n})
: \frac{\partial \bphi}{\partial \xx_p}
= 0 \,.
\label{eq:PdLahb}
\end{align}
\end{subequations}

In deriving (P$^{(L,)\Delta t}_{\alpha,\delta,h}$), we have noted
 (\ref{conv0c}),
 (\ref{eq:symmetric-tr}) and (\ref{Lambdampdef}).

Before proving existence of a solution to (P$^{(L,)\Delta t}_{\alpha,\delta,h}$),
we first derive a discrete analogue of the energy estimate (\ref{eq:estimate-PadL})
for (P$^{(L)}_{\alpha,\delta}$).

\subsection{Energy estimate}

On setting
\begin{align}
 \FdorLh
 (\bv,\bphi) &:= \frac{\Re}{2}\intd\|\bv\|^2 
 + \frac{\e}{2\Wi}\intd
 \pi_h[
 \tr\brk{\bphi-\GdorL(\bphi)-\I}]\nonumber \\
& \hspace{2in}
\qquad \forall (\bv,\bphi) \in\Vhone \times \Shone \,,
 \label{eq:free-energy-PdLah}
\end{align}
we have the following discrete analogue of Proposition \ref{prop:free-energy-PadL}.
\begin{proposition} \label{prop:free-energy-PdLah}
For $n= 1, \ldots, N_T$, a solution $\brk{\buhdLa^{(L,)n},\strhdLa^{(L,)n}}
\in \Vhone\times\Shone$ to (\ref{eq:PdLaha},b), if it exists, satisfies 
\begin{align}
&\frac{\FdorLh(\buhdLa^{(L,)n},\strhdLa^{(L,)n})-\FdorLh(\buhdLa^{(L,)n-1},
\strhdLa^{(L,)n-1})}{\dt_{n}} 
+ \frac{ {\rm Re}}{2\dt_{n}}\intd\|\buhdLa^{(L,)n}-\buhdLa^{(L,)n-1}\|^2
\nonumber
\\
& \hspace{0.1in} 
+(1-\e)\intd\|\gbuhdLa^{(L,)n}\|^2
+\frac{\e}{2{\rm Wi}^2}\intd
\pi_h[\tr(\BdorL(\strhdLa^{(L,)n})+[\BdorL(\strhdLa^{(L,)n})]^{-1}-2\I)]
\nonumber \\
& \hspace{0.1in} 
+ \frac{\alpha\e\delta^2}{2{\rm Wi}}
\intd \|\grad\pi_h[\GdorL'(\strhdLa^{(L,)n})]\|^2
\nonumber \\
& \hspace{1in}
\le
\langle \f^n,\buhdLa^{(L,)n}\rangle_{H^1_0(\D)}
\nonumber \\
& \hspace{1in}
\le 
\frac{1}{2}(1-\e)\intd\|\gbuhdLa^{(L,)n}\|^2
+ \frac{1+C_P}{2(1-\e)} \|\f^{n}\|_{H^{-1}(\D)}^2
\,.
 \label{eq:estimate-PdLah}
\end{align}
\end{proposition}
\begin{proof}
The proof is similar to that of  Proposition~\ref{prop:free-energy-Pdh},
we choose as test functions $\bv=\buhdLa^{(L,)n}\in\Vhone$ and 
$\bphi=\frac{\e}{2\Wi}\brk{\I-\pi_h[\GdorL'(\strhdLa^{(L,)n})]}\in\Shone$ 
in (\ref{eq:PdLaha},b), 
and obtain, on noting (\ref{elemident}), (\ref{eq:inverse-Gd},d,e),
(\ref{eq:inf-bound}) with $g=-G^{(L)'}_{\delta}$ having Lipschitz
constant $\delta^{-2}$, (\ref{Lambdaj}) and (\ref{eq:free-energy-PdLah}) 
that
\begin{align}
& \nonumber
\langle \f^n,\buhdLa^{(L,)n}\rangle_{H^1_0(\D)}
\\
& \qquad \geq
\frac{\FdorLh(\buhdLa^{(L,)n},\strhdLa^{(L,)n})-\FdorLh(\buhdLa^{(L,)n-1},
\strhdLa^{(L,)n-1})}
{\dt_{n}} 
+(1-\e)\intd\|\gbuhdLa^{(L,)n}\|^2
\nonumber \\
& \qquad \qquad 
+ \frac{ {\rm Re}}{2\dt_{n}}\intd\|\buhdLa^{(L,)n}-\buhdLa^{(L,)n-1}\|^2
+ \frac{\alpha\e\delta^2}{2{\rm Wi}}
\intd \|\grad\pi_h[\GdorL'(\strhdLa^{(L,)n})]\|^2 
\nonumber \\
& \qquad \qquad 
+ \frac{\e}{2{\Wi}^2}  \intd
\pi_h[ 
\tr(\BdorL(\strhdLa^{(L,)n})+[\BdorL(\strhdLa^{(L,)n})]^{-1}-2\I)]
\nonumber
\\
& \qquad \qquad
+ \intd \buhdLa^{(L,)n-1} \cdot \grad \pi_h[\tr(\HdorL(\GdorL'(\strhdLa^{(L,)n})))]
\,.
\label{eq:free-energy-PdhLa-demo1}
\end{align}
The first desired inequality in (\ref{eq:estimate-PdLah})
follows immediately from (\ref{eq:free-energy-PdhLa-demo1}) 
on noting (\ref{Vh1}), (\ref{Vh}), (\ref{spaces}) and that $\pi_h : C(\overline{\D}) 
\rightarrow {\rm Q}^1_h$.
The second inequality in (\ref{eq:estimate-PdLah}) follows immediately from (\ref{fbound})
with $\nu^2 = (1-\e)/(1+C_P)$. 
\end{proof}

\subsection{Existence of discrete solutions}

\begin{proposition} 
\label{prop:existence-PdLah}
Given $(\buhdLa^{(L,)n-1},\strhdLa^{(L,)n-1}) \in \Vhone \times \Sh^1$
and for any time step $\dt_{n} > 0$,
then there exists at least 
one solution $\brk{\buhdLa^{(L,)n},\strhdLa^{(L,)n}} \in \Vhone\times\Sh^1$ 
to 
(\ref{eq:PdLaha},b).
\end{proposition}
\begin{proof}
The proof is similar to that of Proposition \ref{prop:existence-Pdh}.
We introduce the following inner product on 
the Hilbert space $\Vhone\times\Sh^1$
$$ \Scal{ (\bw,\bpsi) }{ (\bv,\bphi) }_\D^h = 
\intd \Brk{ \bw\cdot\bv + \pi_h[\bpsi:\bphi] } 
\qquad \forall (\bw,\bpsi),(\bv,\bphi) \in \Vhone\times\Sh^1\,.
$$
Given $(\buhdLa^{(L,)n-1},\strhdLa^{(L,)n-1})\in\Vhone\times\Sh^1$, let
$\mathcal{F}^h : \Vhone\times\Sh^1 \rightarrow \Vhone\times\Sh^1$ be such that for any
$(\bw,\bpsi) \in \Vhone\times\Sh^1$
\begin{align} 
\nonumber
&\Scal{\mathcal{F}^h(\bw,\bpsi)}{(\bv,\bphi)}_\D^h 
\\
&\qquad := \int_\D \Biggl[\Re \left(\frac{\bw-\buhdLa^{(L,)n-1}}
{\dt_{n}}\right) \cdot \bv
+ (1-\e) \grad\bw :\grad\bv + \frac{\e}{\Wi} \pi_h[\BdorL(\bpsi)] 
:\grad\bv   
\nonumber \\   
& \qquad \qquad     
  + \frac{\Re}{2} \left[\left((\buhdLa^{(L,)n-1}\cdot\nabla)\bw \right) \cdot \bv 
  - \bw \cdot \left((\buhdLa^{(L,)n-1}\cdot\nabla)\bv \right) \right]  
+ \alpha \grad \bpsi :: \grad \bphi
\nonumber
\\
& \qquad  \qquad  
- 2  \grad\bw : \pi_h[\bphi\,\BdorL(\bpsi)]
\nonumber
+ \pi_h\left[\left(\frac{\bpsi-\strhdLa^{(L,)n-1}}{\dt_{n}}\right):\bphi 
 + \frac{1}{\Wi}\left(\bpsi-\I\right): \bphi \right]
 \Biggr] 
\nonumber
\\
& \qquad \qquad 
- \intd \sum_{m=1}^d \sum_{p=1}^d 
[\buhdLa^{(L,)n-1}]_m \Lambda^{(L)}_{\delta,m,p}(\bpsi) :
\frac{\partial \bphi}{\partial \xx_p} 
- \langle \f^n, \bv \rangle_{H^1_0(\D)}
\nonumber \\
& \hspace{3in}
\quad \forall (\bv,\bphi) \in \Vhone \times \Sh^1 
\,.
\label{eq:mapping1}
\end{align}
A solution 
$(\buhdLa^{(L,)n},\strhdLa^{(L,)n})$
to (\ref{eq:Pdh},b), if it exists,
corresponds to a zero of $\mathcal{F}^h$. 
On recalling 
(\ref{Lambdampdef})
and (\ref{hLambdajdef},b), 
it is easily deduced that the mapping $\mathcal{F}^h$ is continuous.

For any $(\bw,\bpsi) \in \Vhone\times\Sh^1$, on choosing
$(\bv,\bphi) = \brk{\bw,\frac{\e}{2\Wi}\brk{\I-\pi_h[\GdorL'(\bpsi)]}}$,
we obtain analogously to (\ref{eq:estimate-PdLah}) that
\begin{align} \nonumber
&\Scal{\mathcal{F}^h(\bw,\bpsi)}
{\brk{\bw,\frac{\e}{2\Wi}\brk{\I-\pi_h[\GdorL'(\bpsi)]}}}_\D^h
\\
& \qquad \quad \ge \frac{\FdorLh(\bw,\bpsi)-\FdorLh(\buhdLa^{(L,)n-1},
\strhdLa^{(L,)n-1})}
{\dt_{n}} 
+ \frac{\Re}{2\dt_{n}}\intd\|\bw-\buhdLa^{(L,)n-1}\|^2
\nonumber
\\
& \qquad\qquad \qquad
+ \frac{1-\e}{2} 
\intd\|\grad\bw\|^2
+\frac{\e}{2\Wi^2}\intd\pi_h[\tr(\BdorL(\bpsi)+[\BdorL(\bpsi)]^{-1}-2\I)] 
\nonumber \\
& \qquad \qquad \qquad 
+\frac{\alpha \e \delta^2}{2 \Wi} \intd \|\grad \pi_h[\GdorL'(\bpsi)]\|^2
-\frac{1+C_P}{2(1-\e)}\Norm{\f^{n}}_{H^{-1}(\D)}^2 \,.
 \label{eq:inequality1aL}
\end{align}
 
Let
$$ \Norm{(\bv,\bphi)}_D^h := 
\left[((\bv,\bphi),(\bv,\bphi))_D^h\right]^{\frac{1}{2}} =
\brk{\intd\Brk{\|\bv\|^2+\pi_h[\,\|\bphi\|^2\,]} }^\frac12 \,. $$
If for any $\gamma \in \R_{>0}$, 
the continuous mapping $\mathcal{F}^h$ 
has no zero $(\buhdLa^{(L,)n},\strhdLa^{(L,)n})$,  
which lies in the ball
$$ \mathcal{B}_\gamma^h := \BRK{ (\bv,\bphi) \in \Vhone\times\Sh^1 \,: \, 
\Norm{(\bv,\bphi)}_\D^h\le\gamma } \,;$$
then for such $\gamma$, we can define the continuous mapping $\mathcal{G}_\gamma^h
: \mathcal{B}_\gamma^h \rightarrow  \mathcal{B}_\gamma^h$
such that for all $(\bv,\bphi) \in \mathcal{B}_\gamma^h$
$$ \mathcal{G}_\gamma^h(\bv,\bphi) := -\gamma 
\frac{\mathcal{F}^h(\bv,\bphi)}{\Norm{\mathcal{F}^h(\bv,\bphi)}_\D^h} \,. $$
By the Brouwer fixed point theorem, $\mathcal{G}_\gamma^h$ has at least 
one fixed point $(\bw_\gamma,\bpsi_\gamma)$ in $\mathcal{B}_\gamma^h$. 
Hence it satisfies
\begin{equation}\label{eq:fixed-pointaL}
\Norm{(\bw_\gamma,\bpsi_\gamma)}_\D^h=
 \Norm{\mathcal{G}_\gamma^h(\bw_\gamma,\bpsi_\gamma)}_\D^h=\gamma.
\end{equation}

On noting  
(\ref{inverse}), we have that  
there exists a $\mu_h \in \R_{>0}$
such that for all $\bphi \in \Sh^1$,
\begin{equation}\label{eq:norm_equivalenceaL} 
\|\pi_h[\,\|\bphi\|\,]\|_{L^\infty(\D)}^2
\leq \|\pi_h[\,\|\bphi\|^2]\|_{L^\infty(\D)}
\leq  
\mu_h^2 \int_\D \pi_h[\,\|\bphi\|^2\,].
\end{equation}
It follows from (\ref{eq:free-energy-PdLah}), (\ref{Entropy2}), 
(\ref{eq:norm_equivalenceaL})  
and (\ref{eq:fixed-pointaL}) 
that 
\begin{align} 
\nonumber
&\FdorLh(\bw_\gamma,\bpsi_\gamma) 
\\ 
& \hspace{0.5in}=
\frac{\Re}{2} \intd \|\bw_\gamma\|^2 + \frac{\e}{2\Wi}
\intd
\pi_h[\tr(\bpsi_\gamma-\GdorL(\bpsi_\gamma)-\I)]
\nonumber 
\\
& \hspace{0.5in}\ge
\frac{\Re}{2} \intd\|\bw_\gamma\|^2 + \frac{\e}{4\Wi}
\left[\intd \pi_h[\,\|\bpsi_\gamma\|\,]-2d|\D|\right]
\nonumber
\\
& \hspace{0.5in} \ge 
\frac{\Re}{2} \intd\|\bw_\gamma\|^2 + \frac{\e}{4\Wi\,\mu_h\gamma}
\|\pi_h[\,\|\bpsi_\gamma\|\,]\|_{L^\infty(\D)}
\left[\intd \pi_h[\,\|\bpsi_\gamma\|\,]\right]-\frac{\e d |\D|}{2\Wi}
\nonumber
\\
& \hspace{0.5in} \ge
\min\brk{\frac{\Re}{2},\frac{\e}{4\Wi\,\mu_h\gamma}}
\brk{ \intd\left[\,\|\bw_\gamma\|^2+ \pi_h[\,\|\bpsi_\gamma\|^2\,]\,\right] }
-\frac{\e d |\D|}{2\Wi}
\nonumber
\\
& \hspace{0.5in}= 
\min\brk{\frac{\Re}{2},\frac{\e}{4\Wi\,\mu_h\gamma}}
\gamma^2 
-\frac{\e d |\D|}{2\Wi}
\,.
\label{bb1aL} 
\end{align}
Hence 
for all $\gamma$ sufficiently large,
it follows from (\ref{eq:inequality1aL}) and (\ref{bb1aL})   
that
\beq \label{eq:one-handaL}
\Scal{\mathcal{F}^h(\bw_\gamma,\bpsi_\gamma)}{\brk{\bw_\gamma,\frac{\e}{2\Wi}\brk{\I
-\pi_h[\GdorL'(\bpsi_\gamma)]}}}_\D^h
\ge 0\,.
\eeq

On the other hand as $(\bw_\gamma,\bpsi_\gamma)$ is a fixed point 
of ${\mathcal G}_\gamma^h$, we have that
\begin{align}
\nonumber
&\Scal{\mathcal{F}^h(\bw_\gamma,\bpsi_\gamma)}{\brk{\bw_\gamma,\frac{\e}{2\Wi}
\brk{\I-\pi_h[\GdorL'(\bpsi_\gamma)]}}}_D^h
\\ 
& \hspace{0.3in} = 
-\frac{\Norm{\mathcal{F}^h(\bw_\gamma,\bpsi_\gamma)}_\D^h}{\gamma} 
\intd \left[ \|\bw_\gamma\|^2+ \frac{\e}{2\Wi} 
\pi_h[\bpsi_\gamma :\brk{\I-\GdorL'(\bpsi_\gamma)}]\right] \,.
\label{eq:whereasaL}
\end{align}
It follows from (\ref{Entropy2}), and similarly to (\ref{bb1aL}),
on noting (\ref{eq:norm_equivalenceaL}) and (\ref{eq:fixed-pointaL}) that
\begin{align}
&\intd \left[
\|\bw_\gamma\|^2 + \frac{\e}{2\Wi} \pi_h[\bpsi_\gamma :\brk{\I-\GdorL'(\bpsi_\gamma)}] 
\right] 
\nonumber \\
&\hspace{2in}\ge  
\intd \left[
\|\bw_\gamma\|^2 + \frac{\e}{4\Wi} \left[\pi_h[\,\|\bpsi_\gamma\|\,] - 2d \right]\right]
\nonumber
\\
& \hspace{2in} 
\ge \min\brk{1,\frac{\e}{4\Wi\,\mu_h \gamma}}\gamma^2 - \frac{\e d|\D|}{2\Wi}\,. 
\label{bb2aL}
\end{align}
Therefore on combining (\ref{eq:whereasaL}) and (\ref{bb2aL}), we have for all 
$\gamma$ sufficiently large that 
\begin{align} \label{eq:other-handaL}
\Scal{\mathcal{F}^h(\bw_\gamma,\bpsi_\gamma)}{\brk{\bw_\gamma,\frac{\e}{2\Wi}
\brk{\I-\pi_h[\GdorL'(\bpsi_\gamma)]}}}^h_D < 0 \,,
\end{align}
which obviously contradicts~\eqref{eq:one-handaL}.
Hence the mapping $\mathcal{F}^h$ has a zero in $\mathcal{B}_\gamma^h$
for $\gamma$ sufficiently large. 
\end{proof}

We now have the analogue of stability Theorem \ref{dstabthm}.

\begin{theorem} 
\label{dstabthmaorL}
For any $\delta \in (0, \frac{1}{2}]$, $L \geq 2$,
$N_T \geq 1$ and any
partitioning of $[0,T]$ into $N_T$
time steps, 
there exists a solution  
$\{(\buhdLa^{(L),n},\strhdLa^{(L),n})\}_{n=1}^{N_T}
\in [\Vhone \times \Sh^1]^{N_T}
$  
to (P$^{(L,)\dt}_{\alpha,\delta,h}$).

In addition, it follows for $n=1,\ldots,N_T$ that
\begin{align}
&\FdorLh(\buhdLa^{(L,)n},\strhdLa^{(L,)n}) 
+ 
\frac{\alpha\e\delta^2}{2{\rm Wi}}
\sum_{m=1}^n \dt_m
\intd \|\grad\pi_h[\GdorL'(\strhdLa^{(L,)m})]\|^2
\nonumber \\
& \hspace{0.5in} 
+ \frac{1}{2} \sum_{m=1}^n
\int_\D \left[ {\rm Re} \|\buhdLa^{(L,)m}-\buhdLa^{(L,)m-1}\|^2
+ (1-\e)\dt_{m}\|\gbuhdLa^{(L,)m}\|^2
\right]
\nonumber
\\
&\hspace{0.5in} 
+\frac{\e}{2{\rm Wi}^2}
\sum_{m=1}^n \dt_m
\intd\pi_h[\tr(\BdorL(\strhdLa^{(L,)m})+[\BdorL(\strhdLa^{(L,)m})]^{-1}-2\I)]
\nonumber \\
& \hspace{1in}
\leq
\FdorLh(\buh^0,\strh^0)
+ \frac{1+C_P}{2(1-\e)} \sum_{m=1}^n \dt_m \|\f^m\|_{H^{-1}(\D)}^2
\leq C
\,. 
\label{Fstab1aL}
\end{align}
 Moreover, it follows that 
\begin{align}
\nonumber 
&\max_{n=0, \ldots, N_T} \int_\D \left[ \|\buhdLa^{(L,)n}\|^2 
+ \pi_h[\,\|\strhdLa^{(L,)n}\|\,] + 
\delta^{-1}\,\pi_h[\,\|[\strhdLa^{(L,)n}]_{-}\|\,] \right]
\\
& \hspace{0.1in}+ \sum_{n=1}^{N_T} \intd
\left[ \dt_n \|\gbuhdLa^{(L,)n}\|^2 +
\dt_n \pi_h[\,\|[\BdorL(\strhdLa^{(L,)n})]^{-1}\|\,]
+ \|\buhdLa^{(L,)n}-\buhdLa^{(L,)n-1}\|^2 
\right] 
\nonumber \
\\
&\hspace{4in}
\leq C\,. 
\label{Fstab2aL}
\end{align}  
 \end{theorem}
\begin{proof}
Existence and the stability result (\ref{Fstab1aL}) follow 
immediately from Propositions \ref{prop:existence-PdLah}
and \ref{prop:free-energy-PdLah}, respectively,
on noting 
(\ref{eq:free-energy-PdLah}),
(\ref{idatah}), (\ref{idatahspd}), (\ref{fncont}) and (\ref{freg}).
The bounds (\ref{Fstab2aL}) follow immediately from 
(\ref{Fstab1aL}), on noting   (\ref{eq:positive-term}),
(\ref{Entropy2}), (\ref{modphisq}) and the fact that $\BdorL(\bphi) \in
\RSPD$ for any $\bphi \in \RS$. 
 \end{proof}

\subsection{Convergence of (P$^{(L,)\Delta t}_{\alpha,\delta,h}$) to (P$^{(L,)\Delta t}_{\alpha,h}$)}

We now consider the
corresponding direct finite element approximation of (P$^{(L)}_\alpha$),
i.e. (P$^{(L,)\Delta t}_{\alpha,h}$) without the regularization $\delta$: 

We introduce 
\begin{align}
\ShonePD &=\{\bphi \in \Shone 
 \,:\, \bphi(P_p) \in \RSPD \quad \mbox{for}\quad
 p=1,\ldots, N_P\} \subset \SPD\,. 
 \label{Sh1PD}
\end{align}
It follows from (\ref{idatahspd}) that $\strh^{0} \in \ShonePD$.

({\bf P}$_{\alpha,h}^{(L,)\Delta t}$) 
Setting
$(\buhLa^{(L),0},\strhLa^{(L),0})=
(\buh^{0},\strh^{0})
\in\Vhone\times\ShonePD$, 
then for $n = 1, \ldots, N_T$ 
find $(\buhLa^{(L,)n},$ $\strhLa^{(L,)n})\in\Vhone\times\Shone$ 
such that for any test functions 
$(\bv,\bphi)\in\Vhone\times\Shone$
\begin{subequations}
\begin{align} 
\nonumber 
&\int_\D \biggl[ \Re\left(\frac{\buhLa^{(L,)n}-\buhLa^{(L,)n-1}}{\dt_{n}}\right)
\cdot \bv 
\nonumber \\
& \hspace{0.2in} 
+ \frac{\Re}{2}\Brk{ \left( (\buhLa^{(L,)n-1}\cdot\nabla)\buhLa^{(L,)n}\right) 
 \cdot \bv - 
 \buhLa^{(L,)n} \cdot \left( (\buhLa^{(L,)n-1}\cdot\nabla)\bv \right)}
\nonumber
\\
& \hspace{0.2in}
 + (1-\e) \gbuhLa^{(L,)n}:\grad\bv + \frac{\e}{\Wi} \pi_h[\BorL(\strhLa^{(L,)n})] : 
 \grad\bv 
\biggr] = 
\langle \f^n,\bv\rangle_{H^1_0(\D)}\,,
\label{eq:PLaha}
\\ 
\nonumber
&\int_\D \pi_h \left[\left(\frac{\strhLa^{(L,)n}-\strhLa^{(L,)n-1}}{\dt_{n}}\right) 
: \bphi + \frac{1}{\Wi}\left(\strhLa^{(L,)n}-\I\right): \bphi \right]
\\
& \hspace{1in}
+ \alpha \int_D \grad\strhLa^{(L,)n}::\grad\bphi
- 2\int_\D
\gbuhLa^{(L,)n}: \pi_h[\bphi\,\BorL(\strhLa^{(L,)n})]
\nonumber \\
& \hspace{1in}
- \int_\D \sum_{m=1}^d \sum_{p=1}^d 
[\buhLa^{(L,)n-1}]_m \,\Lambda^{(L)}_{m,p}(\strhLa^{(L,)n})
: \frac{\partial \bphi}{\partial \xx_p}
= 0 \,.
\label{eq:PLahb}
\end{align}
\end{subequations}

\begin{remark}\label{SPD}
Due to the presence of $\BorL$ in 
(\ref{eq:PLaha},b), it is implicitly assumed that
$\strhLa^{(L,)n} \in \ShonePD$, $n=1,\ldots,N_T$; recall
(\ref{eq:BGd}). In addition,  
$\Lambda^{(L)}_{m,p}(\bphi)$ for $\bphi \in \ShonePD$ is defined
similarly to (\ref{Lambdampdef}) with $\widehat \Lambda_{\delta,j}^{(L)}(\widehat \bphi)$
replaced by $\widehat \Lambda_{j}^{(L)}(\widehat \bphi)$,
which
is defined similarly to (\ref{hLambdajdef},b)
with $\lambda_{\delta,j}^{(L)}, \BdorL$ and $\GdorL$
replaced by  $\lambda_{j}^{(L)}, \BorL$ and $\GorL$,
with  $\lambda_{j}^{(L)}$ defined similarly to   
$\lambda_{\delta,j}^{(L)}$ with $\BdorL, \GdorL$ and $\HdorL$
replaced by  $\BorL, \GorL$ and $\HorL$. 
Hence,  
similarly to (\ref{LdmpLinf}), we have that
\begin{align}
\|\Lambda_{m,p}^{L}(\bphi)\|_{L^{\infty}(\D)} \leq C\,L \qquad
\forall 
\bphi \in \ShonePD.
\label{LmpLinf}
\end{align}
\end{remark}

We introduce also the unregularised free energy 
\begin{align} 
\ForLh(\bv,\bphi) := \frac{\Re}{2}\intd\|\bv\|^2 + 
\frac{\e}{2\Wi}\intd \pi_h[\tr(\bphi-\GorL(\bphi)-\I)] \,,
\label{unregL-energy}
\end{align}
which is well defined for all $(\bv,\bphi) \in 
\Vhone \times \ShonePD$.

\begin{theorem} 
\label{dLconthm}
For all regular partitionings $\mathcal{T}_h$ of $\D$
into simplices $\{K_k\}_{k=1}^{N_K}$
and all partitionings $\{\Delta t_n\}_{n=1}^{N_T}$
of $[0,T]$,
there exists a subsequence
$\{\{(\buhdLa^{(L,)n},\strhdLa^{(L,)n})\}_{n=1}^{N_T}\}_{\delta>0}$, where  
$\{(\buhdLa^{(L,)n},$ $\strhdLa^{(L,)n})\}_{n=1}^{N_T} 
\in [\Vhone \times \Shone]^{N_T}$ 
solves (P$^{(L,)\dt}_{\alpha,\delta,h}$), and  
$\{(\buhLa^{(L,)n},\strhLa^{(L,)n})\}_{n=1}^{N_T} \in [\Vhone \times \Shone]^{N_T}$ 
such that for the subsequence
\begin{align}
\buhdLa^{(L,)n} \rightarrow \buhLa^{(L,)n}, \quad 
\strhdLa^{(L,)n} \rightarrow \strhLa^{(L,)n} \quad \mbox{as } \delta \rightarrow 0_+\,,
\quad\mbox{for} \quad n=1, \ldots, N_T\,.
\label{Lconv}
\end{align}
In addition, for 
$n=1,\ldots, N_T$, $\strhLa^{(L,)n} \in \ShonePD,$ 
and $\{(\buhLa^{(L,)n},\strhLa^{(L,)n})
\}_{n=1}^{N_T}
\in [\Vhone \times \ShonePD]^{N_T}$
solves (P$_{\alpha,h}^{(L,)\dt}$).

Moreover, we have for $n=1,\ldots,N_T$ that 
\begin{align}
\nonumber
&\frac{\ForLh(\buhLa^{(L,)n},\strhLa^{(L,)n})-
\ForLh(\buhLa^{(L,)n-1},\strhLa^{(L,)n-1})}{\dt_{n}} 
+ \frac{ {\rm Re}}{2\dt_{n}}\intd\|\buhLa^{(L,)n}-\buhLa^{(L,)n-1}\|^2
\\
& \hspace{0.1in} 
+(1-\e)\intd\|\gbuhLa^{(L,)n}\|^2
+\frac{\e}{2{\rm Wi}^2}\intd \pi_h[\tr(\BL(\strhLa^{(L,)n})+[\BL(\strhLa^{(L,)n})]^{-1}-2\I)]
\nonumber \\
& \qquad 
\le 
\frac{1}{2}(1-\e)\intd\|\gbuhLa^{(L,)n}\|^2
+ \frac{1+C_P}{2(1-\e)} \|\f^{n}\|_{H^{-1}(\D)}^2
\,,
\label{eq:estimate-PLh}
\end{align}
and 
\begin{align}
\nonumber 
&\max_{n=0, \ldots, N_T} \int_\D 
\left[ \|\buhLa^{(L,)n}\|^2 + \pi_h[\,\|\strhLa^{(L,)n}\|\,] \right]
\\
& \hspace{0.1in}+ \sum_{n=1}^{N_T} \intd
\left[ \dt_n \|\gbuhLa^{(L,)n}\|^2 +
\dt_n \pi_h[\,\|[\BorL(\strhLa^{(L,)n})]^{-1}\|\,]
+ \|\buhLa^{(L,)n}-\buhLa^{(L,)n-1}\|^2 
\right] 
\nonumber \\
& \hspace{4in}
\leq C\,. 
\label{Fstab2newaL}
\end{align} 
\end{theorem}
\begin{proof}
For any integer $n \in[1,N_T]$,
the desired subsequence convergence result (\ref{Lconv}) follows immediately 
from (\ref{Fstab2aL}),
as $(\buhdLa^{(L,)n},$ $\strhdLa^{(L,)n})$ 
are finite dimensional for fixed $\Vhone \times S_h^1$.
It also follows from (\ref{Fstab2aL}),
(\ref{Lconv}) and (\ref{Lip}) that $\pi_h[\,[\strhdLa^{(L,)n}]_{-}]$ vanishes on $\D$, 
so that
$\strhLa^{(L,)n}$ must be non-negative definite on $\D$. Hence 
on noting this,  (\ref{eq:Bd}), (\ref{Lip}) and (\ref{Lconv}), 
we have the following subsequence convergence results
\begin{align}
\BdorL(\strhLa^{(L,)n}) \rightarrow \BorL(\strhLa^{(L,)n}) 
\quad \mbox{and}
\quad  
\BdorL(\strhdLa^{(L,)n}) \rightarrow  \BorL(\strhLa^{(L,)n}) \quad \mbox{as} \quad
\delta \rightarrow 0_+\,.
\label{betaLcon}
\end{align}

It also follows from (\ref{Fstab2aL}), (\ref{betaLcon}) and as $[\BdorL(\strhdLa^{(L,)n})]^{-1}
\BdorL(\strhdLa^{(L,)n})=\I$ that the following subsequence result 
\begin{align}
\pi_h[\,[\BdorL(\strhdLa^{(L,)n})]^{-1}] \rightarrow 
\pi_h[\,[\BorL(\strhLa^{(L,)n})]^{-1}]\qquad \mbox{as} \quad
\delta \rightarrow 0_+
\label{posLdef}
\end{align}
holds, and so $\strhLa^{(L,)n} \in \ShonePD$.
Therefore, we have 
from (\ref{Lconv}) and (\ref{eq:Gd}) that
\begin{align}
\GdorL(\strhdLa^{(L,)n}) \rightarrow \GorL(\strhLa^{(L,)n}) \qquad \mbox{as} \quad
\delta \rightarrow 0_+\,.
\label{GLcon}
\end{align}

Similarly to (\ref{GLcon}), it follows from (\ref{Lconv}), (\ref{betaLcon}),
(\ref{Lambdampdef}) and (\ref{hLambdajdef},b)
as $\strhLa^{(L,),n}
\in \ShonePD$ that for $m,\,p = 1, \ldots, d$
\begin{align}
\Lambda_{\delta,m,p}^{(L)}(\strhdLa^{(L,)n}) 
\rightarrow 
\Lambda_{m,p}^{(L)}(\strhLa^{(L,)n}) 
\quad \mbox{as} \quad
\delta \rightarrow 0_+\,.
\label{Lamcon}
\end{align} 
Hence using (\ref{Lconv}), (\ref{betaLcon}) and (\ref{Lamcon}),
we can pass to the limit $\delta \rightarrow 0_+$ in (P$_{\alpha,\delta,h}^{(L,)\dt}$), 
(\ref{eq:PdLaha},b),
to show that $\{(\buhLa^{(L,)n},\strhLa^{(L,)n})\}_{n=1}^{N_T}
\in [\Vhone \times \ShonePD]^{N_T}$
solves (P$_{\alpha,h}^{(L,)\dt}$), (\ref{eq:PLaha},b). 
Similarly, using (\ref{Lconv}), (\ref{betaLcon}), (\ref{posLdef}) and (\ref{GLcon}), 
and noting (\ref{eq:free-energy-PdLah}) and (\ref{unregL-energy}),
we can pass to the limit $\delta \rightarrow 0_+$ in (\ref{eq:estimate-PdLah})
and (\ref{Fstab2aL})
to obtain the desired results (\ref{eq:estimate-PLh}) and (\ref{Fstab2newaL}).
\end{proof}

For later purposes, we introduce the following notation 
in line with (\ref{eq:constant-source}). 
Let $\buhLa^{(L,)\dt} \in C([0,T];\Vhone)$ and $\buhLa^{(L,)\dt,\pm}  
\in L^\infty(0,T;\Vhone)$
be such that
for $n=1,\dots,N_T$
\begin{subequations}
\begin{alignat}{2}
\buhLa^{(L,)\dt}(t,\cdot) &:= \frac{t-t^{n-1}}{\dt_n} \buhLa^{(L,)n}(\cdot) 
+ \frac{t^n-t}{\dt_n} \buhLa^{(L,)n-1}(\cdot)
&&\quad t \in [t^{n-1},t^n], 
\label{timeconaL}\\
\buhLa^{(L,)\dt,+}(t,\cdot)&:= \buhLa^{(L,)n}(\cdot),
\quad \buhLa^{(L,)\dt,-}(t,\cdot):= \buhLa^{(L,)n-1}(\cdot)
&&\quad t \in [t^{n-1},t^n), 
\label{time+-aL}\\
\mbox{and} \quad \Delta(t)&:= \dt_n 
&&\quad t \in [t^{n-1},t^n), 
\label{deltat}
\end{alignat}
\end{subequations}
We note that 
\begin{align}
\buhLa^{(L,)\dt}-\buhLa^{(L,)\dt,\pm}&=(t-t^n_{\pm})\frac{\partial \buhLa^{(L,)\dt}}
{\partial t}
\quad t \in (t^{n-1},t^n), \quad n=1,\dots,N_T\,, 
\label{eqtime+}
\end{align}
where $t^n_{+}:=t^n$ and $t^n_{-}:=t^{n-1}$.
We define 
$\strhLa^{(L,)\dt} \in C([0,T];\ShonePD)$
and $\strhLa^{(L,)\dt,\pm} \in L^\infty(0,T;$ $\ShonePD)$ similarly to
(\ref{timeconaL},b).

Using the notation (\ref{timeconaL},b), 
(\ref{eq:PLaha})
multiplied by $\Delta t_n$ and summed for $n=1,\dots, N_T$
can be restated as:
\begin{align}
\nonumber 
& \displaystyle\int_{0}^{T} \int_\D 
\left[ \Re \frac{\partial \buhLa^{(L,)\dt}}{\partial t}
 \cdot \bv 
+ (1-\e) 
\grad \buhLa^{(L,)\dt,+} :
\grad \bv \right] dt
\\
&\quad+ \frac{\Re}{2} \int_{0}^T \int_{\D}
\left[ \left[ (\buhLa^{(L,)\dt,-} \cdot \grad) \buhLa^{(L,)\dt,+} \right] \cdot \bv
- \left[(\buhLa^{(L,)\dt,-} \cdot \grad) \bv \right] \cdot \buhLa^{(L,)\dt,+}
\right] dt
\nonumber \\
& \qquad =
\int_0^T \left[ \langle \f^+,\bv \rangle_{H^1_0(\D)}
- \frac{\e}{\Wi} \int_{\D}
\pi_h[\BorL(\strhLa^{(L,)\dt,+})]: \grad
\bv \right] dt
\nonumber \\
& \hspace{3in}
\qquad \forall \bv \in L^2(0,T;\Vhone).
\label{equncon}
\end{align}
Similarly, (\ref{eq:PLahb})
multiplied by $\Delta t_n$ and
summed for $n=1,\dots,N_T$ can be restated as:
\begin{align}
\nonumber
&\int_{0}^T \intd \pi_h\left[\frac{\partial \strhLa^{(L,)\dt}}{\partial t}
: \bchi + \frac{1}{\Wi} ( \strhLa^{(L,)\dt,+}-\I): \bchi
\right] dt
\nonumber \\
& \ 
+ \alpha \int_{0}^T \!\intd \grad \strhLa^{(L,)\dt,+} :: \grad \bchi \,dt
-2 \int_{0}^T \!\intd \grad \buhLa^{(L,)\dt,+}
: \pi_h[\bchi\,\BorL(\strhLa^{(L,)\dt,+})] \,dt
\nonumber \\
& \ -\int_{0}^T \intd
\sum_{m=1}^d \sum_{p=1}^d [\buhLa^{(L,)\dt,-}]_m \,\Lambda^{(L)}_{m,p}
(\strhLa^{(L,)\dt,+}) : \frac{\partial \bchi}{\partial \xx_p}\,
dt
=0 
\nonumber \\
& \hspace{3in}
\qquad \forall \bchi \in L^2(0,T;\Shone).
\label{eqpsincon}
\end{align}

We note also the following Lemma for later purposes.
\begin{lemma} 
\label{lemMXitt}
For all $K_k \in {\mathcal T}_h$,
and for all $\bphi \in \ShonePD$ we have that
\begin{align}
&\int_{K_k}
\|\pi_h[\BorL(\bphi)]-\BorL(\bphi)\|^2
+ \max_{m,p=1,\dots,d}
\int_{K_k}
\|\Lambda^{(L)}_{m,p}(\bphi)- \BorL(\bphi)\,\delta_{mp}\|^2 
\nonumber \\
& \hspace{3in}
\leq C  
\,h^2\int_{K_k} \|\grad \bphi\|^2
\,.
\label{MXittxbd}
\end{align}
\end{lemma}
\begin{proof}
First,
we have from  
(\ref{Lip}) that
for all $\bphi \in \ShonePD$ 
\begin{align}
\nonumber
\int_{K_k} \|\pi_h[\BorL(\bphi)]- \BorL(\bphi)\|^2
&\leq C\,|K_k|\,\sum_{j=0}^d \|\BorL(\bphi(P^k_j))-\BorL(\bphi)\|_{L^{\infty}(K_k)}^2
\\
&\leq C\,|K_k|\,\sum_{j=0}^d \|\bphi(P^k_j)-\bphi\|_{L^{\infty}(K_k)}^2
\nonumber \\
&\leq C \,h^2 |K_k|\,
\|\grad \bphi\|_{L^{\infty}(K_k)}^2
\leq C \,h^2 \int_{K_k}
\|\grad \bphi\|^2\,.
\label{MXittpd3}
\end{align}
where $\{P_j^k\}_{j=0}^d$ are the vertices of $K_k$. 
Hence we have the desired first bound in  (\ref{MXittxbd}).

It follows from 
the $\delta$ independent versions
of (\ref{Lambdampdef}) and (\ref{hLambdajdef},b), recall Remark \ref{SPD}, 
(\ref{Breg})
and (\ref{Lip})
that for all $\bphi \in \ShonePD$
\begin{align}
&\int_{K_k}
\|\Lambda^{(L)}_{m,p}(\bphi)- \pi_h[\BorL(\bphi)]\,\delta_{mp}\|^2  
\nonumber \\
& \hspace{1in}
= \int_{K_k} 
\|\sum_{j=1}^d \left[ [(B_k^T)^{-1}]_{mj} 
\,[ \widehat{\Lambda}^{(L)}_j(\widehat \bphi) 
- \pi_h[
\beta^{(L)}(\bphi)]\,]\,[B_k^T]_{jp} \right]\|
\nonumber \\ 
&\hspace{1in}\leq C \int_{K_k}
\sum_{j=1}^d 
\|\widehat \Lambda^{(L)}_{j}(\widehat \bphi)- \pi_h[\BorL(\bphi)]
\|^2  
\nonumber \\
& \hspace{1in}\leq C |K_k|
\max_{i,j=0,\dots,d}
\|\BorL(\bphi(P_j^k))-\BorL(\bphi(P_i^k))\|^2
\nonumber \\
& \hspace{1in}\leq C |K_k|
\max_{i,j=0,\dots,d}
\|\bphi(P_j^k)-\bphi(P_i^k)\|^2
\nonumber \\
&\hspace{1in}\leq C\, h^2 \int_{K_k}
\|\grad \bphi\|^2\,.
\label{MXittxbdp1} 
\end{align}
Combining (\ref{MXittxbdp1}) and the first bound in (\ref{MXittxbd}) 
yields the second bound in (\ref{MXittxbd}).
\end{proof}

\section{Convergence of (P$^{L,\Delta t}_{\alpha,h}$) to
(P$^{L}_{\alpha}$)} 
\label{sec:convergence}

Before proving our convergence result, we first deduce some simple inequalities
that will be required.
We recall the following well-known results concerning the interpolant $\pi_h$:
\begin{subequations}
\begin{align}
\|(\I-\pi_h)\bphi\|_{W^{1,\infty}(K_k)} &\leq C\,h\,|\bphi|_{W^{2,\infty}(K_k)} 
\qquad \forall \bphi \in [W^{2,\infty}(K_k)]_{S}^{d \times d}, 
\nonumber \\
& \hspace{2in}
\quad k=1,\ldots, N_K;
\label{interp1}\\
\|(\I-\pi_h)[\bchi : \bphi]\|_{L^2(\D)}
& \leq C\,h^2\,\|\nabla \bchi\|_{L^2(\D)}\,\|\nabla \bphi\|_{L^\infty(\D)}
\nonumber
\\
&\leq C\,h\,\|\bchi\|_{L^2(\D)}\,\|\nabla \bphi\|_{L^\infty(\D)}
\qquad \forall \bchi,\,\bphi \in \Shone. 
\label{interp2}
\end{align}
\end{subequations}
We note for any $\zeta \in \R_{>0}$ that 
\begin{align}
&\left[\pi_h[\bchi:\bphi]\right](\xx) \leq \frac{1}{2}
\left[\pi_h\left[\zeta \,\|\bchi\|^2 + \zeta^{-1}\|\bphi\|^2 \right]\right](\xx)
\nonumber
\\
&\hspace{1.5in}\forall \xx \in K_k, \quad \forall 
\bchi, \bphi \in [C(\overline{K_k})]^{d \times d}, \quad k=1,\dots,N_K\,.
 \label{vecint}
\end{align}
Combining (\ref{interpinf}), (\ref{normprod}) and (\ref{BdLphi}), 
we have for all $\bphi \in \ShonePD$ and for all $\bpsi \in \Shone$ that
\begin{align}
\int_D \|\pi_h[\bpsi\,\BL(\bphi)]\|^2
& \leq \int_D \pi_h[\,\|\bpsi\,\BL(\bphi)\|^2\,]
\leq \int_D \pi_h[\,\|\bpsi\|^2\,\|\BL(\bphi)\|^2\,]
\nonumber \\
\label{combi} 
&\leq d\,L^2 \int_D \pi_h[\,\|\bpsi\|^2]\,.
\end{align}

We require also the $L^2$ projector ${\mathcal R}_h : \Uz \rightarrow \Vhone$
defined by
\begin{align}
\intd (\bv -{\mathcal R}_h\bv)\bw =0 \qquad \forall \bw \in \Vhone\,. 
\label{Rh}
\end{align}
In addition, we require  
${\mathcal P}_h : \S \rightarrow \Shone$
defined by
\begin{align}
\intd \pi_h[{\mathcal P}_h\bchi :\bphi] =
\intd \bchi : \bphi  \qquad \forall \bphi \in \Shone\,. 
\label{Ph}
\end{align}
It is easily deduced for $p =1,\ldots,N_P$ and $i,\,j
= 1, \ldots, d$ that
\begin{align}
[{\cal P}_h \bchi]_{ij}(P_p) = \frac{1}{\intd \eta_p} \intd [{\cal P}_h \bchi]_{ij}\,\eta_p\,,
\label{strh01def}
\end{align}
where $\eta_p \in {\rm Q}_h^1$ is such that $\eta_p(P_r)=\delta_{pr}$
for $p,\,r = 1,\ldots,N_P$.
It follows from
(\ref{Ph}) and
(\ref{interpinf}) with $\bphi= {\mathcal P}_h \bchi$, in both cases, that 
\begin{align}
\intd \|{\mathcal P}_h \bchi\|^2 \leq \intd \pi_h[\,\|{\mathcal P}_h \bchi\|^2]
\leq \intd \|\bchi\|^2 \qquad \forall \bchi \in [L^2(\D)]^{d \times d}_S\,.
\label{PhstabL2}
\end{align}

We shall assume from now on that $\D$ is convex and that 
the family $\{{\mathcal T}_h\}_{h>0}$ is quasi-uniform,
i.e.\ $h_k \geq C\,h$, $k=1,\ldots, N_K$. 
It then follows that
\begin{align}
\|{\mathcal R}_h \bv\|_{H^1(\D)} \leq C \|\bv\|_{H^1(\D)}
\qquad \forall \bv \in \Uz\,,  
\label{Rhstab}
\end{align}
see Lemma 4.3 in Heywood and Rannacher.\cite{HR} Similarly, it is easily established that
\begin{align}
\|{\mathcal P}_h \bchi\|_{H^1(\D)} \leq C \|\bchi\|_{H^1(\D)}
\qquad \forall \bchi \in [H^1(\D)]^{d \times d}_S \,.  
\label{Phstab}
\end{align}

Let $([H^1(\D)]^{d \times d}_S)'$ be the topological dual of
  $[H^1(\D)]^{d \times d}_S$ with $[L^2(\D)]^{d \times d}_S$
being the pivot space.
Let ${\mathcal E} : ([H^1(\D)]^{d \times d}_S)' \rightarrow [H^1(\D)]^{d \times d}_S$
be such that ${\mathcal E} \bchi$ is the unique solution of the Helmholtz
problem
\begin{align}
\intd \left [ \grad ({\mathcal E} \bchi) ::
\grad \bphi + ({\mathcal E} \bchi) : \bphi \right]=
\langle \bchi,\bphi\rangle_{H^1(\D)} \qquad \forall \bphi \in [H^1(\D)]^{d \times d}_S\,,
\label{Echi}
\end{align}  
where $\langle \cdot,\cdot \rangle_{H^1(\D)}$ denotes the duality pairing 
between $([H^1(\D)]^{d \times d}_S)'$ and $[H^1(\D)]^{d \times d}_S$. We note that
\begin{align}
\langle \bchi,{\mathcal E}\bchi\rangle_{H^1(\D)} =\|{\mathcal E} \bchi\|_{H^1(\D)}^2
\qquad \forall \bchi \in ([H^1(\D)]^{d \times d}_S)'\,,
\label{Echinorm}
\end{align} 
and $\|{\mathcal E}\cdot\|_{H^1(\D)}$ is a norm on $([H^1(\D)]^{d \times d}_S)'$.

Let $\Uz'$ be the topological dual of $\Uz$
with the space of weakly divergent free functions in $[L^2(\D)]^d$
being the pivot space.
Let ${\mathcal S} : \Uz' \rightarrow
\Uz$ be such that ${\mathcal S} \bw$ is the unique solution to the Helmholtz-Stokes
problem
\begin{align}
\intd \left [ \grad ({\mathcal S} \bw) :
\grad \bv + ({\mathcal S} \bw) \cdot \bv \right]=
\langle \bw,\bv\rangle_{\Uz} \qquad \forall \bv \in \Uz\,,
\label{Sw}
\end{align}  
where $\langle \cdot,\cdot \rangle_{\Uz}$ denotes the duality pairing 
between $\Uz'$ and $\Uz$. We note that
\begin{align}
\langle \bw,{\mathcal S}\bw\rangle_{\Uz} =\|{\mathcal S} \bw\|_{H^1(\D)}^2
\qquad \forall \bw \in \Uz' \,,
\label{Swnorm}
\end{align} 
and $\|{\mathcal S}\cdot\|_{H^1(\D)}$ is a norm on the reflexive space $\Uz'$.
As $\Uz$ is continuously embedded in $[H^1_0(\D)]^d$, it follows that
$[H^{-1}(\D)]^d$ is continuously embedded in $\Uz'$.

We recall the following well-known Gagliardo-Nirenberg inequality. 
Let $r \in [2,\infty)$ if $d=2$, and $r \in [2,6]$ if $d=3$ and $\theta =
d(\frac{1}{2}-\frac{1}{r})$. Then, there exists a positive constant $C(\D,r,d)$
such that
\begin{align}
\|\eta\|_{L^r(\D)} \leq C(\D,r,d) \|\eta\|_{L^2(\D)}^{1-\theta} 
\|\eta\|_{H^1(\D)}^\theta \qquad \forall \eta \in H^1(\D)\,.
\label{GN}
\end{align}

We recall also the following compactness result, 
see e.g.\ Theorem 2.1 on p184 in Temam\cite{Temam} and Simon.\cite{Simon}
Let ${\mathcal Y}_0$, ${\mathcal Y}$ and
${\mathcal Y}_1$ be Banach
spaces, ${\mathcal Y}_i$, $i=0,1$, reflexive, with a compact embedding 
${\mathcal Y}_0
\hookrightarrow {\mathcal Y}$ and a continuous embedding ${\mathcal Y} \hookrightarrow
{\mathcal Y}_1$. 
Then, for $\mu_i>1$, $i=0,1$, the following embedding is compact~:
\begin{align}
&\{\,\eta \in L^{\mu_0}(0,T;{\mathcal Y}_0): \frac{\partial \eta}{\partial t}
\in L^{\mu_1}(0,T;{\mathcal Y}_1)\,\} \hookrightarrow L^{\mu_0}(0,T;{\mathcal Y}) \ .
\label{compact1}
\end{align}

\begin{theorem}
\label{dstabthmaL}
Under the assumptions of Theorem \ref{dLconthm},
there exists a solution 
$\{(\buhLa^{L,n},\strhLa^{L,n})\}_{n=1}^{N_T} \in [\Vhone \times \ShonePD]^{N_T}$ 
of (P$^{L,\dt}_{\alpha,h}$)
such that, in addition to the bounds (\ref{eq:estimate-PLh}) and (\ref{Fstab2newaL}),
the following bounds hold:
\begin{subequations}
\begin{align}
&\max_{n=0, \ldots, N_T} \intd
\pi_h[\,\|\strhLa^{L,n}\|^2\,]
+  \sum_{n=1}^{N_T} \intd \left[\dt_n \alpha \|\grad \strhLa^{L,n}\|^2 
+ \pi_h[\,\|\strhLa^{L,n}-\strhLa^{L,n-1}\|^2\,]
\right]
\nonumber \\
& \hspace{3.5in}
\leq C(L)\,,
\label{Fstab3aL}
\\
&\sum_{n=1}^{N_T} \dt_n \left\|{\mathcal S}
\left( \frac{\buhLa^{L,n}-\buhLa^{L,n-1}}{\dt_n}\right)
\right\|^{\frac{4}{\vartheta}}_{H^1(\D)}
+ 
\sum_{n=1}^{N_T} \dt_n \left\|{\mathcal E}
\left( \frac{\strhLa^{L,n}-\strhLa^{L,n-1}}{\dt_n}\right)
\right\|^{2}_{H^1(\D)}
\nonumber \\
&\hspace{3.5in}\leq C(L,T)\,;
\label{Fstab4aL}
\end{align}
\end{subequations}
where
\begin{align}
\vartheta \in (2,4) \quad \mbox{if } d=2
\qquad \mbox{and}
\qquad 
\vartheta = 3 \quad \mbox{if } d=3.
\label{varth}
\end{align}
\end{theorem}
\begin{proof}
Existence  
and the bounds
(\ref{eq:estimate-PLh}) and (\ref{Fstab2newaL})
were proved in Theorem
\ref{dLconthm}.
On choosing $\bphi \equiv \strhLa^{L,n}$ in the version of (\ref{eq:PLahb})
dependent on $L$, it follows
from (\ref{elemident}), (\ref{vecint}),  
(\ref{combi}), 
(\ref{Fstab2newaL}) and (\ref{LmpLinf}) 
on applying a Youngs' inequality that 
\begin{align}
\nonumber
&\frac{1}{2}\intd \pi_h[ \,\|\strhLa^{L,n}\|^2+\|\strhLa^{L,n}-
\strhLa^{L,n-1}\|^2\,] 
\nonumber \\
& \quad
+ \dt_n \alpha \intd \|\grad \strhLa^{L,n}\|^2
+ \frac{\dt_n}{2 \Wi} \intd \pi_h[ \,\|\strhLa^{L,n}\|^2\,]
\nonumber
\\
&\qquad \leq \frac{1}{2}\intd \pi_h[ \,\|\strhLa^{L,n-1}\|^2\,]
+ 2\dt_n
\intd \|\grad \buhLa^{L,n}\| 
\,\|\pi_h[\strhLa^{L,n}\,\BL(\strhLa^{L,n})]\|
\nonumber \\
& \qquad \qquad + \dt_n
\intd \|\buhLa^{L,n-1}\|\,\|\grad \strhLa^{L,n}\| \left(\sum_{m=1}^d \sum_{p=1}^d
\|\Lambda_{m,p}^L( \strhLa^{L,n}) \|^2\right)^{\frac{1}{2}}
+ \frac{\dt_n d|D|}{2\Wi}
\nonumber \\
&\qquad \leq \frac{1}{2} \left[\intd \pi_h[ \,\|\strhLa^{L,n-1}\|^2\,]
+ \dt_n \alpha \intd \|\grad \strhLa^{L,n}\|^2\right]
+ \frac{\dt_n}{4\Wi}\intd \pi_h[\,\|\strhLa^{L,n}\|^2\,] 
\nonumber \\
& \qquad \qquad 
+ \dt_n C(L) \left[ 1 + \intd \|\grad \buhLa^{L,n}\|^2 \right].
\label{sigstab}
\end{align}
Hence, summing (\ref{sigstab}) from $n=1,\dots,m$ for $m=1, \dots,N_T$
yields, on noting (\ref{Fstab2newaL}), the desired result (\ref{Fstab3aL}).

On choosing $\bw = {\mathcal R}_h \left[ {\mathcal S}
\left(\frac{\buhLa^{L,n}-\buhLa^{L,n-1}}{\dt_n}\right)\right] \in \Vhone$ 
in the version of (\ref{eq:PLaha})
dependent on $L$
yields, on noting (\ref{Rh}), (\ref{Swnorm}), (\ref{Rhstab})
 and Sobolev embedding, that
\begin{align}
&\Re \left\|{\mathcal S} \left(\frac{\buhLa^{L,n}-\buhLa^{L,n-1}}{\Delta t_n}\right)
\right\|^2_{H^1(\D)}
\nonumber \\
&\qquad=\Re\int_{\D}
\frac{\buhLa^{L,n}-\buhLa^{L,n-1}}{\Delta t_n} \cdot
{\mathcal R}_h \left[ {\mathcal S}
\left(\frac{\buhLa^{L,n}-\buhLa^{L,n-1}}{\Delta t_n}\right)\right]
\nonumber
\\
&\qquad = -
\intd
\left[(1-\e) \grad
\buhLa^{L,n}  
+ \frac{\e}{\Wi} 
\pi_h[\BL(\strhLa^{L,n})] \right] 
: \grad \left[{\mathcal R}_h \left[{\mathcal S}
\left(\frac{\buhLa^{L,n}-\buhLa^{L,n-1}}{\Delta t_n}\right)\right] \right]
\nonumber \\ 
& \hspace{1in} 
- \frac{\Re}{2} \intd 
\left((\buhLa^{L,n-1} \cdot \grad) \buhLa^{L,n}\right)
\cdot \displaystyle
{\mathcal R}_h \left[{\mathcal S}
\left(\frac{\buhLa^{L,n}-\buhLa^{L,n-1}}{\Delta t_n}\right)
\right]
\nonumber \\
& \hspace{1in} + \frac{\Re}{2}
\intd
\buhLa^{L,n} \cdot
\left((\buhLa^{L,n-1} \cdot \grad)
\left[ {\mathcal R}_h \left[{\mathcal S} 
\left(\frac{\buhLa^{L,n}-\buhLa^{L,n-1}}{\Delta t_n}\right)
\right]\right]\right) \nonumber \\
& \hspace{1in} + \left \langle \f^n, 
\displaystyle
{\mathcal R}_h \left[{\mathcal S}
\left(\frac{\buhLa^{L,n}-\buhLa^{L,n-1}}{\Delta t_n}\right)
\right] \right \rangle_{H^1_0(\D)}
\nonumber \\
& \qquad \leq C \bigl[\|\pi_h[\BL(\strhLa^{L,n})]\|^2_{L^2(\D)}
+\|\grad \buhLa^{L,n}\|^2_{L^2(\D)}
+\|\,\|\buhLa^{L,n-1}\|\,\|\buhLa^{L,n}\|\,\|^2_{L^2(\D)}
\nonumber\\
& \hspace{1in}
+ \|\,\|\buhLa^{L,n-1}\| \,\|\grad \buhLa^{L,n}\|\,\|_{L^{1+\theta}(\D)}^2
+ \|\f^n\|_{H^{-1}(\D)}^2 \bigr],
\label{equndtb}
\end{align}
for any $\theta>0$ if $d=2$ and for $\theta=\frac{1}{5}$ if $d=3$.
Applying the Cauchy--Schwarz and the algebraic-geometric mean
inequalities, in conjunction with 
(\ref{GN}) and the Poincar\'{e} inequality
(\ref{eq:poincare})
yields that
\begin{align}
\nonumber
\displaystyle
\|\,\|\buhLa^{L,n-1}\| \,\|\buhLa^{L,n}\|\,\|_{L^2(\D)}^2
&\leq
\|\buhLa^{L,n-1}\|^2_{L^4(\D)}
\,\|\buhLa^{L,n}\|^2_{L^4(\D)}
\leq
\textstyle \frac{1}{2}
\displaystyle\sum_{m=n-1}^n
\|\buhLa^{L,m}\|^4_{L^4(\D)}
\\ &
\leq C \displaystyle\sum_{m=n-1}^n
\left[
\|\buhLa^{L,m}\|^{4-d}_{L^2(\D)}\,
\|\grad \buhLa^{L,m}\|^{d}_{L^2(\D)}\,
\right].
\label{L4sob}
\end{align}
Similarly, we have for any $\theta \in (0,1)$, if $d=2$, that
\begin{subequations}
\begin{align}
\nonumber
\|\,\|\buhLa^{L,n-1}\| \,\|\grad \buhLa^{L,n}\|\,\|_{L^{1+\theta}(\D)}^2
& \leq
\|\buhLa^{L,n-1}\|_{L^{\frac{2(1+\theta)}{1-\theta}}(\D)}^2
\,\|\grad \buhLa^{L,n}\|_{L^{2}(\D)}^2
\\
&
\leq C \|\buhLa^{L,n-1}\|_{L^2(\D)}^{\frac{2(1-\theta)}{1+\theta}}
\displaystyle\sum_{m=n-1}^n
\|\grad \buhLa^{L,m}\|^{\frac{2(1+3\theta)}{1+\theta}}_{L^2(\D)}\
\,;
\label{sob1}
\end{align}
and if $d=3$, $(\theta = \frac{1}{5})$, that
\begin{align}
\nonumber
\|\,\|\buhLa^{L,n-1}\| \,\|\grad \buhLa^{L,n}\|\,\|_{L^{\frac{6}{5}}(\D)}^2
& \leq
\|\buhLa^{L,n-1}\|_{L^{3}(\D)}^2
\,\|\grad \buhLa^{L,n}\|_{L^{2}(\D)}^2
\\
&\leq C \|\buhLa^{L,n-1}\|_{L^2(\D)} \displaystyle\sum_{m=n-1}^n
\|\grad \buhLa^{L,m}\|^{3}_{L^2(\D)}.
\label{sob2}
\end{align}
\end{subequations}
On taking the $\frac{2}{\vartheta}$
power of both sides of
(\ref{equndtb}), recall (\ref{varth}),
multiplying by $\Delta t_n$,
summing from $n=1,\dots,N_T$
and noting (\ref{L4sob}), (\ref{sob1})
with $\theta = \frac{\vartheta - 2}{6 - \vartheta}
\Leftrightarrow \vartheta = \frac{2(1+3\theta)}{(1+\theta)}$,
(\ref{sob2}), (\ref{Deltatqu}), (\ref{fncont}),
(\ref{Fstab2newaL}), (\ref{idatah}) and (\ref{BdLphi}) yields that
\begin{align}
\nonumber
&\sum_{n=1}^{N_T} \Delta t_n 
\left \|{\mathcal S} \left(\frac{\buhLa^{L,n}-\buhLa^{L,n-1}}{\Delta
t_n}\right)\right\|^{\frac{4}{\vartheta}}_{H^1(\D)}
\\
&\hspace{0.5in} \leq CL^2
+ C(T)\, \left[\,\sum_{n=1}^{N_T} \Delta t_n
\,\left[\|\grad \buhLa^{L,n}\|_{L^2(\D)}^2 +
\|\f^n\|^2_{H^{-1}(\D)}
\right]
\right]^{\frac{2}{\vartheta}}
\nonumber \\
&\hspace{1in} 
+C\,\left[1+ \max_{n=0,\dots, N_T}
\left(\|\buhLa^{L,n}\|^2_{L^2(\D)} 
\right) 
 \right]
\,\left[\sum_{n=0}^{N_T} \Delta t_n\, 
\|\grad \buhLa^{L,n}\|^2_{L^2(\D)}
\right]\nonumber\\
& \hspace{0.5in} \leq C(L,T);
\label{corotfor3dis}
\end{align}
and hence the first bound in (\ref{Fstab4aL}).

On choosing $\bphi
= {\mathcal P}_h\left[{\mathcal E} \left( \frac{\strhLa^{L,n}-\strhLa^{L,n-1}}
{\Delta t_n}
\right) \right] \in \Shone$ in 
the version of (\ref{eq:PLahb})
dependent on $L$
yields,
on noting (\ref{Ph}), (\ref{Echi}),
(\ref{vecint}), (\ref{PhstabL2}), 
(\ref{Phstab}), (\ref{combi}) and
(\ref{LmpLinf}), that
\begin{align}
&\left
\|{\mathcal E} \left( \frac{\strhLa^{L,n}-\strhLa^{L,n-1}}{\Delta t_n}
\right)
\right\|^2_{H^1(\D)}
\nonumber \\
&\qquad =\intd \pi_h\left[ \left(\frac{\strhLa^{L,n}-\strhLa^{L,n-1}}{\Delta t_n}
\right) :
 {\mathcal P}_h \left[{\mathcal E} \left( \frac{\strhLa^{L,n}
 -\strhLa^{L,n-1}}{\Delta t_n}
\right) \right] \right]
\nonumber
\\
&\qquad =\frac{1}{\Wi}\intd \pi_h \left[
(\I-\strhLa^{L,n}) : {\mathcal P}_h 
\left[{\mathcal E} \left( \frac{\strhLa^{L,n}-\strhLa^{L,n-1}}{\Delta t_n}
\right) \right] \right]
\nonumber \\
& \qquad \qquad -\alpha \intd
\grad  \strhLa^{L,n} ::
\grad \left[{\mathcal P}_h \left[{\mathcal E} 
\left( \frac{\strhLa^{L,n}-\strhLa^{L,n-1}}{\Delta t_n}
\right) \right] \right] 
\nonumber \\
& \qquad \qquad + 2 \intd \grad \buhLa^{L,n} :
\pi_{h}
\left[{\mathcal P}_h \left[{\mathcal E} 
\left( \frac{\strhLa^{L,n}-\strhLa^{L,n-1}}{\Delta t_n}
\right) \right] \BL(\strhLa^{L,n})\right]
\nonumber \\
& \qquad \qquad +
\intd \sum_{m=1}^d \sum_{p=1}^d 
[\buhLa^{L,n-1}]_m \,
\Lambda^{L}_{m,p}(\strhLa^{L,n}) : \frac{\partial}{\partial \xx_p}
\left[{\mathcal P}_h 
\left[{\mathcal E} \left( \frac{\strhLa^{L,n}-\strhLa^{L,n-1}}{\Delta t_n}
\right) \right] \right]
\nonumber \\
& \qquad 
\leq
C \intd \pi_h [\,\|\strhLa^{L,n}\|^2\,]
\nonumber \\
& \qquad \quad + C(L) \left[ 1
+\alpha \|\grad \strhLa^{L,n}\|_{L^2(\D)}^2
+ \|\grad \buhLa^{L,n}\|_{L^2(\D)}^2
+ \|\buhLa^{L,n-1}\|_{L^2(\D)}^2
\right]
\,. 
\label{psitG1}
\end{align}
Multiplying (\ref{psitG1}) by $\Delta t_n$, summing from $n=1,...,N_T$
and noting (\ref{Fstab2newaL}) and (\ref{Fstab3aL})
yields the second bound in (\ref{Fstab4aL}).
\end{proof}

\subsection{Convergence of the discrete solutions}
First we note the following result.
\begin{lemma}\label{RQlem}
For $k=1,\ldots, N_K$, it follows that
\begin{align}
\int_{K_k} \|\bchi^{-1}\| &\leq C\,\int_{K_k} \pi_h[ \,\|\bchi^{-1}\|\,]
\qquad \forall \bchi \in \ShonePD\,. 
\label{RQlemeq}
\end{align}
\end{lemma}
\begin{proof}
We recall the well-known result about equivalence of norms
\begin{align}
\frac{1}{d^{\frac{1}{2}}}\|\bphi\|
\leq \|\bphi\|_2 
:= \sup_{\bv \in \R^d,\ \|\bv\|=1} \|\bphi \,\bv\|
\leq \|\bphi\| \qquad \forall \bphi \in \R^{d \times d}\,.
\label{2normeq}
\end{align}
We recall also that if $\bphi \in \RSPD$, then 
\begin{align}
\bz^T \bphi\,\bz \geq \|\bphi^{-1}\|_2^{-1} \|\bz \|^2 \qquad \forall \bz \in \R^d;
\label{eigenmin}
\end{align}
that is, $\|\bphi^{-1}\|_2^{-1}$ is the smallest eigenvalue of $\bphi$. 
For $\bchi \in \ShonePD$, on adopting the notation in the proof of Lemma \ref{lem:inf-bound},
we have that 
\begin{align}
\bchi(\xx) &= \sum_{j=0}^d \bchi(P_j^k)\,\eta^k_j(\xx) \qquad \forall 
\xx \in K_k,  \qquad k=1, \ldots,N_K.
\label{bchiK}
\end{align}
Then for $\bv \in \R^d$, with $\|\bv\|=1$, it follows from (\ref{bchiK})
and (\ref{eigenmin}) that
\begin{align} 
\|\bchi^{-1}(\xx) \,\bv \| &\geq \bv^T \,\bchi^{-1}(\xx) \,\bv 
= (\bchi^{-1}(\xx) \,\bv)^T \,\bchi(\xx)\, (\bchi^{-1}(\xx)\, \bv) 
\nonumber \\
& \geq \left[ \sum_{j=0} ^d \|\bchi^{-1}(P_j^k)\|_2^{-1} \,\eta^k_j(\xx) \right] 
\|\bchi^{-1}(\xx)\, \bv\|^2 
\qquad \forall 
\xx \in K_k;  
\label{bchiKbelow}
\end{align}
where we have noted that $\eta^k_j(\xx) \geq 0$, for all $\xx \in K_k$,
and $\bchi(P_j^k) \in \RSPD$, $j=0,\ldots, d$.
The bound (\ref{bchiKbelow}), on noting (\ref{2normeq}), yields that
\begin{align}
&\|\bchi^{-1}(\xx)\|_2 \leq \left[[\pi_h [ \,\|\bchi^{-1}\|^{-1}_2\,]\,](\xx)\right]^{-1}
\qquad \forall \xx \in K_k, \qquad k=1,\ldots,N_K, 
\nonumber \\
&\hspace{3in}\qquad \forall \bchi \in \ShonePD\,.  
\label{RQ}
\end{align}
Hence it follows from (\ref{2normeq}), (\ref{RQ}) and (\ref{inverse}) 
with $\pi_h [ \,\|\bchi^{-1}\|\,]$
that
\begin{align}
\frac{1}{d^{\frac{1}{2}}} \int_{K_k} \|\bchi^{-1}\| &\leq
\int_{K_k}
\left[\pi_h [ \,\|\bchi^{-1}\|^{-1}_2\,]\right]^{-1} 
\leq |K_k|\, \|\pi_h [ \,\|\bchi^{-1}\|\,]\|_{L^\infty(K_k)}
\nonumber 
\\
&\leq C\,\int_{K_k} \pi_h[ \,\|\bchi^{-1}\|\,]
\qquad k=1,\ldots,N_K, \qquad \forall \bchi \in \ShonePD\,,  
\label{RQmore} 
\end{align}
and hence the desired result (\ref{RQlemeq}).
\end{proof}

We note from (\ref{ipmat}), (\ref{eq:tensor-g}) and (\ref{eq:Bd}) that
\begin{align}
\|\bphi^{-1}\| \leq \|\,[\beta^{L}(\bphi)]^{-1}\| \qquad \forall \bphi \in \RSPD\,.
\label{betainv}
\end{align} 
Therefore (\ref{Fstab2newaL}), (\ref{Fstab3aL},b), (\ref{idatah}), 
(\ref{interpinf}), (\ref{betainv}), (\ref{RQmore})
and (\ref{timeconaL}--c)
yield that
\begin{subequations}
\begin{align}
\nonumber
&\sup_{t \in (0,T)} \|\buhLa^{L,\dt(,\pm)}\|^2_{L^2(\D)}
+ 
\int_{0}^T 
\|\grad \buhLa^{L,\dt (,\pm)}\|^2_{L^2(\D)}
\,dt
\\ &\hspace{0.4in}
+\int_0^T  
\left[ \|\,[\strhLa^{L,\dt,+}]^{-1}\|_{L^{1}(\D)}
+
\frac{\|\buhLa^{L,\dt,+}
-\buhLa^{L,\dt,-}\|^2_{L^2(\D)}}{\Delta(t)}
\right]
 dt \leq C\,,
\label{stab1c}
\\
&\sup_{t \in (0,T)} 
\|\strhLa^{L,\dt(,\pm)}\|^2_{L^2(\D)}
\nonumber \\
\label{stab2c}
& \quad +  \int_{0}^T \left[\alpha \|\grad \strhLa^{L,\dt(, \pm)}\|_{L^2(\D)}^2
+ \frac{\|\strhLa^{L,\dt,+}-\strhLa^{L,\dt,-}\|^2_{L^2(\D)}}{\Delta(t)}
\right]
dt \leq C(L)\,,\\
&
\int_0^T \left[ \left\|{\mathcal S}\,\frac{\partial \buhLa^{L,\dt}}{\partial t}
\right\|_{H^1(\D)}^{\frac{4}{\vartheta}}
+ 
\left\|{\mathcal E}\,\frac{\partial 
\strhLa^{L,\dt}}{\partial t} \right\|_{H^1(\D)}^{2}
\right] dt
\leq C(L,T)\,;
\label{stab3c}
\end{align}
\end{subequations}
where $\vartheta$ is as defined in (\ref{varth}).

We are now in a position to prove the following convergence result 
for (P$^{L,\Delta t}_{\alpha,h}$).
\begin{theorem}\label{convaL}
There exists a subsequence of $\{(\buhLa^{L,\dt},\strhLa^{L,\dt})
\}_{h>0,\Delta t>0}$, 
and functions $\buaL \in
L^{\infty}(0,$ $T;[L^2(\D))]^d)\cap L^{2}(0,T;\Uz) \cap
W^{1,\frac{4}{\vartheta}}(0,T;\Uz')$ and
$\straL \in L^{\infty}(0,T;[L^{2}(\D)]^{d \times d}_{SPD})
\cap L^{2}(0,T;[H^{1}(\D)]^{d \times d}_{SPD})$
$\cap H^1(0,T;([H^{1}(\D)]^{d \times d}_S)')$ 
such that, as $h,\,\Delta t \rightarrow 0_+$,
\begin{subequations}
\begin{alignat}{2}
\buhLa^{L,\Delta t (,\pm)} &\rightarrow \buaL \qquad &&\mbox{weak* in }
L^{\infty}(0,T;[L^2(\D)]^d), \label{uwconL2}\\
\buhLa^{L,\Delta t (,\pm)} &\rightarrow \buaL \qquad &&\mbox{weakly in }
L^{2}(0,T;[H^1(\D)]^{d}), \label{uwconH1}\\
{\mathcal S} \frac{\partial \buhLa^{L,\dt}}{\partial t} 
&\rightarrow {\mathcal S} \frac{\partial \buaL}{\partial t}
 \qquad &&\mbox{weakly in }
L^{\frac{4}{\vartheta}}(0,T;\Uz), \label{utwconL2}\\
\buhLa^{L,\Delta t (,\pm)} &\rightarrow \buaL
\qquad &&\mbox{strongly in }
L^{2}(0,T;[L^{r}(\D)]^d), \label{usconL2}
\end{alignat}
\end{subequations}
and
\begin{subequations}
\begin{alignat}{2}
\strhLa^{L,\Delta t (,\pm)} &\rightarrow
\straL
\quad &&\mbox{weak* in }
L^{\infty}(0,T;[L^2(\D)]^{d \times d}), \label{psiwconL2}\\
\strhLa^{L,\Delta t(,\pm)}
&\rightarrow  \straL
\quad &&\mbox{weakly in }
L^{2}(0,T;[H^1(\D)]^{d \times d}), \label{psiwconH1x}\\
{\mathcal E} \frac{\partial \strhLa^{L,\Delta t}}{\partial t}
&\rightarrow {\mathcal E} \frac{\partial \straL}{\partial t}
 \qquad &&\mbox{weakly in }
L^{2}(0,T;[H^1(\D)]^{d \times d}), \label{psitwconL2}\\
\strhLa^{L,\Delta t (,\pm)} &\rightarrow
\straL
\qquad &&\mbox{strongly in }
L^{2}(0,T;[L^{r}(\D)]^{d \times d}), \label{psisconL2}
\\
\pi_h[\BL(
\strhLa^{L,\Delta t (,\pm)})] &\rightarrow
\beta^L(\straL)
\qquad &&\mbox{strongly in }
L^{2}(0,T;[L^{2}(\D)]^{d \times d}),
\label{piBdLconv}
\\
\Lambda^L_{m,p}(\strhLa^{L,\Delta t (,\pm)}) &\rightarrow
\beta^L(\straL)\,\delta_{mp}
\qquad &&\mbox{strongly in }
L^{2}(0,T;[L^{2}(\D)]^{d \times d}),
\nonumber\\
&&& \hspace{1in}
\quad m,p =1,\dots,d,
 \label{XXxLinf}
\end{alignat}
\end{subequations}
where $\vartheta$ is defined by (\ref{varth})
and $r \in [1,\infty)$ if $d=2$ and $r \in [1,6)$ if $d=3$.

Furthermore, $(\buaL,\straL)$ solve the following problem:

{({\bf P}$^{L}_\alpha$)} Find $\buaL \in L^{\infty}(0,T;[L^{2}(\D)]^d)
\cap L^{2}(0,T;\Uz) \cap W^{1,\frac{4}{\vartheta}}(0,T;\Uz')$
and $\straL
\in L^{\infty}(0,T[L^{2}(\D)]^{d \times d}_{SPD})\cap
L^{2}(0,T;[H^{1}(\D)]^{d \times d}_{SPD})\cap H^{1}(0,T;([H^{1}(\D)]^{d \times d}_S)')$
such that 
\begin{subequations}
\begin{align}
\nonumber
& \displaystyle\int_{0}^{T} {\rm Re} \left\langle \frac{\partial \buaL}{\partial t},
\bv \right\rangle_{\Uz}
dt  
\nonumber \\
& \qquad + \int_{0}^T \intd \left[ 
(1-\e) \,\grad \buaL: \grad \bv +
{\rm Re} \left[ (\buaL \cdot \grad) \buaL
\right]\,\cdot\,\bv \right]  dt
\nonumber 
\\
&\hspace{0.5in} =
\int_0^T \langle \f , \bv \rangle_{H^1_0(\D)} \,dt
- \frac{\e}{\rm Wi} \int_{0}^{T} \intd \beta^L(\straL) 
: \grad \bv \, dt  
\nonumber \\
& \hspace{2in}
\qquad \forall \bv \in L^{\frac{4}{4-\vartheta}}(0,T;\Uz),
\label{weak1d}
\\
&\int_{0}^T  
\nonumber
\left \langle \frac{\partial \straL}{\partial t}
,\bphi
\right \rangle_{H^1(\D)} dt
\\
& \qquad + \int_{0}^T  \int_{\D} \left[
(\buaL \cdot \grad) [\beta^L(\straL)] : \bphi +
\alpha\,\grad \straL :: \grad \bphi \right] \,  dt
\nonumber 
\\
&\hspace{0.5in} = \int_{0}^T \int_{\D}
\left[2\,(\grad \buaL)\,\beta^L(\straL) - \frac{1}{\rm Wi}
(\straL-\I) \right] : \bphi \,  dt
\nonumber \\
& \hspace{2in}\qquad \forall
\bphi \in L^{2}(0,T;[H^1(\D)]^{d \times d}_S);
\label{weak2d}
\\
\mbox{and}\qquad
& \lim_{t \rightarrow 0_+} \intd (\buaL(t,\xx)-\bu^0(\xx))\,\cdot\,\bv=0
\nonumber
\\
&\hspace{1.5in} \forall \bv \in {\rm H} 
:= \{ \bw \in [L^2(\D)]^d : {\rm div} \,\bw =0 \mbox{ in } \D\}\,, 
\nonumber  \\
& \lim_{t \rightarrow 0_+} \intd (\straL(t,\xx)-\strs^0(\xx)):\bchi=0
\qquad \forall \bchi \in [L^2(\D)]^{d\times d}_{SPD}\,.
\label{weakic}
\end{align}
\end{subequations}
\end{theorem}
\begin{proof}
The results  (\ref{uwconL2}--c) follow immediately from the bounds
(\ref{stab1c},c) on noting the notation (\ref{timeconaL}--c).
The denseness of $\bigcup_{h>0}{\rm Q}_h^1$ in $L^2(\D)$
and (\ref{Vh1}) yield that $\buaL \in L^2(0,T;\Uz)$.
The strong convergence result
(\ref{usconL2}) for $\buhLa^{L,\Delta t}$
follows immediately from (\ref{uwconL2}--c) and (\ref{compact1})
with 
$\mu_0=2$, $\mu_1= 4/\vartheta$,
${\cal Y}_0 = [H^1({\cal D})]^d$, 
${\cal Y}_1 = {\mathrm V}'$ with norm $\|{\cal S} \cdot\|_{H^1(\D)}$ and
${\cal Y} = [L^r({\cal D})]^d$ 
for the stated values of $\vartheta$ and $r$.
Here we note that ${\cal Y}$ is a Banach space and ${\cal Y}_i$, $i=0,\,1$, 
are reflexive Banach spaces
with $[L^r({\cal D})]^d$ continuously embedded in ${\mathrm V}'$,
as $[H^{-1}({\cal D})]^d$ is continuously embedded in ${\mathrm V}'$,
and $[H^1({\cal D})]^d$ compactly embedded in $[L^r({\cal D})]^d$ for
the stated values of $r$. 
We now prove (\ref{usconL2}) for
$\buhLa^{L,\Delta t,\pm}$. First we obtain from the bound on the last
term on the left-hand side of (\ref{stab1c}) and
(\ref{eqtime+}) that
\begin{eqnarray}
\|\buhLa^{L,\Delta t}-\buhLa^{L,\Delta t,\pm}\|_{L^{2}(0,T,L^{2}(\D))}^2
\leq C\,\Delta t.
\label{upmr1}
\end{eqnarray}
Second, we note from Sobolev embedding that, for all $\eta \in
L^2(0,T;H^1(\D))$,
\begin{align}
\|\eta\|_{L^2(0,T;L^{r}(\D))} &\leq
\|\eta\|_{L^2(0,T;L^{2}(\D))}^{\theta} \,\|\eta\|_{L^2(0,T;L^{s}(\D))}^{1-\theta}
\nonumber \\
&\leq C\,\|\eta\|_{L^2(0,T;L^{2}(\D))}^{\theta} \,\|\eta\|_{L^2(0,T;H^{1}(\D))}^{1-\theta}
\label{upmr2}
\end{align}
for all $r \in [2,s)$, with any $s \in (2, \infty)$ if $d=2$ or
any $s \in (2,6]$ if $d=3$, and $\theta = [2\,(s-r)]/[r\,(s-2)]
\in (0,1].$ Hence, combining (\ref{upmr1}), (\ref{upmr2}) and
(\ref{usconL2}) for $\buhLa^{L,\Delta t}$ yields (\ref{usconL2}) for
$\buhLa^{L,\Delta t,\pm}$.

Similarly, the results (\ref{psiwconL2}--c)
follow immediately from (\ref{stab2c},c).
The strong convergence result (\ref{psisconL2}) for $\strhLa^{L,\Delta t}$
follows immediately from (\ref{psiwconL2}--c), (\ref{Echinorm}) 
and (\ref{compact1})
with $\mu_0=\mu_1=2$,
${\cal Y}_0 = [H^1({\cal D})]^{d \times d}$, 
${\cal Y}_1 = [H^{-1}({\cal D})]^{d \times d}$ and
${\cal Y} = [L^r({\cal D})]^{d\times d}$ 
for the stated values $r$. 
Similarly to (\ref{upmr1}), the second bound in (\ref{stab2c})
then yields that (\ref{psisconL2}) holds for  $\strhLa^{L,\Delta t(,\pm)}$.

Since $\strhLa^{L,\Delta t(,\pm)} \in L^2(0,T;\ShonePD)$, 
it follows that $\straL$ is symmetric non-negative
definite a.e.\ in $\D_T$. We now establish that $\straL$ is symmetric positive
definite a.e.\ in $\D_T$. Assume that $\straL$ is not symmetric positive
definite a.e.\ in $\D_T^0 \subset \D_T$. Let $\bv \in L^2(0,T;[L^2(\D)]^d)$ be such that
$\straL \,\bv = \bzero$ 
with $\|\bv\|=1$ a.e.\ in $\D_T^0$ and $\bv=\bzero$ a.e.\ in $\D 
\setminus 
\D_T^0$. We then have from (\ref{stab1c}) that  
\begin{align}
\nonumber
|\D^0_T| &= \int_0^T \intd \|\bv\|^2\,dt = \int_0^T \intd 
\left([\strhLa^{L,\Delta t,+}]^{-\frac{1}{2}} \bv \right) : 
\left([\strhLa^{L,\Delta t,+}]^{\frac{1}{2}} \bv \right) dt
\\
&\leq C\, \left( \int_0^T \intd  
\strhLa^{L,\Delta t,+} :: (\bv \bv^T)\, dt \right)^{\frac{1}{2}}\,.
\label{weakSPD}
\end{align}       
Hence it follows from (\ref{weakSPD}) and (\ref{psisconL2}) that $|\D^0_T| =0$.

Finally, the desired results (\ref{piBdLconv},f) follow immediately from (\ref{MXittxbd})
the second bound in (\ref{stab2c}), (\ref{Lip}), 
(\ref{psisconL2}) and the fact that $\straL \in L^\infty(0,T;[L^2(\D)]^{d \times d}_{SPD})$.

It remains to prove that $(\buaL,\straL)$ solves (P$^L_\alpha$).
It follows from (\ref{Vhconv}),
(\ref{stab1c}--c),
(\ref{uwconL2}--d), 
(\ref{piBdLconv}), 
(\ref{fnconv}), (\ref{Sw})
and (\ref{conv0c})
that we may pass to the limit, $h,\,\Delta t \rightarrow 0_+$, in
the $L$-dependent version of (\ref{equncon}) to obtain that $(\buaL, 
\straL)$ 
satisfy (\ref{weak1d}).
It also follows from (\ref{proju0}), (\ref{Vhconv}), 
(\ref{utwconL2},d) and as ${\rm V}$ is dense in ${\rm H}$ 
that $\bua^L(0,\cdot)=\bu^0(\cdot)$ in the
required sense; see (\ref{weakic}) and Lemma 1.4 on p179 in Temam.\cite{Temam}

It follows from (\ref{psiwconL2}--f), (\ref{uwconH1},d), (\ref{Echi}),
(\ref{stab1c}--c), (\ref{interp1},b), 
(\ref{eq:symmetric-tr}) and as $\buaL \in L^2(0,T;{\rm V})$
that we may pass to the limit $h,\, \Delta t \rightarrow 0_+$
in 
the $L$-dependent version of
(\ref{eqpsincon}) with $\bchi \ = \pi_h\, \bphi$
to obtain (\ref{weak2d}) for any
$\bphi \in C^\infty_0(0,T;[C^\infty(\overline{\D})]^{d \times d}_S)$.
For example, in order to pass to the limit on the first term in 
the $L$-dependent version of
(\ref{eqpsincon}),
we note that
\begin{align}
\nonumber
&\int_0^T \int_{\D} \pi_h \left[\left(\frac{\partial \strhLa^{L,\Delta t}}{\partial t}
+ \frac{1}{\Wi} \, \strhLa^{L,\dt,+}\right) : \pi_h\,\bphi \right]   dt \\
&\quad
=
\int_0^T \int_{\D} 
\left\{ 
\left(\frac{\partial \strhLa^{L,\Delta t}}{\partial t}
+ \frac{1}{\Wi} \, \strhLa^{L,\dt,+}\right) : \pi_h \,\bphi
+(I-\pi_h) \left[
\strhLa^{L,\Delta t} : 
\pi_h \left[\frac{\partial \bphi}{\partial t}\right]
\right] \right\}  dt
\nonumber \\
& \hspace{1.8in} - \frac{1}{\Wi} \int_0^T \int_{\D} 
(I-\pi_h) \left[
\strhLa^{L,\Delta t,+} : 
 \pi_h \,\bphi
\right] dt.
\label{intparts} 
\end{align}
The desired result (\ref{weak2d}) then follows from noting that
$ C^\infty_0(0,T;[C^\infty(\overline{\D})]^{d \times d}_S)$
is dense in
$L^2(0,T;$ $[H^1(\D)]^{d \times d}_S)$.
Finally, it follows from (\ref{projp0}),  
(\ref{psitwconL2},d), (\ref{interp1},b)
and (\ref{interpinf})
 that 
$\straL(0,\cdot)=\strs^0(\cdot)$ 
in the
required sense; see (\ref{weakic}) and 
Lemma 1.4 on p179 in Temam.\cite{Temam} 
\end{proof}

\begin{remark}\label{Lalphaindrem}
It follows from (\ref{stab1c},b), (\ref{uwconL2},b) and (\ref{psiwconL2},b) that
\begin{subequations}
\begin{align}
\sup_{t \in (0,T)} \|\buaL \|^2_{L^2(\D)}
+
\int_{0}^T 
\|\grad \buaL \|^2_{L^2(\D)}
dt
&\leq C\,,
\label{finalbd}\\
\sup_{t \in (0,T)} \|\straL \|^2_{L^2(\D)}
+
\alpha \int_{0}^T 
\|\grad \straL \|^2_{L^2(\D)}
dt
&\leq C(L)\,.
\label{finalbdst}
\end{align}
\end{subequations}
Hence, although we have introduced a cut-off $L\gg 1$ 
to certain 
terms, and added diffusion 
with a positive coefficient $\alpha$
in the stress equation compared to the standard Oldroyd-B model;
the bound (\ref{finalbd}) on the velocity $\buaL$ 
is independent of the
parameters $L$ and $\alpha$, where $(\buaL,\straL)$ solves (P$^L_\alpha$),
(\ref{weak1d}--c).
\end{remark}

\section{Convergence of (P$^{\Delta t}_{\alpha,h}$) to
(P$_{\alpha}$) in the case $d=2$} 
\label{sec7}

First, we recall the discrete Gronwall inequality:
\begin{alignat}{2}
(r^0)^2 +(s^0)^2 &\leq (q^0)^2\,, 
\nonumber\\
(r^m)^2 +(s^m)^2 &\leq \sum_{n=0}^{m-1} (\eta^n)^2 (r^n)^2 + \sum_{n=0}^m
(q^n)^2 \qquad &&m\geq 1
\nonumber
\\
\Rightarrow \qquad (r^m)^2 +(s^m)^2 &\leq \exp( \sum_{n=0}^{m-1} (\eta^n)^2)
\sum_{n=0}^m (q^n)^2 \qquad &&m\geq1\,.
\label{DG} 
\end{alignat}

\begin{theorem}
\label{dstabthma}
Under the assumptions of Theorem \ref{dLconthm},
there exists a solution 
$\{(\buhLa^{n},\strhLa^{n})\}_{n=1}^{N_T} \in [\Vhone \times \ShonePD]^{N_T}$ 
of (P$^{\dt}_{\alpha,h}$)
such that the bounds (\ref{eq:estimate-PLh}) and (\ref{Fstab2newaL}) hold.

If $d=2$, $\alpha \leq \frac{1}{2\Wi}$ and $\Delta t \leq C_\star(\zeta^{-1})\,\alpha^{1+\zeta}\,h^2$,
for a $\zeta >0$,
then the following bounds hold:
\begin{align}
&\max_{n=0, \ldots, N_T} \intd
\pi_h[\,\|\strhLa^{n}\|^2\,]
+  \sum_{n=1}^{N_T} \intd \left[\dt_n \alpha \|\grad \strhLa^{n}\|^2 
+ \pi_h[\,\|\strhLa^{n}-\strhLa^{n-1}\|^2\,]
\right]
\nonumber \\
& \hspace{0.5in}+
\sum_{n=1}^{N_T} \dt_n \left\|{\mathcal S}
\left( \frac{\buhLa^{n}-\buhLa^{n-1}}{\dt_n}\right)
\right\|^{\frac{4}{\vartheta}}_{H^1(\D)}
\leq C(\alpha^{-1},T)\,;
\label{Fstab3a}
\end{align}
where $\vartheta \in (2,4)$.
\end{theorem}
\begin{proof}
Existence  
and the bounds
(\ref{eq:estimate-PLh}) and (\ref{Fstab2newaL})
were proved in Theorem
\ref{dLconthm}.

On choosing $\bphi \equiv \strhLa^{n}$ 
in the $L$-independent version of (\ref{eq:PLahb}), 
it follows
from (\ref{elemident})  
and on applying a Youngs' inequality for any $\zeta >0$ that 
\begin{align}
\nonumber
&\frac{1}{2}\intd \pi_h[ \,\|\strhLa^{n}\|^2+\|\strhLa^{n}-
\strhLa^{n-1}\|^2\,] + \dt_n \alpha \intd \|\grad \strhLa^{n}\|^2
+ \frac{\dt_n}{2 \Wi} \intd \pi_h[ \,\|\strhLa^{n}\|^2\,]
\\
&\qquad \leq \frac{1}{2}\intd \pi_h[ \,\|\strhLa^{n-1}\|^2\,] 
+ \frac{\dt_n d\,|\D|}{2\Wi}
+ 2\dt_n
\intd \grad \buhLa^{n} : 
\pi_h[(\strhLa^{n})^2]
\nonumber \\
& \qquad \qquad + \dt_n
\int_\D \sum_{m=1}^d \sum_{p=1}^d 
[\buhLa^{n-1}]_m \,\Lambda_{m,p}(\strhLa^{n})
: \frac{\partial \strhLa^n}{\partial \xx_p}
\nonumber \\
&\qquad \leq \frac{1}{2} \intd \pi_h[\, \|\strhLa^{n-1}\|^2\,]
+ C\,\dt_n \left[ 1 + \|\grad \buhLa^{n}\|_{L^2(\D)} 
\,
\|\pi_h[(\strhLa^{n})^2]\|_{L^2(\D)}\right]
\nonumber \\
& \qquad \qquad 
+ C\,\dt_n\,\|\buhLa^{n-1}\|_{L^{\frac{2(2+\zeta)}{\zeta}}(\D)}\,
\|\Lambda_{m,p}(\strhLa^{n})\|_{L^{2+\zeta}(\D)}\,
\|\grad \strhLa^{n}\|_{L^2(\D)}\,.
\label{sigstaba} 
\end{align}

It follows from (\ref{interpinf}), (\ref{normprod}), (\ref{inverse}) and (\ref{GN}),
as $d=2$, that
\begin{align}
\nonumber
\|\pi_h[(\strhLa^{n})^2]\|_{L^2(\D)}^2
&=
\intd  \|\,\pi_h[(\strhLa^{n})^2]\,\|^2 \leq 
\intd  \pi_h[\, \|(\strhLa^{n})^2\|^2\,] \leq 
\intd  \pi_h[\, \|\strhLa^{n}\|^4\,] \\
& = \sum_{k=1}^{N_K} \int_{K_k}  \pi_h[\, \|\strhLa^{n}\|^4\,] 
\leq \sum_{k=1}^{N_K} |K_k| \,\|\strhLa^{n}\|_{L^\infty(K_k)}^4
\nonumber \\
&\leq C\,\sum_{k=1}^{N_K} |K_k| \,\left(|K_k|^{-1}\,\|\strhLa^{n}\|_{L^1(K_k)}\right)^4
\leq C\,\sum_{k=1}^{N_K} \|\strhLa^{n}\|_{L^4(K_k)}^4
\nonumber \\
&= C\,\|\strhLa^{n}\|_{L^4(\D)}^4
\leq C\,\|\strhLa^{n}\|_{L^2(\D)}^2\,\|\strhLa^{n}\|_{H^1(\D)}^2\,.
\label{sigmaL4} 
\end{align}
Similarly, it follows from the $\delta$-independent versions of (\ref{Lambdampdef}), (\ref{hLambdajdef},b),
recall Remark \ref{SPD}, (\ref{Breg}), (\ref{inverse}) and (\ref{GN}) that
for all $\zeta >0$ 
\begin{align}
\nonumber 
\|\Lambda_{m,p}(\strhLa^{n})\|_{L^{2+\zeta}(\D)}^{2+\zeta}
&\leq \sum_{k=1}^{N_K} |K_k|\, \|\Lambda_{m,p}(\strhLa^{n})\|_{L^\infty(K_k)}^{2+\zeta}
\leq C\,\sum_{k=1}^{N_K} |K_k|\, \|\strhLa^{n}\|_{L^\infty(K_k)}^{2+\zeta}
\\
&= C\,\|\strhLa^{n}\|_{L^{2+\zeta}(\D)}^{2+\zeta}
\leq C(\zeta)\,\|\strhLa^{n}\|_{L^2(\D)}^2\,\|\strhLa^{n}\|_{H^1(\D)}^\zeta\,.
\label{LambdaL2}
\end{align}
In addition, we note from (\ref{GN}), (\ref{eq:poincare}) and (\ref{Fstab2newaL}) that
for all $\zeta>0$
\begin{align}
\|\buhLa^{n-1}\|_{L^{\frac{2(2+\zeta)}{\zeta}}(\D)}
\leq C(\zeta^{-1})\, \|\buhLa^{n-1}\|_{L^{2}(\D)}^{\frac{\zeta}{2+\zeta}}
\,\|\buhLa^{n-1}\|_{H^{1}(\D)}^{\frac{2}{2+\zeta}}
\leq C(\zeta^{-1})\,
\|\grad \buhLa^{n-1}\|_{L^{2}(\D)}^{\frac{2}{2+\zeta}}\,.
\label{uzeta}
\end{align}

Combining (\ref{sigstaba}), (\ref{sigmaL4}), (\ref{LambdaL2}) and (\ref{uzeta}), and
on noting (\ref{interpinf}) 
and that $\alpha \leq \frac{1}{2\Wi}$, 
yields on applying a Young's inequality that
for all $\zeta >0$
\begin{align}
\nonumber
&\intd \pi_h[ \,\|\strhLa^{n}\|^2+ \|\strhLa^{n}-
\strhLa^{n-1}\|^2\,] + \dt_n \alpha \intd \|\grad \strhLa^{n}\|^2
+ \frac{\dt_n}{2 \Wi} \intd \pi_h[ \,\|\strhLa^{n}\|^2\,]
\\
& \quad
\leq 
\intd \pi_h[ \,\|\strhLa^{n-1}\|^2\,]
+ C\,\dt_n 
\nonumber \\
& \qquad + C(\zeta^{-1})\,\dt_n \,\alpha^{-(1+\zeta)}
\,\left[ \|\grad \buhLa^{n}\|_{L^2(\D)}^2\,
+ \|\grad \buhLa^{n-1}\|_{L^2(\D)}^2 \right]
\,
\intd \pi_h[\,\|\strhLa^{n}\|^2\,]\,.
\label{sigstabb}
\end{align}
Hence, summing (\ref{sigstabb}) from $n=1,\dots,m$ for $m=1, \dots,N_T$
yields, 
for any $\zeta >0$ that
\begin{align}
\nonumber
&\intd \pi_h[ \,\|\strhLa^{m}\|^2\,]
+ \sum_{n=1}^m 
\dt_n 
\intd \left[
\alpha \|\grad \strhLa^{n}\|^2
+ \frac{1}{2 \Wi} \pi_h[ \,\|\strhLa^{n}\|^2\,]
\right]
\\
& \quad
+ \sum_{n=1}^m 
\intd
\pi_h[\,\|\strhLa^{n}-\strhLa^{n-1}\|^2\,] 
\nonumber
\\
& \quad \quad
\leq 
\intd \pi_h[\, \|\strhLa^{0}\|^2\,] + C
\nonumber \\
& \quad \quad \quad  + C(\zeta^{-1})\,\alpha^{-(1+\zeta)}
\,\sum_{n=1}^m
\dt_n \,
\left[ \sum_{k=n-1}^n \|\grad \buhLa^{k}\|_{L^2(\D)}^2
\right]
\,
\intd \pi_h[\,\|\strhLa^{n}\|^2\,]\,.
\label{sigstabc}  
\end{align}
Applying the discrete Gronwall inequality (\ref{DG}) to (\ref{sigstabc}), and noting
(\ref{Deltatqu}), (\ref{eqnorm}), (\ref{idatah}), (\ref{Fstab2newaL}), 
(\ref{Vinverse}) and that $\dt \leq C_\star(\zeta^{-1})
\,\alpha^{1+\zeta}\,h^2$, 
for a $\zeta >0$ where $C_\star(\zeta^{-1})$ is sufficiently small, 
yields the first three bounds in (\ref{Fstab3a}).

Similarly to (\ref{equndtb}), 
on choosing $\bw = {\mathcal R}_h \left[ {\mathcal S}
\left(\frac{\buhLa^{n}-\buhLa^{n-1}}{\dt_n}\right)\right] \in \Vhone$ in
the $L$-independent version of (\ref{eq:PLaha}) yields, 
on noting (\ref{Rh}), (\ref{Swnorm}), (\ref{Rhstab})
 and Sobolev embedding, that
\begin{align}
\nonumber
&\Re \left\|{\mathcal S} \left(\frac{\buhLa^{n}-\buhLa^{n-1}}{\Delta t_n}\right)
\right\|^2_{H^1(\D)}
=\Re\int_{\D}
\frac{\buhLa^{n}-\buhLa^{n-1}}{\Delta t_n} \cdot
{\mathcal R}_h \left[ {\mathcal S}
\left(\frac{\buhLa^{n}-\buhLa^{n-1}}{\Delta t_n}\right)\right]
\\
& \qquad \leq C \bigl[\|\strhLa^{n}\|^2_{L^2(\D)}
+\|\grad \buhLa^{n}\|^2_{L^2(\D)}
+\|\,\|\buhLa^{n-1}\|\,\|\buhLa^{n}\|\,\|^2_{L^2(\D)}
\nonumber\\
& \qquad \qquad
+ \|\,\|\buhLa^{n-1}\| \,\|\grad \buhLa^{n}\|\,\|_{L^{1+\theta}(\D)}^2
+ \|\f^n\|_{H^{-1}(\D)}^2 \bigr]
\label{equndtba}
\end{align}
for any $\theta>0$ as $d=2$.  
On taking the $\frac{2}{\vartheta}$
power of both sides of
(\ref{equndtba}),
multiplying by $\Delta t_n$,
summing from $n=1,\dots,N_T$
and noting the $L$-independent versions of (\ref{L4sob}) and (\ref{sob1})
with $\theta = (\vartheta - 2)/(6 - \vartheta)$,
(\ref{Deltatqu}), (\ref{fncont}),
(\ref{Fstab2newaL}), (\ref{idatah}), (\ref{interpinf}) and 
the first bound in (\ref{Fstab3a})
yields the last bound in (\ref{Fstab3a}).
\end{proof}

It follows from (\ref{Fstab2newaL}), (\ref{Fstab3a}), (\ref{idatah}), 
(\ref{interpinf}), (\ref{RQmore}) and
(\ref{timeconaL}--c)
 that
\begin{subequations}
\begin{align}
\nonumber
&\sup_{t \in (0,T)} \|\buhLa^{\dt(,\pm)}\|^2_{L^2(\D)}
+ 
\int_{0}^T 
\|\grad \buhLa^{\dt (,\pm)}\|^2_{L^2(\D)}
\,dt
\\ &\hspace{0.8in}
+\int_0^T \left[
\|[\strhLa^{\dt,+}]^{-1}\|_{L^1(\D)} 
+
\frac{\|\buhLa^{\dt,+}
-\buhLa^{\dt,-}\|^2_{L^2(\D)}}{\Delta(t)} \right]\, dt \leq C
\label{stab1ca}
\end{align}
and
\begin{align}
\nonumber
&\sup_{t \in (0,T)} \| \strhLa^{\dt(,\pm)}\|^2_{L^2(\D)}
+  \int_{0}^T 
\left[
\alpha
\|\grad \strhLa^{\dt(, \pm)}\|_{L^2(\D)}^2
+ 
\frac{\|\strhLa^{\dt,+}-\strhLa^{\dt,-}\|^2_{L^2(\D)}}{\Delta(t)}
\right] dt\\
&\hspace{0.8in}
+ \int_0^T \left\|{\mathcal S}\,\frac{\partial \buhLa^{\dt}}{\partial t}
\right\|_{H^1(\D)}^{\frac{4}{\vartheta}}
dt
\leq C(\alpha^{-1},T),
\label{stab2ca}
\end{align}
\end{subequations}
where $\vartheta \in (2,4)$.

We note that we have no control on
the time derivative of 
$\strhLa^{\Delta t}$ in (\ref{stab2ca}). This is because 
if we choose
$\bphi
= {\mathcal P}_h\left[{\mathcal E} \left( \frac{\strhLa^{n}-\strhLa^{n-1}}
{\Delta t_n}
\right) \right] \in \Shone$ in 
the $L$-independent 
version of (\ref{eq:PLahb}), the terms involving $\buhLa^m$, $m=n-1$ and $m=n$,
cannot now be controlled in the absence of the cut-off on $\strhLa^{n}$. 
We are now in a position to prove the following convergence result 
for (P$^{\Delta t}_{\alpha,h}$).
The key difference between the following theorem and Theorem \ref{convaL}
for (P$_{\alpha,h}^{L,\Delta t}$) is that no control on the time derivative
of $\strhLa^{\Delta t}$ in (\ref{stab2ca}) implies no strong convergence
for $\strhLa^{\Delta t(,\pm)}$.

\begin{theorem}\label{conva} 
Let all the assumptions of Theorem \ref{dstabthma} hold.
Then
there exists a subsequence of $\{(\buhLa^{\dt},\strhLa^{\dt})
\}_{h>0,\Delta t>0}$, 
and functions $\bua \in
L^{\infty}(0,$ $T;[L^2(\D))]^2)\cap L^{2}(0,T;\Uz) \cap
W^{1,\frac{4}{\vartheta}}(0,T;\Uz')$ and
$\stra \in L^{\infty}(0,T;[L^{2}(\D)]^{2 \times 2}_{SPD})
\cap L^{2}(0,T;[H^{1}(\D)]^{2 \times 2}_{SPD})$ 
such that, as $h,\,\Delta t \rightarrow 0_+$,
\begin{subequations}
\begin{alignat}{2}
\buhLa^{\Delta t (,\pm)} &\rightarrow \bua \qquad &&\mbox{weak* in }
L^{\infty}(0,T;[L^2(\D)]^2), \label{uwconL2a}\\
\buhLa^{\Delta t (,\pm)} &\rightarrow \bua \qquad &&\mbox{weakly in }
L^{2}(0,T;[H^1(\D)]^{2}), \label{uwconH1a}\\
{\mathcal S} \frac{\partial \buhLa^{\dt}}{\partial t} 
&\rightarrow {\mathcal S} \frac{\partial \bua}{\partial t}
 \qquad &&\mbox{weakly in }
L^{\frac{4}{\vartheta}}(0,T;\Uz), \label{utwconL2a}\\
\buhLa^{\Delta t (,\pm)} &\rightarrow \bua
\qquad &&\mbox{strongly in }
L^{2}(0,T;[L^{r}(\D)]^2), \label{usconL2a}
\end{alignat}
\end{subequations}
and
\begin{subequations}
\begin{alignat}{2}
\strhLa^{\Delta t (,\pm)} &\rightarrow
\stra
\quad &&\mbox{weak* in }
L^{\infty}(0,T;[L^2(\D)]^{2 \times 2}), \label{psiwconL2a}\\
\strhLa^{\Delta t(,\pm)}
&\rightarrow  \stra
\quad &&\mbox{weakly in }
L^{2}(0,T;[H^1(\D)]^{2 \times 2}), \label{psiwconH1xa}\\
\Lambda_{m,p}(\strhLa^{\Delta t (,\pm)}) &\rightarrow
\stra\,\delta_{mp}
\qquad &&\mbox{weakly in }
L^{2}(0,T;[L^{2}(\D)]^{2 \times 2}),
\nonumber \\
& &&\hspace{1.5in}\quad m,p =1,\,2,
 \label{XXxLinfa}
\end{alignat}
\end{subequations}
where $\vartheta \in (2,4)$ 
and $r \in [1,\infty)$. 

Furthermore, $(\bua,\stra)$ solve the following problem:

({\bf P}$_\alpha$) Find $\bua \in L^{\infty}(0,T;[L^{2}(\D)]^2)
\cap L^{2}(0,T;\Uz) \cap W^{1,\frac{4}{\vartheta}}(0,T;\Uz')$
and $\stra
\in L^{\infty}(0,T;[L^{2}(\D)]^{2 \times 2}_{SPD})\cap
L^{2}(0,T;[H^{1}(\D)]^{2 \times 2}_{SPD})$ 
such that 
\begin{subequations}
\begin{align}
\nonumber
& \displaystyle\int_{0}^{T} {\rm Re} \left\langle \frac{\partial \bua}{\partial t},
\bv \right\rangle_{\Uz}
dt 
\\
&\quad
+
 \int_{0}^T \intd \left[ 
(1-\e) \,\grad \bua: \grad \bv +
{\rm Re} \left[ (\bua \cdot \grad) \bua
\right]\,\cdot\,\bv \right]  dt
\nonumber
\\
&\qquad =
\int_0^T \langle \f , \bv \rangle_{H^1_0(\D)} \,dt
- \frac{\e}{\rm Wi} \int_{0}^{T} \intd \stra 
: \grad \bv \, dt  \qquad \forall \bv \in L^{\frac{4}{4-\vartheta}}(0,T;\Uz),
\label{weak1da}
\\
& -\int_{0}^T  
\intd
\stra
: \frac{\partial \bphi}{\partial t}
\, dt
- \intd \strs^0 : \bphi 
\nonumber
\\ 
&\quad
+ \int_{0}^T  \int_{\D} \left[
(\bua \cdot \grad) \stra : \bphi +
\alpha\,\grad \stra :: \grad \bphi \right] \,  dt
\nonumber 
\\
&\qquad = \int_{0}^T \int_{\D}
\left[2\,(\grad \bua)\,\stra - \frac{1}{\rm Wi}
(\stra-\I) \right] : \bphi \,  dt
\nonumber
\\
& \hspace{0.4in}
\qquad \forall
\bphi \in L^{2}(0,T;[H^1(\D)]^{2 \times 2}_S)
\cap W^{1,1}_0(-T,T;[L^2(\D)]^{2 \times 2}_S);
\label{weak2da}\\
\mbox{and}\quad
& \lim_{t \rightarrow 0_+} \intd (\bua(t,\xx)-\bu^0(\xx))\,\cdot\,\bv=0
\qquad \forall \bv \in {\rm H}\,.  
\label{weakia} 
\end{align}
\end{subequations}
\end{theorem}
\begin{proof}
The results (\ref{uwconL2a}--d) and (\ref{psiwconL2a},b) 
follow immediately from the bounds (\ref{stab1ca},b),
as in the proof of Theorem \ref{convaL}.
Similarly, the proof of positive definiteness of $\stra$ follows
as in Theorem \ref{convaL}; that is, (\ref{weakSPD}) and the weak convergence
(\ref{psiwconL2a}) is adequate for this. 
The result (\ref{XXxLinfa}) follows from (\ref{psiwconL2a}), (\ref{MXittxbd})
and (\ref{stab2ca}) and the
 fact that $\stra \in L^\infty(0,T;[L^2(\D)]^{2 \times 2}_{SPD})$.

It follows from (\ref{Vhconv}),
(\ref{stab1ca},b),
(\ref{uwconL2a}--d), 
(\ref{psiwconL2a}), 
(\ref{fnconv}), (\ref{Sw})
and (\ref{conv0c})
that we may pass to the limit, $h,\,\Delta t \rightarrow 0_+$, in
the $L$-independent version of (\ref{equncon}) to obtain that $(\bua,\stra)$  
satisfy (\ref{weak1da}).
It also follows from (\ref{proju0}), (\ref{Vhconv}), 
(\ref{utwconL2a},d) and as ${\rm V}$ is dense in ${\rm H}$ 
that $\bua(0,\cdot)=\bu^0(\cdot)$ in the
required sense; see (\ref{weakia}) and Lemma 1.4 on p179 in Temam.\cite{Temam}

It follows from (\ref{psiwconL2a}--c), (\ref{usconL2a}), (\ref{Echi}),
(\ref{stab1ca},b), (\ref{interp1},b), 
(\ref{eq:symmetric-tr}) and as $\bua \in L^2(0,T;{\rm V})$
that we may pass to the limit $h,\, \Delta t \rightarrow 0_+$
in 
the $L$-independent version of
(\ref{eqpsincon}) with $\bchi \ = \pi_h\, \bphi$
to obtain (\ref{weak2da}) for any
$\bphi \in C^\infty_0(-T,T;[C^\infty(\overline{\D})]^{2 \times 2}_S)$.
For example, in order to pass to the limit on the first and third terms in 
the $L$-independent version of
(\ref{eqpsincon}),
we note that
\begin{subequations}
\begin{align}
\nonumber
&\int_0^T \int_{\D} \pi_h \left[\left(\frac{\partial \strhLa^{\Delta t}}{\partial t}
+ \frac{1}{\Wi} \, \strhLa^{\dt,+}\right) : \pi_h\,\bphi \right]   dt 
\\
& \quad =
\int_0^T \int_{\D} 
\left[
\frac{1}{\Wi} \,\strhLa^{\dt,+} : \pi_h \,\bphi
- \strhLa^{\dt} : \pi_h \left[\frac{\partial \bphi}{\partial t} \right] 
\right] dt - \int_D \pi_h \left[
\strhLa^{\dt} : \pi_h \,\bphi \right](0,\cdot)
\nonumber
\\
& \hspace{0.5in} +
\int_0^T \int_{\D} (\pi_h-\I)
\left[
\frac{1}{\Wi} \,\strhLa^{\dt,+} : \pi_h \,\bphi
- \strhLa^{\dt} : \pi_h \left[\frac{\partial \bphi}{\partial t} \right] 
\right] dt\,,
\label{intpartsa}
 \\
&\intd  \grad \buhLa^{\Delta t,+} 
: \pi_h [\strhLa^{\dt,+}\,\pi_h \,\bphi]
\nonumber\\
& \quad = \intd \grad \buhLa^{\Delta t,+} 
: (\pi_h - \I)[\strhLa^{\dt,+}\,\pi_h \,\bphi]
\nonumber
\\
& \hspace{0.5in} 
-\intd  \left \{
\left((\grad \pi_h\,\bphi) \, \buhLa^{\Delta t,+}\right) : \strhLa^{\dt,+}
+\buhLa^{\Delta, +} \,\cdot\, \left((\pi_h\,\bphi)\,{\rm div}\, \strhLa^{\dt,+}\right) 
\right\}\,;
\label{intpartsb} 
\end{align}
\end{subequations}
where $((\grad \pi_h\,\bphi) \, \buhLa^{\Delta t,+})(t,\xx) \in \R^{2 \times 2}$
with $[(\grad \pi_h \,\bphi) \, \buhLa^{\Delta t,+}]_{ij} = \sum_{k=1}^2
\frac{\partial (\pi_h \bphi)_{ik}}{\partial \xx_j} \, [\buhLa^{\Delta t,+}]_k$. 
The desired result (\ref{weak2da}) then follows from noting that
$ C^\infty_0(-T,T;[C^\infty(\overline{\D})]^{2 \times 2}_S)$
is dense in
$W^{1,1}_0(0,T;$ $[H^1(\D)]^{2 \times 2}_S)$. 
\end{proof}

We have the analogue of Remark \ref{Lalphaindrem}. 

\begin{remark}\label{Lalphaindrema}
It follows from (\ref{stab1ca},b), (\ref{uwconL2a},b) and (\ref{psiwconL2a},b) that
\begin{subequations}
\begin{align}
\sup_{t \in (0,T)} \|\bua\|^2_{L^2(\D)}+
\int_{0}^T 
\|\grad \bua \|^2_{L^2(\D)}
dt
&\leq C\,,
\label{finalbda}\\
\sup_{t \in (0,T)} \|\stra\|^2_{L^2(\D)}+
\alpha \int_{0}^T 
\|\grad \stra \|^2_{L^2(\D)}
dt
&\leq C(\alpha^{-1},T)\,.
\label{finalbdas}
\end{align}
\end{subequations}
Hence, although we have introduced  diffusion 
with a positive coefficient $\alpha$
into the stress equation (\ref{weak2da}) compared to the standard Oldroyd-B model;
the bound (\ref{finalbda}) on the velocity $\bua$ 
is independent of the
parameter $\alpha$, where $(\bua,\stra)$ solves (P$_\alpha$),
(\ref{weak1da}--c).
\end{remark}

\section*{Acknowledgement} This work was initiated, whilst the authors were visiting 
the Beijing International Center for Mathematical Research. 
We would like to thank Professor Pingwen Zhang for his kind hospitality.

\end{document}